\documentclass[reqno,12pt]{amsart}

\usepackage[margin=1in]{geometry}
\usepackage{xspace}
\usepackage[utf8]{inputenc}
\usepackage{graphicx}
\usepackage{amssymb}
\usepackage{mathrsfs}
\usepackage{url}
\usepackage{hyperref}
\usepackage{amsmath,amstext,amsxtra,amsgen,amsbsy,amsopn,amscd,amsthm,amsfonts,mathtools}
\usepackage{latexsym}
\usepackage{dsfont}
\usepackage{enumitem}

\usepackage[final]{showlabels}

\newtheorem{proposition}{Proposition}
\newtheorem{theorem}{Theorem}
\newtheorem{lemma}[proposition]{Lemma}
\newtheorem{corollary}[proposition]{Corollary}

\theoremstyle{definition}

\newtheorem{definition}[proposition]{Definition}

\theoremstyle{remark}

\newtheorem{remark}[proposition]{Remark}

\numberwithin{equation}{section}
\numberwithin{proposition}{section}

\newcommand\R{{\ensuremath {\mathbb R} }}
\newcommand\C{{\ensuremath {\mathbb C} }}
\newcommand\N{{\ensuremath {\mathbb N} }}
\newcommand\Z{{\ensuremath {\mathbb Z} }}

\renewcommand\phi{\varphi}

\renewcommand\le{\leqslant}
\renewcommand\ge{\geqslant}
\renewcommand\epsilon{\varepsilon}
\renewcommand\hat{\widehat}
\renewcommand\tilde{\widetilde}
\renewcommand\bar{\overline}

\newcommand{\gH}{\mathfrak{H}}
\newcommand{\gK}{\mathfrak{K}}

\newcommand{\gS}{\mathfrak{S}}

\newcommand{\cD}{\mathcal{D}}
\newcommand{\cH}{\mathcal{H}}
\newcommand\ii{{\ensuremath {\infty}}}

\newcommand{\cB}{\mathcal{B}}

\newcommand{\cU}{\mathcal{U}}

\newcommand{\cL}{\mathcal{L}}
\newcommand{\cF}{\mathcal{F}}

\newcommand{\cI}{\mathcal{I}}

\newcommand\1{{\ensuremath {\mathds 1} }}
\newcommand{\Sph}{\mathbb{S}}

\newcommand{\esp}{\mathbb{E}}
\newcommand{\sd}{\langle \nabla_x \rangle}
\newcommand{\nsd}{\langle \nabla \rangle}

\newcommand\cS{\mathcal{S}}

\DeclareMathOperator{\im}{Im}

\DeclareMathOperator{\re}{Re}

\DeclareMathOperator{\supp}{supp}

\DeclareMathOperator{\tr}{Tr}

\newcommand{\cA}{\mathcal{A}}

\newcommand{\vertiii}[1]{{\left\vert\kern-0.25ex\left\vert\kern-0.25ex\left\vert #1
    \right\vert\kern-0.25ex\right\vert\kern-0.25ex\right\vert}}

\title{Scattering for the positive density Hartree equation}

\author[A. Borie]{Antoine Borie}
\address{Univ Rennes, CNRS\\ IRMAR - UMR
  6625\\ F-35000 Rennes, France}
\email{antoine.borie@univ-rennes.fr}

\author[S. Hadama]{Sonae Hadama}
\address{Research Institute for Mathematical Sciences, Kyoto University, Kita-Shirakawa, Sakyoku, Kyoto, Japan 606-8502}
\email{hadama@kurims.kyoto-u.ac.jp}

\author[J. Sabin]{Julien Sabin}
\address{Univ Rennes, CNRS\\ IRMAR - UMR
  6625\\ F-35000 Rennes, France}
\email{julien.sabin@univ-rennes.fr}

\begin{document}

\maketitle

\begin{abstract}
 We study the asymptotic stability for large times of homogeneous stationary states for the nonlinear Hartree equation for density matrices in $\R^d$ for $d\ge3$. We can reach both the optimal Sobolev and Schatten exponents for the initial data, with a wide class of interaction potentials $w$ (under the sole assumption that $\hat{w}$ is bounded, including in particular delta potentials). Our method relies on fractional Leibniz rules for density matrices to deal with the fractional critical Sobolev regularity $s=d/2-1$ for odd $d$, as well as Christ-Kiselev lemmas in Schatten spaces.
\end{abstract}

\tableofcontents

\section*{Introduction}

\subsection{Context}

The time-dependent Hartree equation
\begin{equation}\label{eq:hartree}
 i\partial_t \gamma = [-\Delta_x+w*\rho_\gamma,\gamma]
\end{equation}
models the mean-field evolution of  several quantum particles moving in $\R^d$ and interacting via a two-body potential $w:\R^d\to\R$. At time $t\in\R$, the quantum state of these particles is described by its one-body density matrix $\gamma(t)$ which is a non-negative bounded operator on $L^2_x(\R^d)$. If $\gamma(t,x,y)$ denotes the integral kernel of $\gamma(t)$ (where $x,y\in\R^d$), we used the notation
$$\rho_{\gamma(t)}(x):=\gamma(t,x,x),\ x\in\R^d,$$ which corresponds physically to the spatial distribution of particles, and which is refered to as the \emph{density} of $\gamma(t)$. Notice that when $\gamma=|u\rangle\langle u|$ is a rank-one operator, one recovers the nonlinear Schr\"odinger equation where the unknown is the wavefunction $u\in L^2_x(\R^d,\C)$,
$$i\partial_t u = -\Delta_x u + (w*|u|^2)u.$$

If $g\in L^1(\R^d,\R_+)$, then $\gamma=g(-i\nabla_x)$ is a stationary (that is, time independent) solution to \eqref{eq:hartree} such that $\rho_\gamma$ is constant \cite{LewSab-13a}, and this constant is positive if $g\neq0$. The Hartree equation around these stationary states (that is, for $\gamma=g(-i\nabla_x)+Q$ for some (compact) perturbation $Q$) is thus coined the \emph{positive density Hartree equation}. Physically, these stationary states model homogeneous gases, important examples of which including non-interacting electron gases at zero or positive temperature \cite{GiuVig-05}. For other relevant physical examples, see \cite{LewSab-13a}.

Mathematically, the study of the nonlinear Hartree equation for density matrices started in \cite{BovPraFan-74,BovPraFan-76,ChaGla-75,Chadam-76,Zagatti-92}, for either finite-rank or trace-class $\gamma$. Since $g(-i\nabla_x)$ is never compact if $g\neq0$, solutions to the positive density Hartree equation are never compact and hence are not covered by these works. The first works about the positive density Hartree equation were \cite{LewSab-13a,LewSab-13b} and we carry on the study initiated there.

More precisely, we focus on the question of the asymptotic stability of the stationary states $g(-i\nabla_x)$: if $\gamma(0)$ is close enough to $g(-i\nabla_x)$ in some adequate topology, is it true that $\gamma(t)\to g(-i\nabla_x)$ as $t\to+\ii$ (possibly for a different topology)? This question was tackled for the first time in \cite{LewSab-13b} where the case $d=2$ and nice enough $w$ was treated. In \cite{CheHonPav-17}, the case $d\ge3$ and nice enough $w$ was treated. In \cite{ColDeS-20,ColDeS-22}, the authors managed to obtain results for $d\ge2$ and less regular $w$ (in particular, their assumptions include $w=c\delta_0$ for some $c\in\R$, which corresponds to the cubic nonlinear Schr\"odinger equation) by considering a particular ansatz for the initial data which relates \eqref{eq:hartree} to a kind of nonlinear Schr\"odinger where the unknown is a wavefunction rather than a density matrix. These works have been extended to treat more general nonlinearities (quintic or with an exchange term) in  \cite{Maleze-23,ColDanDeSMal-23}.  Finally, in \cite{Hadama-23}, the case $d=3$ and less regular $w$ (including $w=c\delta_0$) is treated, for a class of initial data which is sharp in a sense that we detail below. In this work, we extend the results of \cite{Hadama-23} to any dimension $d\ge3$ using a different method.

Our approach relies on several tools. When dealing with cubic (or any power) Schr\"odinger equations, a powerful way to obtain scattering of small solutions at the critical Sobolev regularity (which is $s=d/2-1$ for cubic nonlinearities) is to combine (fractional) Leibniz rules (to estimate the nonlinearity) with Strichartz estimates (see for instance \cite{CazWei-90}), and this is one of the reasons that the strategy of \cite{ColDeS-20,ColDeS-22} allowed to treat $w$ with low regularity. However, it was never properly understood how to extend this idea to density matrices. While \cite{Hadama-23} was a first step in this direction, we fully introduce the way to exploit fractional Leibniz rules and Strichartz estimates in the context of density matrices. In particular, we extend the approach of \cite{Li-19} to fractional Leibniz rules to our setting. Fractional Leibniz rules are useful in odd dimensions $d$ for which $s=d/2-1$ is not an integer. In even dimensions $d$, $s$ is an integer and standard Leibniz rules are enough. In odd dimensions, one could alternatively work with the non-optimal $s=(d-1)/2$ which is now an integer and for which standard Leibniz rules are enough. Still, there is another major tool which is widely used in the context of semilinear Schrödinger equations and for the which a satisfying extension to density matrices was not known: inhomogeneous Strichartz estimates. In this work, we also extend \cite{Hadama-23} in this regard by finding an adequate setting to state and use Christ-Kiselev lemmas \cite{ChrKis-01,Tao-00} adapted to density matrices, which is the key tool to obtain inhomogeneous Strichartz estimates.

If one introduces a semiclassical parameter $\hbar>0$ in the positive density Hartree equation, it can be shown that one recovers the positive density Vlasov equation as $\hbar\to0$ \cite{LewSab-20}. The stationary states $g(-i\nabla_x)$ (or $g(-i\hbar\nabla_x)$ in a semiclassical context) correspond to phase-space distributions $f(x,v)=g(v)$ that are homogeneous in the spatial variable $x$ (and only depend on the velocity variable $v$). These homogeneous distributions are then also stationary solutions to the nonlinear Vlasov equation, and their asymptotic stability is related to the so-called \emph{Landau damping} phenomenon \cite{Landau-46}, which was the subject of several mathematical works lately \cite{MouVil-11,BedMouMas-18,HanNguRou-21,IonPauWanWid-24}. In the context of the Hartree equation, the semilinear structure of the problem enables the use of powerful Strichartz estimates which allow to treat quite singular $w$'s like $w=c\delta_0$ which is not covered by the results for the Vlasov equation, as well as low regularity $g$'s (like a characteristic function of a ball, corresponding to the free Fermi sea at zero temperature \cite{Hadama-23}). Interestingly, the stability of both the linearized Hartree and Vlasov equations around homogeneous states boils down to a very similar criterion which is called Penrose criterion in the Vlasov case \cite{Penrose-60}, see Section \ref{sec:penrose}. In the physically relevant setting of the (repulsive) Coulomb potential $w(x)=|x|^{-1}$ in $d=3$ (which is the setting originally studied by Landau and Penrose), the study of the linearized operator has been performed in \cite{HanNguRou-21b} in the Vlasov case and in \cite{NguYou-23} in the Hartree case. Extending these results to the full nonlinear equation is a major open problem in both the Vlasov and Hartree cases.

\subsection{Main results}

To state our main result, we need to introduce the adequate functional spaces in which we work. For any $s\ge0$ and $\alpha\ge1$, we define
$$\cH^{s,\alpha}:=\{Q\in\gS^\alpha(L^2(\R^d)),\ \sd^s Q\sd^s \in \gS^\alpha(L^2(\R^d))\},$$
where $\sd:=(1-\Delta_x)^{1/2}$ and $\gS^\alpha(\gH)$ denotes the Schatten space of all compact operators $A$ on the Hilbert space $\gH$ such that $\tr_\gH|A|^\alpha<+\ii$. We then denote the associated Schatten norm by $\|A\|_{\gS^\alpha}:=(\tr_\gH|A|^\alpha)^{1/\alpha}$. Similarly, the space $\cH^{s,\alpha}$ is endowed with the natural norm $\|Q\|_{\cH^{s,\alpha}}:=\|\sd^sQ\sd^s\|_{\gS^\alpha}$. The spaces $\cH^{s,\alpha}$ are the analogue of Sobolev spaces in the context of density matrices, with an additional summability index $\alpha$. As was already noticed in \cite{LewSab-13b}, Schatten spaces are a very useful tool to study the positive density Hartree equation, in particular to control the terms coming from the background $g(-i\nabla_x)$. This will become apparent in Section \ref{sec:strategy} below where we explain our proof strategy.

Our main result is:

\begin{theorem}\label{thm:main1}

 Let $d\ge3$. Let $w\in\cS'(\R^d)$ be such that $\hat{w}$ is even, continuous and bounded. Let $g\in L^1(\R^d,\R_+)$ be such that $\langle\xi\rangle^{2(d-2)}g(\xi)\in L^\ii_\xi(\R^d)$ and $\sup_{\omega\in\Sph^{d-1}}\int_0^\ii t|\hat{g}(t\omega)|\,dt<+\ii$. Assume that $(w,g)$ is Penrose stable in the sense of Definition \ref{def:penrose} below. Then, there exists $\epsilon_0>0$ such that for any $Q_{\rm{in}}\in\cH^{d/2-1,2d/(d+1)}$ with $\|Q_{\rm{in}}\|_{\cH^{d/2-1,2d/(d+1)}}\le\epsilon_0$, there exists $Q\in C^0_t(\R,\cH^{d/2-1,2})$ with $\rho_Q\in L^2_t(\R,H^{d/2-1}_x)$ such that $\gamma(t)=g(-i\nabla_x)+Q(t)$ is a solution to \eqref{eq:hartree} and such that $\gamma(0)=g(-i\nabla_x)+Q_{\rm{in}}$. Furthermore, $\gamma(t)$ scatters linearly around $g(-i\nabla)$ as $t\to\pm\ii$: there exists $Q_\pm\in\cH^{d/2-1,2d/(d-1)}$ such that
 $$\lim_{t\to\pm\ii}\left\| \gamma(t) - g(-i\nabla_x) - e^{it\Delta_x}Q_\pm e^{-it\Delta_x} \right\|_{\cH^{d/2-1,2d/(d-1)}}=0.$$

\end{theorem}

\begin{remark}
The space $\cH^{d/2-1,2d/(d+1)}$ for the initial data is sharp in two senses: firstly, the number of derivatives $s=d/2-1$ is critical for cubic nonlinearities (such as $w=c\delta_0$) and so cannot be lowered in general. Secondly, the Schatten exponent $\alpha=2d/(d+1)$ is also sharp since if one considers the linear evolution ($g=0$ and $w=0$), for which $\gamma(t)=e^{it\Delta_x}Q_{\text{in}}e^{-it\Delta_x}$. The estimate
$$\|\rho_{\gamma(t)}\|_{L^2_t H^s_x}\lesssim \|\sd^s Q_{\text{in}}\sd^s\|_{\gS^\alpha}$$
then holds for $s=d/2-1$ only if $\alpha\le 2d/(d+1)$. Indeed, it implies that
$$\|\rho_{\gamma(t)}\|_{L^2_{t,x}}\lesssim \|\sd^{s} Q_{\text{in}}\sd^{s}\|_{\gS^\alpha}$$
and this implies that $\alpha\le2d/(d+1)$ by \cite[Sec. 4.7]{BezHonLeeNakSaw-19}.

\end{remark}

\begin{remark}
 Sufficient conditions ensuring that $(w,g)$ is Penrose stable are given in Proposition \ref{prop:penrose-sufficient}.
\end{remark}

\begin{remark}
 In Theorem \ref{thm:main1}, we only mention existence of solutions to the Hartree equation. A result about uniqueness is given in Proposition \ref{prop:uniqueness} below, and requires slightly more decay assumptions on $g$.
\end{remark}

\subsection{Comparison with existing results}

\begin{enumerate}

 \item In \cite{LewSab-13b}, the authors consider $d=2$ where the critical exponent is $s=d/2-1=0$ so there are no derivatives. The initial data $Q_{\text{in}}$ is taken in $\gS^{4/3}$ which is indeed the sharp Schatten space $\gS^\alpha$ with $\alpha=2d/(d+1)$. The interaction potential $w$ has additional assumptions compared to us, most notably that $|\xi|^{1/2}\hat{w}(\xi)\in L^\ii(\R^2)$ due to the reaction terms from $g(-i\nabla_x)$ that are quadratic in $V$ (see \cite[Prop. 4]{LewSab-13b}). Due to these quadratic terms, our method does not provide a substantial improvement in $d=2$. The only result of \cite{LewSab-13b} that we can improve is \cite[Lem. 3]{LewSab-13b}: using our method, we can replace the assumption that $\check{g}\in L^1(\R^2)$ by $g\in L^1$  (because we only need that $g(-i\nabla_x)$ is a bounded operator).

 \item In \cite{CheHonPav-17}, the authors consider $d\ge3$ and work with a regularity $s>d/2-1$ which does not reach the critical one. On the other hand, they can consider initial data $Q_{\text{in}}$ belonging to $\cH^{s,\alpha}$ with $\alpha=2$ while we have the stronger assumption with $\alpha=2d/(d+1)<2$. For the terms involving only $Q_{\text{in}}$ \cite[Prop. 6.1]{CheHonPav-17}, they only need $\hat{w}\in L^\ii(\R^d)$  which is the same as us (but $s>d/2-1$ is still needed). For the reaction terms coming from $g(-i\nabla_x)$, they need much more assumptions on $w$: it must be split as $w=w_1*w_2$ with several assumptions on $\hat{w_1}$ and $\hat{w_2}$ which are not so explicit in terms of $w$ (see \cite[Thm. 1.1, (ii)]{CheHonPav-17}). In particular, their assumptions imply that that $|\xi|^{-1/2}\hat{w}(\xi)$ remains bounded around $\xi=0$, so in particular $\hat{w}(0)=0$ which does not seem relevant from a physical point of view. Also, in \cite{CheHonPav-17}, the authors overlook the issue regarding nested time integrals that requires that we use a kind of Christ-Kiselev lemma. For instance in \cite[Eq. (5.8)]{CheHonPav-17}, the authors estimate the nested integrals $\int_0^\ii dt\int_0^t dt_1\int_0^{t_1}\cdots$ by $\int_0^\ii dt\int_0^\ii dt_1\int_0^{\ii}\cdots$ invoking the inequality $|\tr AB|\le \tr|A||B|$. While this inequality is true, what they use is a version of it with more operators, for instance $|\tr ABC|\le \tr |A||B||C|$ which is clearly wrong (the trace on the right side may not even be non-negative). This issue can be resolved using the Christ-Kiselev lemmas that we present here.

 \item In \cite{ColDeS-20,ColDeS-22}, the authors consider initial data of the form $\esp|X\rangle\langle X|$ where $X$ is a random variable with values in $L^2(\R^d)$. They show that one can find solutions that still have this form, and the corresponding equation on $X$ is
 $$i\partial_t X= -\Delta_x X + (w*_x\esp|X|^2)X,$$
 so that one may use techniques closer to standard nonlinear Schr\"odinger equations. They can then treat any $d\ge2$ and their assumptions on $w$ are  weaker than the ones of \cite{CheHonPav-17} (but stronger than our $\hat{w}\in L^\ii$): in $d=3$ they can treat any $w$ which is a finite Borel measure and in $d\ge4$ they can treat any $w$ with $w\in W^{d/2-1,1}$ such that $\langle\xi\rangle\hat{w}(\xi)\in L^{2(d+2)/(d-2)}_\xi$. They also work at the critical regularity $s=d/2-1$. A natural question is to ask what kind of initial $Q_{\text{in}}$ can be written in the above form. In \cite[Remark  1.5]{ColDeS-22}, the authors mention that they can treat any non-negative $Q_{\text{in}}\in\cH^{s,1}$ with $\rho_{Q_{\text{in}}}\in L^{3/4}_x$, while we can treat any $Q_{\text{in}}\in\cH^{s,2d/(d+1)}$. Hence, we improve on the Schatten exponent, and we don't need the non-negativity of $Q_{\text{in}}$ (which is important from a physical point of view, we can both add and remove locally particles to the gas), neither the condition $\rho_{Q_{\text{in}}}\in L^{3/4}_x$.

 \item In \cite{Hadama-23}, which is the main motivation of this work, the author obtains the same result as us for $d=3$. In this work, the author estimates $\rho$ in $H^{1/2}_x$ by interpolating between $L^2_x$ and $H^1_x$, while we are able to treat any dimension $d\ge3$, with a different method based on Leibniz rules to estimate terms in $H^s$. It may be possible to obtain our result with interpolation methods, but it is not clear to us how to find the right spaces to do interpolation with. Also, Leibniz rules for operators or operators densities may have other applications elsewhere so this is why we choose this approach.

 \item Finally, let us mention the very recent works \cite{You-24,Smith-24}, which have the common feature of extending the strategy of the Vlasov case to the Hartree case. In \cite{You-24}, this is done by proving pointwise-in-time decay estimates on $\rho$ (and its derivatives), of the type $|\rho_Q(t,x)|\lesssim t^{-d}$, similarly to \cite{HanNguRou-21} in the Vlasov case. In \cite{Smith-24}, bounds of the type $|\hat{\rho_Q}(t,k)|\lesssim\langle k,tk\rangle^{-N}$ are proved, similarly to \cite{BedMouMas-18}. This last work has the advantage to be uniform in the semiclassical limit, allowing to treat both quantum and classical cases at once. In both \cite{You-24,Smith-24}, the price to pay is to have very regular and decaying initial data, while we work at the lowest level of regularity (and without pointwise space decay assumptions).

\end{enumerate}

\subsection{Proof strategy and novelties}\label{sec:strategy}

The general strategy is the same as the original one of \cite{LewSab-13b}, namely to build first the density $\rho_Q$ by a fixed point strategy, and then deduce properties of the full state $\gamma(t)$ using the relation
$$\gamma(t)=U_V(t)(g(-i\nabla_x)+Q_{\text{in}})U_V(t)^*,$$
where $V=w*\rho_Q$ and $U_V(t)$ is the propagator of $-\Delta+V(t)$, namely the unique solution to
$$
\begin{cases}
 i\partial_t U_V(t) = (-\Delta+V(t))U_V(t),\\
 U_V(0)=\text{id}_{L^2(\R^d)}.
\end{cases}
$$
To find the right space to do the fixed point theorem for $\rho_Q$, we notice that in the case of the cubic nonlinear Schr\"odinger equation $i\partial_t u = -\Delta u +c|u|^2 u$ (which we can find from the Hartree equation choosing $g=0$,  $Q_{\text{in}}=|u_{\text{in}}\rangle\langle u_{\text{in}}|$ and $w=c\delta_0$), a natural space to obtain scattering for small solutions is $u\in L^4_t W^{s,2d/(d-1)}_x$ where $s=d/2-1$ is the critical regularity (this can be seen for instance in \cite{CazWei-90}). The $L^4_t$ part is nice with respect to the cubic nonlinearity, since then $|u|^2 u$ belongs to $L^{4/3}_t$ which is the dual of $L^4_t$, which is relevant when combined with inhomogeneous Strichartz. Once the power $4$ is chosen for the time Lebesgue exponent, the space Lebesgue exponent $2d/(d-1)$ is chosen such that $(4,2d/(d-1))$ is admissible for the Strichartz estimate. Now since $\rho_\gamma=|u|^2$ for $\gamma=|u\rangle\langle u|$ and since $u\in L^4_t W^{s,2d/(d-1)}_x$ implies that $|u|^2$ belongs to $L^2_t H^s_x$, we deduce that $L^2_t H^s_x$ is  a natural space to put $\rho_Q$. Note also that since $\hat{w}\in L^\ii$, $\rho_Q\in L^2_t H^s_x$ implies that $V\in L^2_t H^s_x$.

We thus have to estimate $\rho_Q$ in $L^2_t H^s_x$. When $s$ is an integer, this can be done recalling that $\rho_Q(x)=Q(x,x)$ so that for instance $\partial_j \rho_Q = \rho_{\partial_j Q}-\rho_{Q\partial_j}$ and the derivatives are moved on the operator $Q$ (in the same spirit as standard Leibniz rules for the integer derivatives of a product of two functions). When $s$ is not an integer, we have to understand how such a fact persists and this is the content of our Lemma \ref{lem:leibniz-gamma}. Taking fractional derivatives has other consequences that we will mention later, but for now let us explain our strategy in the simple case $s=1$ (that corresponds to $d=4$). Since we consider an initial condition $
g(-i\nabla_x)+Q_{\text{in}}$, we deduce that for $Q(t)=\gamma(t)-g(-i\nabla_x)$ we have
$$\rho_{Q(t)}=\rho_{U_V(t)Q_{\text{in}}U_V(t)^*}+\rho_{U_V(t)g(-i\nabla_x)U_V(t)^*-g(-i\nabla_x)}.$$
In particular, we have two contributions to $\rho_{Q(t)}$: one coming from the initial data $Q_{\text{in}}$ and one coming from the background $g(-i\nabla_x)$ (these last ones can be understood as the reaction of the background to the potential $V$). The contribution coming from the initial data is the same as the one that we would have if there was no background ($g=0$), which corresponds to a more ``standard'' nonlinear Hartree equation.

As we said, to estimate $\rho_{U_V(t)Q_{\text{in}}U_V(t)^*}$ in $L^2_t H^1_x$, we need to estimate
$$\rho_{U_V(t)Q_{\text{in}}U_V(t)^*},\ \rho_{\partial_j U_V(t)Q_{\text{in}}U_V(t)^*},\ \rho_{U_V(t)Q_{\text{in}}U_V(t)^*\partial_j}$$
in $L^2_{t,x}$. Here, we can guess the right assumption on $Q_{\text{in}}$ to do so, by looking at the case $V=0$ so that $U_V(t)=e^{it\Delta_x}$ is the free propagator. It is a result of \cite{BezHonLeeNakSaw-19} and that we recall below in Corollary \ref{coro:deriv-distrib} that
\begin{equation}\label{eq:intro-bez}
 \|\rho_{e^{it\Delta_x}\Gamma e^{-it\Delta_x}}\|_{L^2_{t,x}}\lesssim \|\sd^{s_1}\Gamma\sd^{s_2}\|_{\gS^\alpha},
\end{equation}
with $\alpha=2d/(d+1)=8/5$ (which is sharp) and $s_1+s_2=s=1$. Actually, in \cite{BezHonLeeNakSaw-19} it is only stated for $s_1=s_2=s/2$ and we explain how to have it for general $s_1$, $s_2$ in Corollary \ref{coro:deriv-distrib} below. This generalization to $s_1\neq s_2$ is quite important to us since we want to apply it to $\Gamma=Q_{\text{in}}$, $\Gamma=\partial_jQ_{\text{in}}$, and $\Gamma=Q_{\text{in}}\partial_j$. In particular, in order to get a symmetric estimate in the end it is natural to choose $s_1=s_2=1/2$ for $Q_{\text{in}}$, $s_1=0$, $s_2=1$ for $\partial_j Q_{\text{in}}$, and $s_1=1$, $s_2=0$ for $Q_{\text{in}}\partial_j$. We deduce that
\begin{equation}\label{eq:intro-rho-H1}
\|\rho_{e^{it\Delta_x}Q_{\text{in}} e^{-it\Delta_x}}\|_{L^2_t H^1_x}\lesssim \|\sd Q_{\text{in}}\sd\|_{\gS^{8/5}},
\end{equation}
so that $Q_{\text{in}}\in\cH^{1,8/5}$ is the natural assumption on $Q_{\text{in}}$. In general dimensions, the analogue assumption will be $Q_{\text{in}}\in\cH^{s,\alpha}$ with $s=d/2-1$ and $\alpha=2d/(d+1)$.

To show the same estimate with $e^{it\Delta_x}$ replaced by $U_V(t)$ is the content of Section \ref{sec:strichartz-UV}. The first difference between $e^{it\Delta_x}$ and $U_V(t)$ is that $U_V(t)$ does not commute with derivatives, which we used in the argument above to estimate
$\rho_{\partial_j e^{it\Delta_x}Q_{\text{in}}e^{-it\Delta_x}} = \rho_{e^{it\Delta_x}\partial_j Q_{\text{in}}e^{-it\Delta_x}}$
in terms of a Schatten norm of $\partial_j Q_{\text{in}}$. Hence, we not only show that $U_V(t)$ satisfies the same Strichartz estimates as $e^{it\Delta_x}$ but also that $\sd U_V(t) \sd^{-1}$ satisfies the same estimates as $e^{it\Delta_x}$. For standard nonlinear Schrödinger equation, an important tool to estimate $U_V$ perturbatively is inhomogeneous Strichartz estimates, which can be obtained from the homogeneous ones from the Christ-Kiselev lemma \cite{ChrKis-01,Tao-00}. In the context of density matrices, the analogue of inhomogeneous Strichartz estimates \cite[Cor. 1]{FraLewLieSei-13} is not as powerful as its analogue for functions. There are several ways around this issue: in \cite{LewSab-13b}, $U_V$ was developed in a Dyson expansion and each term in the expansion is estimated by going back to the proof of Strichartz estimates. In \cite{CheHonPav-17}, the authors manage to use a version of inhomogeneous Strichartz estimates as powerful as their analogue for functions, using Strichartz estimates for the integral kernel $Q_{\text{in}}(x,y)$. The price to pay is to use norms on $Q_{\text{in}}$ based on $L^2_{x,y}$ (in other words, Hilbert-Schmidt norms on $Q_{\text{in}}$) and in particular they cannot reach the optimal regularity $s=d/2-1$, and need the interaction potential $w$ to be regularizing as well. In \cite{ColDeS-20,ColDeS-22}, they author work with a nonlinear Schrödinger equation on functions (which depend on an additional parameter $\omega$), so that they are able to use ``standard'' inhomogeneous Strichartz estimates and can reach the optimal regularity $s=d/2-1$ and an interaction potential $w$ merely in $L^1$ (or a finite measure). Finally, in \cite{Hadama-23}, a better inhomogeneous Strichartz estimate for density matrices is used we extend this strategy in this work. It allows to work at the sharp regularity $s=d/2-1$ and at the sharp Schatten exponent $\alpha=2d/(d+1)$. However, in \cite{Hadama-23} the author uses interpolation estimates to deal with $H^s$-norms (in his case, $H^{1/2}$) and hence needs to again use Dyson expansions to do interpolation. Since we do not use interpolation, we do not need to use Dyson expansions and estimate directly $U_V$ with a strategy that is closer to standard inhomogeneous estimates. Furthermore, the range of exponents for which one can apply interpolation methods is unclear. With our method based on fractional Leibniz rules and Christ-Kiselev lemmas, we do not rely on interpolation and prove Strichartz estimates for $U_V$ with general assumptions on the various exponents involved.

More precisely, we prove Strichartz estimates for $\cU(t):=\sd U_V(t) \sd^{-1}$  through their dual formulation (that we recall in Section \ref{sec:free} below),
$$\|h(t)\cU(t)\|_{\gS^{2\alpha'}(L^2_x\to L^2_{t,x})} \lesssim \|h\|_{L^{2p'}_t L^{2q'}_x},$$
together with the Duhamel formula
$$h(t)\cU(t)=h(t)e^{it\Delta_x}-ih(t)e^{it\Delta_x}\int_0^t e^{-i\tau\Delta_x}\sd V(\tau)\sd^{-1}\cU(\tau)\,d\tau.$$
First, let us notice that using Leibniz type formulas, we may think as $\sd V(\tau)\sd^{-1}$ as being the operator $V(\tau)+(\sd V)(\tau)\sd^{-1}$. Next, notice that if the integral $\int_0^t\,d\tau$ is replaced by $\int_0^\ii\,d\tau$, it becomes easy to estimate it, because $h(t)e^{it\Delta_x}:L^2_x\to L^2_{t,x}$ belongs to a Schatten space due to Strichartz estimates for $e^{it\Delta_x}$, while $\int_0^\ii e^{-i\tau\Delta_x}V(\tau)\cU(\tau)$ can be seen as the composition of the operators $\sqrt{V(\tau)}\cU(\tau):L^2_x\to L^2_{\tau,x}$ (which is similar to the one we estimate so it can be absorbed by the left side if $V$ is small enough), and the operator $\int_0^\ii e^{-i\tau\Delta_x}\sqrt{|V(\tau)|}\,d\tau:L^2_{\tau,x}\to L^2_x$ which is known to be bounded (even in some Schatten class) by Strichartz estimates for $e^{it\Delta_x}$. As for inhomogeneous Strichartz estimates, the replacement of $\int_0^\ii\,d\tau$ by $\int_0^t\,d\tau$ is non-trivial and can be done via the use of a Christ-Kiselev lemma. We thus extend the proof of  \cite{Tao-00} to Schatten spaces and provide the relevant results in Section~ \ref{sec:CK}. It implies in particular that an operator of the form
$$h(t)e^{it\Delta_x}\int_0^t e^{-i\tau\Delta_x}\tilde{h}(\tau)\,d\tau: L^2_{\tau,x}\to L^2_{t,x}$$
belongs to Schatten spaces since $h(t)e^{it\Delta_x}:L^2_x\to L^2_{t,x}$, $\tilde{h}(\tau)e^{i\tau\Delta_x}:L^2_x\to L^2_{\tau,x}$ belong to Schatten spaces. At this point, we emphasize that this type of Christ-Kiselev result deals with operators from $L^2_{\tau,x}$ to $L^2_{t,x}$ (that is, there is a time variable both in the source and target spaces). To obtain such a structure, it was important to see the operator $\int_0^t e^{-i\tau\Delta_x}V(\tau)\cU(\tau)\,d\tau: L^2_x\to L^2_{t,x}$ as the composition of the operators $\sqrt{V(\tau)}\cU(\tau):L^2_x\to L^2_{\tau,x}$ and $\int_0^t e^{-i\tau\Delta_x}\sqrt{|V(\tau)|}\,d\tau: L^2_{\tau,x}\to L^2_{t,x}$ (so that the second operator indeed maps $L^2_{\tau,x}$ to $L^2_{t,x}$), which comes from the splitting $V(\tau)=\sqrt{|V(\tau)|}\sqrt{V(\tau)}$. While this splitting is always possible for any multiplication operator, we will need to provide for more general operators of the type $\sd^s V(\tau)\sd^{-s}$ and this will greatly influence the way we write our Leibniz rules in Section \ref{sec:leibniz} (for instance, see in this respect \eqref{eq:decomp-leibniz-schatten-J} where $\sd^s V(\tau)\sd^{-s}$ is written as a sum of terms, each of which is a product with a good structure to combine it with a Christ-Kiselev lemma). Notice also that while $V(\tau)=\sqrt{|V(\tau)|}\sqrt{V(\tau)}$ can be seen as a composition of two multiplication operators, we will expand $\sd^s V(\tau)\sd^{-s}$ in compositions of more general operators than multiplication operators (typically, $A^*B$ with $A,B:L^2_x\to L^2(X)$ for some measure space $X$. This general measure space is in particular useful to include Littlewood-Paley frequency summation, see the proof of Lemma \ref{lem:leibniz-op-general} below). For this reason, we also need to extend (dual) Strichartz estimates to more general operators than multiplication operators (see Section \ref{sec:gene-stri}). To conclude, Strichartz estimates for $U_V$ are thus obtained through the use of two main ingredients: Christ-Kiselev lemmas in Schatten spaces and fractional Leibniz rules adapted to these Christ-Kiselev lemmas. This is done in Section \ref{sec:strichartz-UV}.

Let us now detail how to treat the reaction terms $\rho_{U_V(t)g(-i\nabla_x)U_V(t)^*-g(-i\nabla_x)}$. To understand how they behave, it useful to think of Dyson expansions where $U_V(t)$ is expanded in powers of $V$. Since $\rho_{U_V(t)g(-i\nabla_x)U_V(t)^*-g(-i\nabla_x)}$ vanishes when $V=0$, there is no constant term in this expansion. The linear terms in $V$ are treated independently by assuming linear stability (with sufficient conditions to have it, see Section \ref{sec:penrose}). One may then think of the term of order $n\ge2$ in $V$ of $U_V(t)g(-i\nabla_x) U_V(t)^*-g(-i\nabla_x)$ as
$$e^{it\Delta_x}\left(\int_\R e^{-i\tau\Delta_x}V(\tau)e^{i\tau\Delta_x}\,d\tau\right)^ng(-i\nabla_x)e^{-it\Delta_x}$$
so that by \eqref{eq:intro-rho-H1}, we need to estimate
$$\left\| \sd \left(\int_\R e^{-i\tau\Delta_x}V(\tau)e^{i\tau\Delta_x}\,d\tau\right)^n g(-i\nabla_x)\sd \right\|_{\gS^{8/5}}.$$
Under the assumption that $\langle\xi\rangle^2g(\xi)\in L^\ii_\xi$, the idea to control this term is to commute successively the $\sd$ with each of the $V$'s to obtain finally $\sd g(-i\nabla_x)\sd$ which is a bounded operator. Since the number of derivatives is an integer here, we can use standard Leibniz rules to infer that is is enough to estimate for any $k=0,\ldots,n-1$
$$
\begin{multlined}
  \left\|\left(\int_\R e^{-i\tau\Delta_x}V(\tau)e^{i\tau\Delta_x}\,d\tau\right)^{k} \left(\int_\R e^{-i\tau\Delta_x}(\partial_jV)(\tau)e^{i\tau\Delta_x}\,d\tau\right)\right.\times\\
  \left.\times\left(\int_\R e^{-i\tau\Delta_x}V(\tau)e^{i\tau\Delta_x}\,d\tau\right)^{n-1-k} g(-i\nabla_x)\sd \right\|_{\gS^{8/5}}
\end{multlined}
$$
as well as
$$\left\|  \left(\int_\R e^{-i\tau\Delta_x}V(\tau)e^{i\tau\Delta_x}\,d\tau\right)^n \sd g(-i\nabla_x)\sd \right\|_{\gS^{8/5}}.$$
This last term is easier to estimate so let us explain how it works. The first ingredient is the dual of \eqref{eq:intro-bez} which implies that
$$\int_\R e^{-i\tau\Delta_x}h(\tau)e^{i\tau\Delta_x}\,d\tau\sd^{-1}\in\gS^{8/3}$$
if $h\in L^2_{t,x}$. The second ingredient is the fact that  $V\in L^2_t H^1_x\hookrightarrow L^2_t L^4_x$ so that
$$\int_\R e^{-i\tau\Delta_x}V(\tau)e^{i\tau\Delta_x}\,d\tau\in\gS^8.$$
Hence, writing
$$
\begin{multlined}
\left(\int_\R e^{-i\tau\Delta_x}V(\tau)e^{i\tau\Delta_x}\,d\tau\right)^n \sd g(-i\nabla_x)\sd \\
= \left(\int_\R e^{-i\tau\Delta_x}V(\tau)e^{i\tau\Delta_x}\,d\tau\right)^{n-1} \int_\R e^{-i\tau\Delta_x}V(\tau)e^{i\tau\Delta_x}\,d\tau\sd^{-1}\sd^2g(-i\nabla_x)\sd
\end{multlined}
$$
and using the Hölder inequality in Schatten spaces, we deduce that
$$\left(\int_\R e^{-i\tau\Delta_x}V(\tau)e^{i\tau\Delta_x}\,d\tau\right)^n \sd g(-i\nabla_x)\sd\in\gS^{8/5}$$
if $n\ge3$ and $\langle\xi\rangle^3g(\xi)\in L^\ii_\xi$. If $n=2$, we introduce another derivative by writing
$$
\begin{multlined}
\left(\int_\R e^{-i\tau\Delta_x}V(\tau)e^{i\tau\Delta_x}\,d\tau\right)^2 \sd g(-i\nabla_x)\sd \\
= \int_\R e^{-i\tau\Delta_x}V(\tau)e^{i\tau\Delta_x}\,d\tau \sd^{-1}\int_\R e^{-i\tau\Delta_x}\sd V(\tau)e^{i\tau\Delta_x}\,d\tau\sd^{-1}\sd^2 g(-i\nabla_x)\sd.
\end{multlined}
$$
Expanding $\sd V\sim V+\sum_j \partial_j V$ by a Leibniz rule and using that both terms involving $V$ belong to $\gS^{8/3}$, we deduce again by the Hölder inequality in Schatten spaces that
$$\left(\int_\R e^{-i\tau\Delta_x}V(\tau)e^{i\tau\Delta_x}\,d\tau\right)^2 \sd g(-i\nabla_x)\sd \in \gS^{8/5}.$$
The terms
$$
\begin{multlined}
\left(\int_\R e^{-i\tau\Delta_x}V(\tau)e^{i\tau\Delta_x}\,d\tau\right)^{k} \left(\int_\R e^{-i\tau\Delta_x}(\partial_jV)(\tau)e^{i\tau\Delta_x}\,d\tau\right)\times\\
\times\left(\int_\R e^{-i\tau\Delta_x}V(\tau)e^{i\tau\Delta_x}\,d\tau\right)^{n-1-k} g(-i\nabla_x)\sd
\end{multlined}
$$
can be estimated using the same ideas, except for one term: the one with $n=2$ and $k=1$ given by
$$
\left(\int_\R e^{-i\tau\Delta_x}V(\tau)e^{i\tau\Delta_x}\,d\tau\right) \left(\int_\R e^{-i\tau\Delta_x}(\partial_jV)(\tau)e^{i\tau\Delta_x}\,d\tau\right)\sd^{-1}\sd g(-i\nabla_x)\sd
$$
which only belongs to $\gS^2$ (since the first $V$ contributes to $\gS^8$ and the second $V$ contributes to $\gS^{8/3}$) and not to $\gS^{8/5}$. To get around this problem, we cannot introduce a $\sd^{-1}\sd$ between the $V$ and $\partial_j V$ as before since $\partial_j V$ belongs only to $L^2_{t,x}$ and we cannot take another derivative of it. Here, the idea is to introduce half a derivative on $\partial_j V$ by writing
$$\partial_j V = \sd^{-1/2} \sd^{1/2}\partial_j V \sim \sd^{-1/2} (\partial_j V + |D|^{1/2}\partial_j V).$$
Using that $V\in L^2_t H^1_x \hookrightarrow L^2_t L^{8/3}_x$ we have
$$\int_\R e^{-i\tau\Delta_x}V(\tau)e^{i\tau\Delta_x}\,d\tau\sd^{-1/2}\in\gS^4.$$
The term involving $|D|^{1/2}\partial_jV$ is peculiar because $\partial_j V$ is merely in $L^2_x$ so that $|D|^{1/2}\partial_jV\in\dot{H}^{-1/2}_x$. Then, we may use \cite[Thm. 3.1]{CheHonPav-17} which implies that
$$\int_\R e^{-i\tau\Delta_x} (|D|^{1/2}h(\tau))e^{i\tau\Delta_x}\,d\tau\sd^{-3/2-\epsilon}\in\gS^2$$
for any $\epsilon>0$ if $h\in L^2_{t,x}$ which concludes the argument. The rigorous argument in any dimension is carried out in Section \ref{sec:reaction}, with the additional difficulty that the time integral are not over all $\R$ but rather from $0$ to some $t$ and one again has to use some Christ-Kiselev ideas. This last point is particularly relevant for the term in $\dot{H}^{-1/2}$ that appeared at the very end of the argument. Indeed, as we argued above an important point of the Christ-Kiselev argument above is to write the potential terms as products of two functions/operators. For a potential in $\dot{H}^{-1/2}_x$, we did not find a way to factorize it properly so we had to resort to the fact that it has good Schatten properties (as we said, it belongs to $\gS^2$) so that a weaker Christ-Kiselev lemma may be used (see Corollary \ref{coro:CK-2} and its application in Step 6 in the proof of Proposition \ref{prop:rho22}).

\subsection{Structure of the article}

In Section \ref{sec:prelim}, we begin with setting the fixed point argument by which we prove Theorem \ref{thm:main1}. We also discuss the Hartree equation linearized around $g(-i\nabla_x)$ and in particular we state the quantum version of the Penrose criterion for linear stability. Finally, we recall the known Strichartz estimates for $e^{it\Delta_x}$ that we will use.

In Section \ref{sec:gene-stri}, we prove Strichartz estimates which are generalized in the sense that standard (dual) Strichartz estimates like \cite[Thm. 2]{FraLewLieSei-13} deal with
multiplication operators by functions $V\in L^\mu_t L^\nu_x$. We consider more general operators with some $L^q_x$-mapping properties, and this will be useful in conjunction with the fractional Leibniz rules of Section \ref{sec:leibniz}. We state these generalized Strichartz estimates not only for $e^{it\Delta_x}$ but for any propagator $\cU(t)$ which satisfies standard Strichartz estimates (we will apply it later on to $\cU(t)=\sd^s U_V(t)\sd^{-s}$).

In Section \ref{sec:leibniz} we prove fractional Leibniz rules adapted to our setting, in particular to estimate quantities like $|D|^s\rho_\gamma$. We follow the proof of \cite{Li-19}, and in particular there is no essential novelty except the way we write the results. This section can be seen as independent from the rest of the article and may have other applications to any problem related to the analysis of density matrices. Notice also that the results of this section are useful only to work at the critical regularity $s=d/2-1$ (which is not an integer if $d$ is odd). If one accepts to work with non-optimal regularity, one can use our argument with $s=\lceil d/2-1\rceil$ for instance and apply standard Leibniz rules to differentiate a product/density instead.

In Section \ref{sec:CK} we state and prove the Christ-Kiselev lemmas that are key to deal with retarded integrals appearing in (iterated) Duhamel formulas. We prove them extending the proof of \cite{Tao-00} to Schatten spaces.

In Section \ref{sec:strichartz-UV}, we provide Strichartz estimates for the propagator $U_V(t)$ for general $V\in L^\mu_t W^{s,\nu}_x$ that are key to our argument but may also be of independent interest. To prove them, we use the fractional Leibniz rules and Christ-Kiselev lemmas of the preceding sections.

In Section \ref{sec:reaction} we estimate all the terms coming from the reaction of the background $g(-i\nabla_x)$ to an external potential $V\in L^2_t H^s_x$. This is the core of our argument, which brings together all the results of the preceding sections.

In the final Section \ref{sec:proof-thm}, we conclude by providing the proof of Theorem \ref{thm:main1}.

\bigskip

\noindent\textbf{Acknowledgements.} The authors want to thank Dong Li for his explanations on his article on fractional Leibniz rules. S.H. was supported by JSPS KAKENHI Grant Number 24KJ1338. A.B. and J. S. were supported by the ANR CPJ grant ANR-22-CPJ1-0068-01. This project has received financial support from the CNRS through the MITI interdisciplinary programs.

\newpage

\section{Preliminaries}\label{sec:prelim}

\subsection{Notations}

\begin{itemize}

 \item If $\gH$ and $\gK$ are two Hilbert spaces, we denote by $\cB(\gH,\gK)$ the space of all bounded operators from $\gH$ to $\gK$. We will also abbreviate $\cB(\gH):=\cB(\gH,\gH)$. For $\alpha\in[1,+\ii)$, we will also write $\gS^\alpha(\gH\to\gK)$ or $\gS^\alpha(\gH,\gK)$ the space of all compact operators $A:\gH\to\gK$ such that $\|A\|_{\gS^\alpha(\gH\to\gK)}^\alpha:=\tr_{\gH}(A^*A)^{\alpha/2}<+\ii$.

 \item For any $s\in\R$ and $q\in(1,+\ii)$, we denote by $W^{s,q}_x$ or $L^q_x$ the space $W^{s,q}(\R^d)$ or $L^q(\R^d)$. The index $x$ emphasizes that these are functions of the space variable $x\in\R^d$. Similarly, for any $p\in(1,+\ii)$, $L^p_t$ denotes the space $L^p(\R)$ of functions of the time variable $t\in\R$, and $L^p_t W^{s,q}_x$, $L^p_t L^q_x$ denote the spaces $L^p(\R,W^{s,q}(\R^d))$, $L^p(\R,L^q(\R^d))$.

 \item For any $z\in\C$ and any $\theta\in(0,1)$, we denote by $z^\theta:=0$ if $z=0$ and $z^\theta:=z/|z|^{1-\theta}$ if $z\neq0$, so that we always have $z=z^\theta|z|^{1-\theta}$. We define $\sqrt{z}:=z/\sqrt{|z|}$ similarly.

 \item For any $\xi\in\R^d$ we use the notation $\langle \xi\rangle:=\sqrt{1+|\xi|^2}$, so that by functional calculus we have $\sd=(1-\Delta_x)^{1/2}$.

 \item We will use the notation $A\lesssim\|f\|_X$ as an abbreviation of $A\le C\|f\|_X$ for some $C>0$ which is independent of $f$ (which belongs to some normed space $(X,\|\cdot\|_X)$).

\end{itemize}

\subsection{Fixed point setting}\label{sec:fixed}

As we mentioned in the introduction, the key to prove Theorem~\ref{thm:main1} will be to first build the density $\rho_Q$ with $Q=\gamma-g(-i\nabla_x)$ by a fixed point argument. It will rely on the fact that the solution to
\begin{equation}\label{eq:hartree-Q}
\begin{cases}
i\partial_tQ = [-\Delta+V,g(-i\nabla_x)+Q],\\
Q_{|t=0} =  Q_{\text{in}}
\end{cases}
\end{equation}
is given by
$$Q(t) = U_V(t,0)Q_{\text{in}}U_V(t,0)^*-i\int_0^tU_V(t,\tau)[V(\tau),g(-i\nabla_x)]U_V(t,\tau)^*\,d\tau ,$$
where $U_V(t,\tau)$ is the propagator of $-\Delta+V(t)$ on $L^2(\R^d)$, that is the solution to
$$
\begin{cases}
 i\partial_t U_V(t,\tau) = (-\Delta+V(t))U_V(t,\tau),\\
 U_V(\tau,\tau)=\text{id}_{L^2(\R^d)}.
\end{cases}
$$
For now, this argument is formal since we did not specify the assumptions on $V$ but the goal of Section \ref{sec:strichartz-UV} is to build the propagator $U_V(t,\tau)$ with the adequate functional properties for (small enough) $V\in L^\mu_t W^{s,\nu}_x$. As a consequence, we deduce that
$$\rho_{Q(t)} = \rho_{U_{w*\rho_Q}(t,0)Q_{\text{in}}U_{w*\rho_Q}(t,0)^*}+\rho\left[ -i\int_0^tU_{w*\rho_Q}(t,\tau)[w*\rho_Q(\tau),g(-i\nabla_x)]U_{w*\rho_Q}(t,\tau)^*\,d\tau \right],$$
an equation which depends only on $\rho_Q$ (and also on $g$ and $Q_{\text{in}}$), and which we see as a fixed point equation $\rho_Q=\Phi(\rho_Q)$. The strategy is then to apply Banach's fixed point theorem to $\Phi$. The fixed point equation that we solve is not exactly this one, because we first need to extract the linear term in $\Phi$ to properly set the fixed point argument. In \cite{LewSab-13b}, this was done by a full Dyson expansion of $U_V(t)$ in powers of $V$, while here we rather extract the first linear term and deal with the higher orders terms self-consistently (in the end, both arguments would lead to the same result so we could also use full Dyson expansions). By the Duhamel formula, we further expand
$$U_V(t,\tau) = e^{i(t-\tau)\Delta_x} + D_V(t,\tau),
$$
where
$$D_V(t,\tau):=-i\int_\tau^t e^{i(t-t_1)\Delta_x}V(t_1)U_V(t_1,\tau)\,dt_1 .$$
We can thus write
\begin{equation*}
    \begin{multlined}
     -i\int_0^tU_V(t,\tau)[V(\tau),g(-i\nabla_x)]U_V(t,\tau)^*\,d\tau  = -i\int_0^te^{i(t-\tau)\Delta_x}[V(\tau),g(-i\nabla_x)]e^{i(\tau-t)\Delta_x}\,d\tau \\
    +   -i\int_0^te^{i(t-\tau)\Delta_x}[V(\tau),g(-i\nabla_x)]D_V(t,\tau)^*\,d\tau \\
    + -i\int_0^tD_V(t,\tau)[V(\tau),g(-i\nabla_x)]e^{i(\tau-t)\Delta_x}\,d\tau \\
    + -i\int_0^tD_V(t,\tau)[V(\tau),g(-i\nabla_x)]D_V(t,\tau)^*\,d\tau ,
    \end{multlined}
\end{equation*}
and we recall the definition of the linear response
\begin{equation}\label{eq:def-cL}
-\cL[\varrho](t):=  \rho\left[ -i\int_0^te^{i(t-\tau)\Delta_x}[w*\varrho(\tau),g(-i\nabla_x)]e^{i(\tau-t)\Delta_x}\,d\tau \right].
\end{equation}
We thus define
\begin{equation}
\begin{multlined}\label{eq:Phi}
 \Phi(\varrho):=(1+\cL)^{-1}\left(\rho_{U_{w*\varrho}(t,0)Q_{\text{in}}U_{w*\varrho}(t,0)^*}+ \rho\left[ -i\int_0^te^{i(t-\tau)\Delta_x}[w*\varrho(\tau),g(-i\nabla_x)]D_{w*\varrho}(t,\tau)^*\,d\tau \right] \right. \\
 +\left. \rho\left[ -i\int_0^tD_{w*\varrho}(t,\tau)[w*\varrho(\tau),g(-i\nabla_x)]e^{i(\tau-t)\Delta_x}\,d\tau \right]\right.\\
 \left.+
 \rho\left[ -i\int_0^tD_{w*\varrho}(t,\tau)[w*\varrho(\tau),g(-i\nabla_x)]D_{w*\varrho}(t,\tau)^*\,d\tau \right]
 \right),
\end{multlined}
\end{equation}
and we will prove in Section \ref{sec:proof-thm} that $\Phi$ is a contraction on some small ball in $L^2_t H^s_x$ if $Q_{\text{in}}$ is small enough in $\cH^{d/2-1,2d/(d+1)}$.

\subsection{Linear stability}\label{sec:penrose}

As is shown in \eqref{eq:Phi}, we will need that the operator $1+\cL$ is invertible on $L^2_t H^s_x$, where $\cL$ is defined in \eqref{eq:def-cL}. This invertibility will be ensured by the following condition which is the quantum analogue of the Penrose condition for the Vlasov equation \cite{Penrose-60} (see \cite{MouVil-11} for additional details).

\begin{definition}[Quantum Penrose stability condition]\label{def:penrose}

Let $d\ge1$, $g\in L^1(\R^d,\R_+)$ be such that
\begin{equation}\label{eq:penrose-bounded}
\sup_{\omega\in\Sph^{d-1}}\int_0^\ii t|\hat{g}(t\omega)|\,dt<+\ii,
\end{equation}
and $w\in\cS'(\R^d,\R)$ be such that $\hat{w}$ is continuous and bounded. We say that $(w,g)$ is \emph{Penrose stable} if
$$\inf_{(\tau,\omega,\xi)\in(0,+\ii)\times\R\times\R^d}\left|1+2\hat{w}(\xi)\int_0^\ii e^{-t\tau}e^{-it\omega}\sin(t|\xi|^2)\hat{g}(2t\xi)\,dt\right|>0.$$

\end{definition}

\begin{remark}
 The conditions on $w$ and $g$ ensure that the function
 \begin{equation}\label{eq:penrose-integral}
     (0,+\ii)\times\R\times\R^d\ni(\tau,\omega,\xi) \mapsto \hat{w}(\xi)\int_0^\ii e^{-t\tau}e^{-it\omega}\sin(t|\xi|^2)\hat{g}(2t\xi)\,dt
 \end{equation}
 is well-defined and bounded (changing variables $t'=t|\xi|$ in the integral).
\end{remark}

\begin{remark}
  In \cite[Cor. 1]{LewSab-13b}, the factor $e^{-t\tau}$ was not included (equivalently, the condition only mentioned $\tau=0$), because the argument below including the Paley-Wiener theorem was overlooked. This was indicated to us by D. Han Kwan and F. Rousset and we are grateful to them for this remark. Despite this overlook, the results of \cite{LewSab-13b} are still valid since their conditions imply that $(w,g)$ is Penrose stable, as can for instance be seen from our argument below. Let us also notice that this correction was also introduced in \cite{Maleze-23}.
\end{remark}

Penrose stability implies the invertibility of $1+\cL$ on $L^2_t H^s_x$ as is shown below.

\begin{proposition}\label{prop:penrose-inverse}
 Let $d\ge1$, $g\in L^1(\R^d,\R_+)$ be such that
$$\sup_{\omega\in\Sph^{d-1}}\int_0^\ii t|\hat{g}(t\omega)|\,dt<+\ii,$$
and $w\in\cS'(\R^d,\R)$ be such that $\hat{w}$ is continuous and bounded. If $(w,g)$ is Penrose stable, then $1+\cL$ is invertible on $L^2_t(\R_+,H^s_x(\R^d))$ for any $s\in\R$.
\end{proposition}

\begin{proof}
 By \cite[Prop. 1]{LewSab-13b}, for any $\varrho\in \cD_{t,x}(\R_+\times\R^d)$ and any $(t,\xi)\in\R_+\times\R^d$ we have
 $$\hat{\cL[\varrho]}(t,\xi)=\int_0^\ii K(t-t',\xi)\hat{\varrho(t')}(\xi)\,dt'=K(\cdot,\xi)*_t (\hat{\varrho_0(\cdot)}(\xi))(t),$$
 where $\varrho_0(t)$ extends $\varrho(t)$ by $0$ for $t\le0$ and with
 $$K(t,\xi):=2\hat{w}(\xi)\1_{t\ge0}\sin(t|\xi|^2)\hat{g}(2t\xi).$$
 In particular, $\cL$ is a bounded operator on $L^2_t(\R_+,H^s_x(\R^d))$, with
 $$\|\cL\|_{L^2_t(\R_+,H^s_x(\R^d))\to L^2_t(\R_+,H^s_x(\R^d))}\lesssim\|\hat{w}\|_{L^\ii}\sup_{\omega\in\Sph^{d-1}}\int_0^\ii t|\hat{g}(t\omega)|\,dt.$$
 For any fixed $\xi\in\R^d$, we define $\cL_\xi:L^2_t(\R_+)\to L^2_t(\R_+)$ by
 $$(\cL_\xi f)(t)=\int_0^\ii K(t-t',\xi)f(t')\,dt'.$$
Let us first show that $1+\cL_\xi$ is invertible on $L^2_t(\R_+)$ with
 $$\sup_{\xi\in\R^d}\|(1+\cL_\xi)^{-1}\|_{L^2_t(\R_+)\to L^2_t(\R_+)}<\ii.$$
 For fixed $\xi\in\R^d$, the function $t\mapsto K(t,\xi)$ is integrable on $\R$ and supported on $[0,+\ii)$ hence its Fourier transform in time $\tilde{K}(z,\xi)$ is analytic on $\C_-:=\{z\in\C,\ \im z<0\}$. The Penrose stability condition implies that
 $$\kappa_\xi:=\inf_{z\in\C_-}|1+\tilde{K}(z,\xi)|>0,$$
 so that $z\mapsto (1+\tilde{K}(z,\xi))^{-1}$ is a bounded analytic function on $\C_-$ (bounded by $\kappa_\xi^{-1}$). For any $h\in L^2_t(\R_+)$, $\tilde{h}$ is an analytic function on $\C_-$ (where $\tilde{h}$ is computed by extending $h$ by $0$ on $\R_-$) with
 $$\|h\|_{L^2_t(\R_+)}^2 = \sup_{\tau>0}\int_\R|\tilde{h}(\omega-i\tau)|^2\,d\omega.$$
 Hence, the function $z\mapsto (1+\tilde{K}(z,\xi))^{-1}\tilde{h}(z)$ is analytic on $\C_-$ and verifies the bound
 $$\sup_{\tau>0}\int_\R|(1+\tilde{K}(\omega-i\tau,\xi))^{-1}\tilde{h}(\omega-i\tau)|^2\,d\omega\le\kappa_\xi^{-2}\|h\|_{L^2_t(\R_+)}^2.$$
 The Paley-Wiener theorem \cite[Thm. 19.2]{Rudin} implies that there exists $f\in L^2_t(\R_+)$ such that for any $z\in\C_-$,
 $$\tilde{f}(z)=(1+\tilde{K}(z,\xi))^{-1}\tilde{h}(z),$$
 (where again $\tilde{f}$ is computed by extending $f$ by $0$ on $\R_-$). In particular, we have
 $$\|f\|_{L^2_t(\R_+)}\le\kappa_\xi^{-1}\|h\|_{L^2_t(\R_+)}.$$
 By computing their Fourier transform in time, one can check that $(1+\cL_\xi)f=h$. Furthermore, $f$ is the only such function by the same arguments. Hence, we deduce that $f=(1+\cL_\xi)^{-1}h$ and $\|(1+\cL_\xi)^{-1}\|_{L^2_t(\R_+)\to L^2_t(\R_+)}\le\kappa_\xi^{-1}$. Finally, we deduce by the Penrose stability that
 $$\sup_{\xi\in\R^d}\|(1+\cL_\xi)^{-1}\|_{L^2_t(\R_+)\to L^2_t(\R_+)}\le\sup_{\xi\in\R^d}\kappa_\xi^{-1}=\frac{1}{\inf_{\xi\in\R^d}\kappa_\xi}<+\ii.
 $$
 Let us now show that $1+\cL$ is invertible on $L^2_t H^s_x$. To do so, let $W\in L^2_t H^s_x$ and let us show that there exists a unique $\varrho\in L^2_t H^s_x$ such that $(1+\cL)\varrho=0$. To prove uniqueness, notice that $(1+\cL)\varrho=0$ implies that for a.e. $\xi\in\R^d$, $(1+\cL_\xi)\hat{\varrho}(\cdot,\xi)=0$. Since $\varrho\in L^2_t H^s_x$, $\hat{\varrho}(\cdot,\xi)$ belongs to $L^2_t$ for a.e. $\xi\in\R^d$. By the invertibility of $1+\cL_\xi$ on $L^2_t$, we deduce that $\hat{\varrho}(\cdot,\xi)=0$ for a.e. $\xi$ and hence that $\varrho=0$. Let us now turn to the existence of $\varrho$ such that $(1+\cL)\varrho=W$. For a.e. $\xi$, $\hat{W}(\cdot,\xi)\in L^2_t$. Hence, let us define $h:=\cF^{-1}(1+\cL_\xi)^{-1}\hat{W}(\cdot,\xi)$. Then, we have
 \begin{align*}
 \|h\|_{L^2_t H^s_x} &= \|\langle\xi\rangle^s(1+\cL_\xi)^{-1}\hat{W}(\cdot,\xi)\|_{L^2_t L^2_\xi} \\
 &\le\sup_{\xi\in\R^d}\|(1+\cL_\xi)^{-1}\|_{L^2_t(\R_+)\to L^2_t(\R_+)}\|\langle\xi\rangle^s\hat{W}(\cdot,\xi)\|_{L^2_t L^2_\xi}\\
 &\lesssim \|W\|_{L^2_t H^s_x}
 \end{align*}
 so that $h\in L^2_t H^s_x$ and indeed verifies that $(1+\cL)h=W$ since for a.e. $\xi$, $(1+\cL_\xi)\hat{h}(\cdot,\xi)=\hat{W}(\cdot,\xi)$.
\end{proof}

\begin{remark}\label{rk:penrose-negative-times}
 One can show by the same argument that $1+\cL$ in invertible on $L^2_t(\R_-,H^s_x)$. Indeed, mapping $f\in L^2_t(\R_-)$ to $t\mapsto f(-t)\in L^2_t(\R_+)$, one can compute that $1+\cL$ on $L^2(\R_-)$ corresponds to the operator on $L^2(\R_+)$ with kernel $K(t-t',-\xi)$. Since the Penrose criterion is invariant by the change $\xi\to -\xi$ (since $\hat{w}$ is even), this shows that $1+\cL$ in invertible on $L^2_t(\R_-,H^s_x)$
\end{remark}

Finally, we provide sufficient conditions on $(w,g)$ for Penrose stability which are more general than the ones of \cite{LewSab-13b,Maleze-23}. This is because the argument of \cite[Prop. 2.1]{MouVil-11} extends to the quantum case (in particular, we do not need that $g$ is radially decreasing for $d\ge3$). Since it is new in the quantum case, we provide the proof for completeness.

\begin{proposition}\label{prop:penrose-sufficient}
 Let $d\ge1$, $g\in L^1(\R^d,\R_+)$ and $w\in\cS'(\R^d,\R)$ be such that $\hat{w}$ is continuous and bounded. Assume that $g$ is a radial function and that $\int_{\R^d}|\hat{g}(x)||x|^{2-d}\,dx<+\ii$.
 \begin{enumerate}
  \item (\cite[Prop. 1]{LewSab-13b}) If
  $$\frac{\|\hat{w}\|_{L^\ii}}{2|\Sph^{d-1}|}\int_{\R^d}\frac{|\hat{g}(x)|}{|x|^{d-2}}\,dx<1,$$
  then $(w,g)$ is Penrose stable.
  \item (\cite[Cor. 1]{LewSab-13b} and \cite[Prop. 5.2]{Maleze-23}) Assume that $g$ is such that $|\xi|g(\xi)\in L^1_\xi(\R^d)$, $\int_{\R^d}|\nabla\hat{g}(x)||x|^{2-d}\,dx<+\ii$, and such that $\partial_r g$ is continuous on $(0,+\ii)$. Assume furthermore that $\partial_r g<0$ if $d=1,2$, and that $$\frac{\|\hat{w}_-\|_{L^\ii}}{2|\Sph^{d-1}|}\int_{\R^d}\frac{|\hat{g}(x)|}{|x|^{d-2}}\,dx<1.$$
  Then, $(w,g)$ is Penrose stable.
 \end{enumerate}

\end{proposition}

\begin{proof}
The proof of (1) is identical to \cite[Prop. 1]{LewSab-13b} so we omit it. To prove (2), we first treat the case $d=1$. Since $g$ is even, $\hat{g}$ is also even so that for $\xi\neq0$ we have
$$2\int_0^\ii e^{-t\tau}e^{-it\omega}\sin(t|\xi|^2)\hat{g}(2t\xi)\,dt = \int_0^\ii e^{-t\tau/(2|\xi|)}e^{-it\omega/(2|\xi|)}\frac{\sin(t|\xi|/2)}{|\xi|}\hat{g}(t)\,dt.$$
Notice that for $\xi=0$, the left side is equal to $0$. Defining
$$m(\tau,\omega,\xi):=\int_0^\ii e^{-t\tau}e^{-it\omega}\frac{\sin(t|\xi|)}{|\xi|}\hat{g}(t)\,dt,$$
we thus have to prove that
$$\inf_{(\tau,\omega,\xi)\in\R_+^*\times\R\times\R^*}|1+\hat{w}(2\xi)m(\tau,\omega,\xi)/2|>0.$$
Since $t\hat{g}(t)\in L^1(\R_+)$, $m$ is continuous on $\R_+\times\R\times\R$. Furthermore, due to the bounds
$$|m(\tau,\omega,\xi)|\le \int_0^\ii e^{-t\tau} t|\hat{g}(t)|\,dt,$$
$$|m(\tau,\omega,\xi)|\le \frac{1}{|\xi|}\int_0^\ii |\hat{g}(t)|\,dt,$$
$$|m(\tau,\omega,\xi)|\le\frac{C}{|\omega|}(1+\tau)\int_0^\ii(1+|t|)(|\hat{g}(t)|+|\hat{g}'(t)|)\,dt,$$
we have that $m(\tau,\omega,\xi)\to0$ as $\tau+|\omega|+|\xi|\to\ii$. As a consequence, there exists $R>0$ such that for all $\tau+|\omega|+|\xi|>R$ we have
$$|1+\frac12\hat{w}(2\xi)m(\tau,\omega,\xi)|\ge\frac12.$$
By continuity of $m$, it thus remains to show that for each $\tau+|\omega|+|\xi|\le R$, we have $1+\hat{w}(2\xi)m(\tau,\omega,\xi)/2\neq0$. To show this, notice that we may assume that $\xi$ is such that $\hat{w}(2\xi)\neq0$ (for otherwise $1+\hat{w}(\xi)m(\tau,\omega,\xi)/2=1\neq0$). Let us compute explicity the imaginary part of $m$ by undoing the Fourier transform of $g$: for any $\tau>0$ and $\xi\neq0$,
$$
m(\tau,\omega,\xi)=\frac{1}{2i|\xi|\sqrt{2\pi}}\int_\R g(r)\left(\frac{1}{\tau+i\omega-i|\xi|+ir}-\frac{1}{\tau+i\omega+i|\xi|+ir}\right)\,dr,
$$
so that
$$
\im m(\tau,\omega,\xi)=\frac{1}{2|\xi|\sqrt{2\pi}}\int_\R g(r)\left(\frac{\tau}{\tau^2+(\omega+|\xi|+r)^2}-\frac{\tau}{\tau^2+(\omega-|\xi|+r)^2}\right)\,dr.
$$
By integration by parts, we have
\begin{align*}
  \im m(\tau,\omega,\xi) &= -\frac{1}{2|\xi|\sqrt{2\pi}}\int_\R g'(r)\int_0^r\left(\frac{\tau}{\tau^2+(\omega+|\xi|+r')^2}-\frac{\tau}{\tau^2+(\omega-|\xi|+r')^2}\right)\,dr \\
  &= -\frac{1}{2|\xi|\sqrt{2\pi}}\int_\R g'(r)\left(\int_{|\xi|}^{r+|\xi|}-\int_{-|\xi|}^{r-|\xi|}\right)\frac{\tau}{\tau^2+(\omega+z)^2}\,dz\,dr.
\end{align*}
Since $g'$ is odd, we deduce that
\begin{align*}
  \im m(\tau,\omega,\xi) &= -\frac{1}{2|\xi|\sqrt{2\pi}}\int_0^\ii g'(r)\left(\int_{|\xi|}^{r+|\xi|}-\int_{-|\xi|}^{r-|\xi|}-\int_{|\xi|}^{-r+|\xi|}+\int_{-|\xi|}^{-r-|\xi|}\right)\frac{\tau}{\tau^2+(\omega+z)^2}\,dz\,dr \\
  &= -\frac{1}{2|\xi|\sqrt{2\pi}}\int_0^\ii g'(r)\left(\int_{-r+|\xi|}^{r+|\xi|}-\int_{-r-|\xi|}^{r-|\xi|}\right)\frac{\tau}{\tau^2+(\omega+z)^2}\,dz\,dr \\
  &= -\frac{1}{2|\xi|\sqrt{2\pi}}\int_0^\ii g'(r)\int_{-r+|\xi|}^{r+|\xi|}\left(\frac{\tau}{\tau^2+(\omega+z)^2}-\frac{\tau}{\tau^2+(\omega-z)^2}\right)\,dz\,dr \\
  &= -\frac{1}{2|\xi|\sqrt{2\pi}}\int_0^\ii g'(r)\int_{||\xi|-r|}^{r+|\xi|}\left(\frac{\tau}{\tau^2+(\omega+z)^2}-\frac{\tau}{\tau^2+(\omega-z)^2}\right)\,dz\,dr \\
  &= \frac{2}{|\xi|\sqrt{2\pi}}\int_0^\ii g'(r)\int_{||\xi|-r|}^{r+|\xi|}\frac{\tau\omega z}{(\tau^2+(\omega+z)^2)(\tau^2+(\omega-z)^2)}\,dz\,dr.
\end{align*}
From this relation and the fact that $g'<0$, we find that $\im m(\tau,\omega,\xi)\neq0$ (and thus $1+\hat{w}(2\xi)m(\tau,\omega,\xi)/2\neq0$) if $\xi\neq0$, $\tau>0$, and $\omega\neq0$. Let us now examine the regimes in which one of these parameters vanishes. Since $\im m(\tau,0,\xi)=0$, let us first assume that $\omega\neq0$. Since $\im m(\tau,\omega,\xi)$ is odd in $\omega$, we may also assume that $\omega>0$. First, when $\tau>0$ we thus have to examine $\xi=0$. We have
$$
\im m(\tau,\omega,0) = \frac{4}{\sqrt{2\pi}}\int_0^\ii g'(r)\frac{\tau\omega r}{(\tau^2+(\omega+r)^2)(\tau^2+(\omega-r)^2)}\,dr,
$$
which again does not vanish if $\omega>0$ and $\tau>0$. We now turn to the case $\tau=0$. To compute $\im m(0,\omega,\xi)$, we use that
$$
\int_{-r+|\xi|}^{r+|\xi|}\left(\frac{\tau}{\tau^2+(\omega+z)^2}-\frac{\tau}{\tau^2+(\omega-z)^2}\right)\,dz
=\int^{\tfrac{|\xi|+\omega+r}{\tau}}_{\tfrac{|\xi|+\omega-r}{\tau}}\frac{dz}{1+z^2}-\int^{\tfrac{|\xi|-\omega+r}{\tau}}_{\tfrac{|\xi|-\omega-r}{\tau}}\frac{dz}{1+z^2}
$$
which converges to $-\pi\1_{|\omega-|\xi||<r<\omega+|\xi|}$ almost everywhere in $r$, so that for all $\omega>0$ and $\xi\neq0$ we have
$$
\im m(0,\omega,\xi) = \frac{\sqrt{\pi}}{2|\xi|\sqrt{2}}\int_{|\omega-|\xi||}^{\omega+|\xi|} g'(r)\,dr
$$
which does not vanish since $g'<0$. When $\xi=0$, we find that
$$\im m(0,\omega,0) = \sqrt{\frac{\pi}{2}}g'(\omega)<0$$
for all $\omega>0$. We thus have proved that in the case $\omega>0$ and $\tau=0$, $\im(0,\omega,\xi)\neq0$ for all $\xi$. It thus remains to examine the case $\omega=0$, where the non-vanishing will come from the real part. Repeating the above computations but for the real part, we find
$$\re m(\tau,0,\xi) = -\frac{1}{2|\xi|\sqrt{2\pi}}\int_0^\ii g'(r)\int_{|r-|\xi||}^{r+|\xi|}\frac{z}{\tau^2+z^2}\,dz.$$
We deduce that $\re m(\tau,0,\xi)\ge0$
for all $\tau>0$ and $\xi\neq0$, and thus for all $\tau\ge0$ and $\xi\in\R$ by continuity. As a consequence,
$$1+\hat{w}(2\xi)\re m(\tau,0,\xi)/2 \ge 1-\frac{\|\hat{w}_-\|_{L^\ii}}{2}\int_0^\ii t|\hat{g}(t)|\,dt>0
$$
for all $\tau\ge0$ and $\xi\in\R$. We have finished the proof of (2) in the case $d=1$. When $d\ge2$, we use the idea, already used in \cite[Proof of Prop. 2.1]{MouVil-11}, that since $g$ is radial we have $\hat{g}(2t\xi)=\hat{\phi}(2t|\xi|)$ where $\phi:\R\to\R$ is defined by
$$\phi(r)=\int_{\R^{d-1}}g(\sqrt{r^2+|x'|^2})\,dx'.$$
Notice that $\phi$ is an even function, and that $\phi\in L^1(\R)$ if $g\in L^1(\R^d)$. If $g$ is radially decreasing, then $\phi$ decreases for $r>0$. This last fact is true even if $g$ is not radially decreasing in $d\ge3$, as explained in \cite[Remark  2.2]{MouVil-11}. Hence, we may apply the case $d=1$ to the function $\phi$, leading to the result in dimensions $d\ge2$.

\end{proof}

\begin{remark}
Another interesting condition ensuring the Penrose stability of $(w,g)$ is given by \cite[Prop. 1.6]{Hadama-23b}. It allows to consider $g(\xi)=\1(|\xi|^2\le\mu)$ (the Free fermi sea), and has no analogue for the Vlasov equation. For this choice of $g$, notice that we have $|\hat{g}(x)|\lesssim\langle x\rangle^{-(d+1)/2}$ so that \eqref{eq:penrose-bounded} fails if $d\le3$. The condition \eqref{eq:penrose-bounded} is needed to ensure that the integral in \eqref{eq:penrose-integral} is well-defined, and to ensure that $\cL$ is bounded on $L^2_t H^s_x$. When \eqref{eq:penrose-bounded} is not satisfied, one can still give a meaning to the integral in \eqref{eq:penrose-integral} as a semi-convergent integral if $g\in L^1(\R^d)$, for instance by the formula
$$2\int_0^\ii e^{-t\tau}e^{-it\omega}\sin(t|\xi|^2)\hat{g}(2t\xi)\,dt
=\frac{1}{(2\pi)^{d/2}}\int_{\R^d}\frac{g(\eta+\tfrac{\xi}{2})-g(\eta-\tfrac{\xi}{2})}{2\eta\cdot\xi+\omega-i\lambda}\,d\eta.$$
However, when \eqref{eq:penrose-bounded} is not satisfied, it is unclear whether $\cL$ is bounded on $L^2_t H^s_x$ or not. Under the Penrose condition, the operator $(1+\cL)^{-1}$ still makes sense as a bounded operator on $L^2_t H^s_x$ and one can still define the function $\Phi$ to perform the fixed point. In the case $g(\xi)=\1(|\xi|^2\le\mu)$ and $d=3$, $\cL$ is bounded from $L^2_t L^2_x$ to $L^2_t \dot{H}^{1}_x$ by the bound
$$\left|\int_0^\ii e^{-t\tau}e^{-it\omega}\sin(t|\xi|^2)\hat{g}(2t\xi)\,dt\right|\lesssim \int_0^\ii (t|\xi|^2)^\theta|\hat{g}(2t\xi)|\,dt\lesssim\frac{1}{|\xi|^{1-\theta}}\sup_{\omega\in\Sph^{d-1}}\int_0^\ii t^\theta |\hat{g}(2t\omega)|\,dt,
$$
for any $\theta\in[0,1]$. One also has $(1+\cL)(1+\cL)^{-1}\varrho=\varrho$ for any $\varrho$ in $L^2_t H^{1/2}_x$ (as an identity in $L^2_t \dot{H}^{1}_x$), which is enough for our applications (see the proof of Theorem \ref{thm:main1} below).
\end{remark}

\subsection{Strichartz estimates for the free Laplacian}\label{sec:free}

In this section we recall the Strichartz estimates for $e^{it\Delta_x}$ that we will use throughout the article. We also extend them slightly.  A first important tool will be \cite[Thm. 1.5]{BezHonLeeNakSaw-19}, which we cite in the version of \cite[Thm. 3.5]{Hadama-23}.

\begin{proposition}\label{prop:bez}

 Let $d\ge1$, $\sigma\ge0$ and $p,q,\alpha\in(1,+\ii)$ be such that $2/p+d/q=d-2\sigma$, $1/\alpha\ge 1/(dp)+1/q$ and $\alpha<p$. Then, there exists $C>0$ such that for all $\gamma\in\cH^{\sigma,\alpha}$ we have
 $$\| \rho_{e^{it\Delta_x}\gamma e^{-it\Delta_x}} \|_{L^p_t L^q_x} \le C\| \sd^\sigma \gamma \sd^\sigma\|_{\gS^\alpha(L^2_x)}.$$

\end{proposition}

Proposition \ref{prop:bez} was first proved in the case $\sigma=0$ and $1\le q<(d+1)/(d-1)$ in \cite{FraLewLieSei-13,FraSab-17}.

\begin{remark}\label{rk:stri-dual}
 Proposition \ref{prop:bez} has a dual formulation that was first discovered in \cite{FraLewLieSei-13}: there exists $C>0$ such that for all $V\in L^{p'}_t L^{q'}_x$ we have
 $$\left\| \sd^{-\sigma}\int_\R e^{-it\Delta_x}V(t,x)e^{it\Delta_x}\,dt\sd^{-\sigma}\right\|_{\gS^{\alpha'}(L^2_x)}\le C\|V\|_{L^{p'}_t L^{q'}_x},
 $$
 or equivalently (taking $V=|W|^2$)
 \begin{equation}\label{eq:stri-bez-xtotx}
    \|W(t,x)e^{it\Delta_x}\sd^{-\sigma}\|_{\gS^{2\alpha'}(L^2_x\to L^2_{t,x})} \le \sqrt{C}\|W\|_{L^{2p'}_t L^{2q'}_x}.
 \end{equation}
 These estimates all have a version which is localized on a time interval $I$: an estimation of $\rho$ in $L^p_t(I)$  is equivalent to Schatten  estimates with $V\in L^{p'}_t(I)$ or $W\in L^{2p'}_t(I)$ (in this last case, $L^2_{t,x}$ is also replaced by $L^2_t(I,L^2_x)$). We will use heavily these various formulations throughout the article.
\end{remark}

\begin{remark}
 In the case $\sigma=0$, $\alpha=1$, and $\gamma=|u\rangle\langle u|$ for $u\in L^2_x$, Proposition \ref{prop:UV-bez} reduces to the (non-endpoint) standard Strichartz estimates
 \begin{equation}\label{eq:stri-hom}
    \|e^{it\Delta_x}u\|_{L^{2p}_t L^{2q}_x} \lesssim \|u\|_{L^2_x}
 \end{equation}
 which go back to Strichartz \cite{Strichartz-77} and Ginibre-Velo \cite{GinVel-92}. We also recall the associated inhomogeneous Strichartz estimates
 \begin{equation}\label{eq:stri-inhom}
  \left\|\int_0^t e^{i(t-\tau)\Delta_x}F(\tau)\,d\tau\right\|_{L^{2p}_t L^{2q}_x}\lesssim \|F\|_{L^{(2\tilde{p})'}_t L^{(2\tilde{q})'}_x}
 \end{equation}
 if $(p,\tilde{p},q,\tilde{q})\in(1,+\ii)$ are such that $2/p+d/q=d=2/\tilde{p}+d/\tilde{q}$.
\end{remark}

We also recall a recent result of Chen, Hong and Pavlovic which will be useful to us in the proof of Proposition \ref{prop:rho22}. We provide an alternative proof of this result in Appendix \ref{app:CHP}.

\begin{theorem}[Thm. 3.1 in \cite{CheHonPav-17}]\label{thm:strichartz-HS}
  Let $d\ge1$ and $\alpha_0,\alpha_1,\alpha_2\ge0$ be such that $\alpha_1+\alpha_2>(d-1)/2$ and
   $$
  \begin{cases}
   \alpha_0=\alpha_1+\alpha_2-\tfrac{d-1}{2} &\text{if}\ \max(\alpha_1,\alpha_2)<\tfrac{d-1}{2},\\
   \alpha_0<\min(\alpha_1,\alpha_2) &\text{if}\ \max(\alpha_1,\alpha_2)=\tfrac{d-1}{2},\\
   \alpha_0=\min(\alpha_1,\alpha_2) &\text{if}\ \max(\alpha_1,\alpha_2)>\tfrac{d-1}{2}.\\
  \end{cases}
  $$
  Then, there exists $C>0$ such that for all $V\in L^2_t H^{-\alpha_0}_x$ we have
  $$\left\|\sd^{-\alpha_1}\int_\R dt\,e^{-it\Delta_x} (|\nabla_x|^{1/2}V)(t,x) e^{it\Delta_x} \sd^{-\alpha_2}\right\|_{\gS^2} \le C\| \sd^{-\alpha_0}V\|_{L^2_{t,x}}.$$
 \end{theorem}

We extend Proposition \ref{prop:bez} by allowing different powers of $\sd$ on the right side, which will be quite important for us as we explained in the introduction.

\begin{corollary}\label{coro:deriv-distrib}
 Let $d\ge1$ and  $p,q,\alpha\in(1,+\ii)$ be such that $\alpha<p$, $1/\alpha\ge1/(dp)+1/q$. Let $\sigma_1,\sigma_2\ge0$ be such that $\sigma_1+\sigma_2=d-2/p-d/q$ and
 $$\sigma_1,\sigma_2<\min(d/2,d/2-2/p+1)$$
Then, there exists $C>0$ such that for all $\gamma\in\cH^{s,\alpha}$ we have
 $$\| \rho_{e^{it\Delta_x} \gamma  e^{-it\Delta_x}} \|_{L^p_t L^{q}_x} \le C\| \sd^{\sigma_1} \gamma \sd^{\sigma_2}\|_{\gS^\alpha}.$$
\end{corollary}

\begin{remark}
In the special case $\sigma_1=\sigma_2=\sigma$, we recover Proposition \ref{prop:bez} in the sense that the condition $\sigma<\min(d/2,d/2-2/p+1)$ is a consequence of the relation $2\sigma=d-2/p-d/q$.
\end{remark}

\begin{proof}

The dual version of the inequality is
 $$\left\| \int_\R \langle\nabla\rangle^{-\sigma_2} e^{-it\Delta_x} W(t)e^{it\Delta_x}\langle\nabla\rangle^{-\sigma_1}\right\|_{\gS^{\alpha'}(L^2_x)} \le C\| W\|_{L^{p'}_t L^{q'}_x}.$$
 We show that under the above assumptions on $p,q,\alpha,\sigma_1,\sigma_2$, one can then find $p_1,p_2,q_1,q_2,\alpha_1,\alpha_2\in(1,+\ii)$ such that $2/p_j+d/q_j=d-2\sigma_j$,
 $1/\alpha_j\ge1/(dp_j)+1/q_j$, $\alpha_j<p_j$, and $1/\alpha=1/(2\alpha_1)+1/(2\alpha_2)$, $1/p=1/(2p_1)+1/(2p_2)$. Once the existence of these exponents is obtained, we write $W=W_2W_1$ with
 $$W_2(t,x) = \frac{W(t,x)^{q'/(2q_2')}}{\|W(t,\cdot)\|_{L^{q'}_x}^{q'/(2q_2')-p'/(2p_2')}},\ W_1(t,x) = \frac{|W(t,x)|^{q'/(2q_1')}}{\|W(t,\cdot)\|_{L^{q'}_x}^{q'/(2q_1')-p'/(2p_1')}},$$
 so that by the dual version of Proposition \ref{prop:bez} we have
 $$\|W_j e^{it\Delta_x} \langle\nabla\rangle^{-\sigma_j} \|_{\gS^{2\alpha_j'}(L^2_x\to L^2_{t,x})} \le C \|W_j\|_{L^{2p_j'}_t L^{2q_j'}_x}=C\|W\|_{L^{p'}_t L^{q'}_x}^{p'/(2p_j')}.$$
 We then conclude by H\"older's inequality in Schatten spaces. To show the existence of such parameters, we state all conditions as a condition on $p_1$. Hence, let $p_1\in(1,+\ii)$. Then, there exists $q_1\in(1,+\ii)$ such that $2/p_1+d/q_1=d-2\sigma_1$ if and only if $1/p_1<d/2-\sigma_1$, which is compatible with  $1/p_1>0$ (meaning that the set of $p_1$ satisfying these conditions is non-empty) since $\sigma_1<d/2$. Next, we define $p_2$ by the relation $1/p_2=2/p-1/p_1$. We have $p_2\in(1,+\ii)$ if and only if $2/p-1<1/p_1<2/p$. This is compatible with $0<1/p_1<1$ and $1/p_1<d/2-\sigma_1$ if $\sigma_1<d/2-2/p+1$. Next, we define $q_2$ by the relation $2/p_2+d/q_2=d-2\sigma_2$. We have $q_2\in(1,+\ii)$ if and only if
 $$2/p+\sigma_2-d/2<1/p_1<2/p+\sigma_2.$$
 The second inequality holds because $1/p_1<2/p$ and $\sigma_2\ge0$. The first
 inequality is compatible with $1/p_1<1$, $1/p_1<d/2-\sigma_1$, and $1/p_1<2/p$ since $\sigma_2<\min(d/2,d/2-2/p+1)$ and since $\sigma_1+\sigma_2<d-2/p$, which we can infer from the relation $\sigma_1+\sigma_2=d-2/p-d/q$. To sum up, under the conditions
 $$\max(0,2/p-1,2/p+\sigma_2-d/2)<1/p_1<\min(1,2/p,d/2-\sigma_1),$$
 (which is a non-empty set of $1/p_1$), we have built $q_1,p_2,q_2\in(1,+\ii)$ such that $2/p_j+d/q_j=d-2\sigma_j$. Let us now construct the $\alpha_j$. For this we distinguish two cases. First, if $1/p\ge(d-s)/(d+1)$ (where we used the notation $s=\sigma_1+\sigma_2$), the condition $1/\alpha\ge1/(dp)+1/q$ is implied by $1/\alpha>1/p$. Hence, let us find $\alpha_j$ such that $1/\alpha=1/(2\alpha_1)+1/(2\alpha_2)$ and $1/\alpha_j>1/p_j$ with the additional condition that $1/p_j\ge(d-2\sigma_j)/(d+1)$ (so that $1/\alpha_j>1/p_j$ implies $1/\alpha_j\ge1/(dp_j)+1/q_j$). The condition $1/p_1\ge(d-2\sigma_1)/(d+1)$ is compatible with $1/p_1<\min(1,2/p,d/2-\sigma_1)$ since $\sigma_1,\sigma_2<d/2$ and since $1/p\ge(d-s)/(d+1)$. We then let $\alpha_1\in(1,+\ii)$ such that $1/\alpha_1$. Next, we define $\alpha_2$ by the relation $1/\alpha_2=2/\alpha-1/\alpha_1$. We want $1/\alpha_2>1/p_2$ where we recall that $1/p_2=2/p-1/p_1$, as well as $1/p_2\ge(d-2\sigma_2)/(d+1)$. This last relation is equivalent to $1/p_1\le2/p-(d-2\sigma_2)/(d+1)$, which is compatible with $\max(0,2/p-1,2/p+\sigma_2-d/2)<1/p_1$ since $\sigma_1,\sigma_2<d/2$ and since $1/p\ge(d-s)/(d+1)$. Finally, $1/\alpha_2>1/p_2$ is equivalent to $1/\alpha_1<2/\alpha-2/p+1/p_1$, which is compatible with $1/\alpha_1>1/p_1$ since $1/\alpha>1/p$. We thus have built the $\alpha_j$, which concludes this case. In the second case $1/p<(d-s)/(d+1)$, the condition $1/\alpha\ge1/(dp)+1/q$ implies $1/\alpha>1/p$ so we again build $\alpha_j$ such that $1/\alpha=1/(2\alpha_1)+1/(2\alpha_2)$ and $1/\alpha_j\ge1/(dp_j)+1/q_j$. Under the additional requirement that $1/p_j<(d-2\sigma_j)/(d+1)$, this last inequality implies that $1/\alpha_j>1/p_j$. We argue as in the first case to show that one can always build such $\alpha_j$'s under our conditions.

\end{proof}

\begin{remark}\label{rk:integer-derivatives}
 A first consequence of Corollary \ref{coro:deriv-distrib} is to give Strichartz estimates in (integer) Sobolev spaces for operator densities. Indeed, given $n\in\N$, by the Leibniz rule for differentiation, to estimate $\rho_\gamma$ in $W^{n,q}_x$ it is enough to estimate $\rho_{\partial^\alpha\gamma\partial^\beta}$ in $L^q_x$ for all multi-indices $\alpha,\beta$ with $|\alpha|+|\beta|\le n$. Applying Corollary \ref{coro:deriv-distrib} to estimate $\rho_{e^{it\Delta_x}\partial^\alpha\gamma\partial^\beta e^{-it\Delta_x}}$ with $\sigma_1=n-|\alpha|$ and $\sigma_2=|\alpha|$ (under the assumption that $n<\min(d/2,d/2-2/p+1)$), we get that
 $$\|\rho_{e^{it\Delta_x}\gamma e^{-it\Delta_x}}\|_{L^p_t W^{n,q}_x} \lesssim \|\sd^n\gamma\sd^n\|_{\gS^\alpha}$$
 if $2/p+d/q=d-n$, $1/\alpha\ge1/(dp)+1/q$, $\alpha<p$. In Theorem \ref{thm:stri-potential} below, we will extend this inequality to non-integer $n$ (we will also prove it for $U_V(t,0)$ instead of $e^{it\Delta_x}$), using fractional Leibniz rules. This will be one of the key results of this work, since by the fixed point strategy above we need to estimate $\rho_{U_V(t,0)Q_{\text{in}}U_V(t,0)^*}$ in $L^2_t H^s_x$.
\end{remark}

\section{Generalized Strichartz estimates}\label{sec:gene-stri}

As we mentioned in Remark \ref{rk:stri-dual}, the standard Strichartz estimates for orthonormal systems provides Schatten bounds for operators of the type $W(t,x)e^{it\Delta_x}\sd^{-\sigma}:L^2_x\to L^2_{t,x}$, where $W(t,x)$ is a multiplication operator. Below we provide a slightly more general version of these results.

\begin{lemma}\label{lem:strichartz-A}
 Let $d\ge1$, $p,q\in(1,+\ii)$, and $\alpha\in(1,+\ii)$. Assume that $(\cU(t))_{t\in\R}$ is a family of operators on $L^2_x$ such that for any $W\in L^{2p'}_tL^{2q'}_x$ we have
 $$\|W(t)\cU(t)\|_{\gS^{2\alpha'}(L^2_x\to L^2_{t,x})} \lesssim \|W\|_{L^{2p'}_tL^{2q'}_x}.$$
 Then, for any family $(A(t))_{t\in\R}$ of linear operators from $L^{2q}_x$ to $L^2(X)$ for some measure space $X$, we have
 $$\| A(t)\cU(t)\|_{\gS^{2\alpha'}(L^2_x\to L^2_tL^2(X))}^2 \lesssim \left(\int_\R\|A(t)\|_{L^{2q}_x\to L^2(X)}^{2p'}\,dt\right)^{\tfrac{1}{p'}}.$$
\end{lemma}

\begin{proof}
 By Remark \ref{rk:stri-dual}, the assumption on $\cU(t)$ is equivalent to the estimate
 $$\|\rho_{\cU(t)\gamma\cU(t)^*}\|_{L^p_tL^q_x}\lesssim\|\gamma\|_{\gS^\alpha(L^2_x)}$$
 for all $\gamma\in\gS^\alpha(L^2_x)$. We have to estimate
 $$\left\|\int_\R \cU(t)^*A(t)^* A(t)\cU(t)\,dt \right\|_{\gS^{\alpha'}(L^2_x)}.$$
 By duality (notice that the operator we have to estimate is self-adjoint), we test it against a self-adjoint $\Gamma\in\gS^{\alpha}(L^2_x)$ which has the decomposition $\Gamma=\sum_k\lambda_k|u_k\rangle\langle u_k|$, we get
 $$
 \begin{multlined}
     \tr_{L^2_x}\int_\R \cU(t)^*A(t)^* A(t)\cU(t)\,dt\Gamma \\
     = \int_\R \langle A(t)\cU(t)\sqrt{\lambda_k}u_k,A(t)\cU(t)\sqrt{|\lambda_k|}u_k\rangle_{L^2(X,\ell^2_k)}\,dt.
 \end{multlined}
 $$
 As a consequence, by the Hölder inequality we get
 $$
 \begin{multlined}
    |\tr_{L^2_x}\int_\R \cU(t)^*A(t)^* A(t)\cU(t)\,dt\Gamma|\\
    \le \|A(t)\cU(t)\sqrt{\lambda_k}u_k\|_{L^2(\R_t\times X,\ell^2_k)}\|A(t)\cU(t)\sqrt{|\lambda_k|}u_k\|_{L^2(\R_t\times X,\ell^2_k)}.
 \end{multlined}
 $$
 Now we write
 $$
 \|A(t)\cU(t)\sqrt{\lambda_k}u_k\|_{L^2(\R_t\times X,\ell^2_k)} \le \|A(t)\|_{L^{2p}_tL^{2q}_{x}\ell^2_k \to L^2(\R_t\times X,\ell^2_k)}\| \cU(t)\sqrt{\lambda_k}u_k \|_{L^{2p}_tL^{2q}_{x}\ell^2_k},
 $$
 and we know by assumption that
 \begin{equation}\label{eq:app-bez}
    \| \cU(t)\sqrt{\lambda_k}u_k \|_{L^{2p}_tL^{2q}_{x}\ell^2_k}^2\lesssim \|\lambda_k\|_{\ell^{\alpha}_k}=\|\Gamma\|_{\gS^\alpha(L^2_x)}.
 \end{equation}
 As a consequence, we have proved by duality that
 $$\| A(t)\cU(t)\|_{\gS^{2\alpha'}(L^2_x\to L^2(\R_t\times X))} \lesssim \|A(t)\|_{L^{2p}_tL^{2q}_{x}\ell^2_k \to L^2(\R_t\times X,\ell^2_k)}.$$
Using \cite[Theorem 5.5.1]{Grafakos-book} (which is stated for $p=q$, but the proof also works for $p\neq q$), we have
$$\|A(t)\|_{L^{2p}_tL^{2q}_{x}\ell^2_k \to L^2(\R_t\times X,\ell^2_k)} \lesssim \|A(t)\|_{L^{2p}_tL^{2q}_{x} \to L^2(\R_t\times X)}.$$
Finally, recalling that $A(t)$ acts by multiplication in $t$, we have for any $f\in L^{2p}_tL^{2q}_{x}$,
\begin{align*}
  \|A(t)f\|_{L^2(\R_t\times X)}^2 &=
  \int_\R \|A(t)f(t)\|_{L^2(X)}^2\,dt \\
  &\le \int_\R \|A(t)\|^2_{L^{2q}_x\to L^2(X)} \|f(t)\|_{L^{2q}_x}^2\,dt\\
  &\le \left(\int_\R\|A(t)\|_{L^{2q}_x\to L^2(X)}^{2p'}\,dt\right)^{\tfrac{1}{p'}}\|f\|_{L^{2p}_tL^{2q}_{x}}^2,
\end{align*}
which concludes the proof.
\end{proof}

Combining Lemma \ref{lem:strichartz-A} (applied to $\cU(t):=e^{it\Delta_x}\sd^{-\sigma}$) with Proposition \ref{prop:bez}, we obtain:

\begin{corollary}\label{coro:general-schatten}
 Let $d\ge1$, $\sigma\ge0$, and $p,q,\alpha\in(1,+\ii)$ be such that
 $$
 \frac{2}{p}+\frac dq = d-2\sigma,\quad \frac{1}{\alpha}\ge\frac{1}{dp}+\frac1q,\quad\alpha<p.
 $$
 Let $(A(t))_{t\in\R}$ be a family of linear operators from $L^{2q}_x$ to $L^2(X)$ for some measure space $X$. Then, we have
 $$\| A(t)e^{it\Delta_x}\langle\nabla_x\rangle^{-\sigma}\|_{\gS^{2\alpha'}(L^2_x\to L^2_tL^2(X))}^2 \lesssim \left(\int_\R\|A(t)\|_{L^{2q}_x\to L^2(X)}^{2p'}\,dt\right)^{\tfrac{1}{p'}}.$$
\end{corollary}

\section{Fractional Leibniz rules}\label{sec:leibniz}

A key point of our proof is to distribute fractional derivatives. A way to do it is to use fractional Leibniz rules. While there are many references on the topic, we follow the approach of \cite{Li-19}. One of the versions of fractional Leibniz rules is the following estimate \cite[Eq. (1.15)]{Li-19}:
 \begin{equation}\label{eq:leibniz-standard}
    \||D|^\sigma(Wf)\|_{L^q_x}\lesssim \||D|^\sigma W\|_{L^{q_1}_x}\|f\|_{L^{q_2}_x}+\|W\|_{L^{r_1}_x}\||D|^\sigma f\|_{L^{r_2}_x},
 \end{equation}
 if $\sigma\ge0$ and $q_1,q_2,r_1,r_2\in(1,+\ii)$ are such that $1/q_1+1/q_2=1/q=1/r_1+1/r_2$. If we furthermore have $1/q_1-\sigma/d\le 1/r_1\le 1/q_1$, then $W^{\sigma,q_1}\hookrightarrow L^{r_1}$ and $W^{\sigma,r_2}\hookrightarrow L^{q_2}$ so that \eqref{eq:leibniz-standard} implies
 \begin{equation}\label{eq:leibniz-standard-2}
    \|Wf\|_{W^{\sigma,q}_x}\lesssim \|W\|_{W^{\sigma,q_1}_x}\|f\|_{W^{\sigma,r_2}_x},
 \end{equation}
 an estimate which can be obtained directly by interpolation from the case $\sigma\in\N$ and standard Leibniz rules. In applications to nonlinear Schrödinger equations, \eqref{eq:leibniz-standard-2} is often enough and thus precise fractional Leibniz rules such as \eqref{eq:leibniz-standard} or the results of \cite{Li-19} are not needed. For us, estimates of the type \eqref{eq:leibniz-standard-2} will not be enough (see for instance the proof of Proposition \ref{prop:UV-bez} below) and we need the full version of fractional Leibniz rules.

 Another way to state the estimate \eqref{eq:leibniz-standard} is that for any $W\in \dot{W}^{\sigma,q_1}\cap L^{r_1}$, the operator $f\mapsto |D|^\sigma(Wf)$ is bounded from $\dot{W}^{\sigma,r_2}\cap L^{q_2}$ to $L^q$. In Lemma \ref{lem:leibniz-op-general} below, we show that this operator can be written as a sum of terms that are each factorized. Our results do not provide an essential novelty compared to \cite{Li-19} but rather another way to write them, which is useful for our purposes. Hence, we first recall the notations of \cite{Li-19} before stating our results.

\subsection{Notations and key results}

Let us start with recalling facts about Littlewood-Paley multipliers. Let $\phi_0:\R^d\to[0,1]$ be a smooth function such that $\phi_0\equiv1$ on $B(0,1)$ and $\phi_0\equiv0$ away from $B(0,\tfrac76)$. Defining $\phi(\xi)=\phi_0(\xi)-\phi_0(2\xi)$ and for $j\in\Z$ the Littlewood-Paley multiplier for any $f\in\cS(\R^d)$
$$\hat{P_j f}(\xi) := \phi(\tfrac{\xi}{2^j})\hat{f}(\xi),$$
we have $\sum_{j\in\Z} P_j=1$ while $\phi$ is supported on $\{\tfrac12\le|\xi|\le\tfrac76 \}$. We also have $P_j P_{j'}=0$ if $|j-j'|>1$. Let us also define
$$\hat{P_{\le j}f}(\xi):=\phi_0(\tfrac{\xi}{2^j})\hat{f}(\xi).$$
We also use the shortcut notations
$$f_j:=P_j f,\quad \tilde{g_j}:=\tilde{P_j}g,\quad \tilde{P_j}:= P_{j-1}+P_j+P_{j+1}.$$
For any $\sigma\in\R$, we denote by $|D|^\sigma$ the Fourier multiplier by $\xi\in\R^d\mapsto|\xi|^\sigma$. We use this notation rather than $|\nabla_x|^\sigma$ that we use in other sections to stick with the notations of~ \cite{Li-19}.  Similarly, for any $\alpha\in\N^d$, $D^{\sigma ,\alpha}$ is defined as the Fourier multiplier by the function $(\tfrac1i\partial_\xi)^\alpha(|\xi|^\sigma)$. Following the argument of \cite[Sec. 4]{Li-19}, we get the following decomposition:

\begin{lemma}\label{lem:bony-op}

Let $d\ge1$, $W\in\cS(\R^d)$, and $\sigma>0$. Then, for any $\sigma_1,\sigma_2\ge0$ such that $\sigma_1+\sigma_2=\sigma$ we have, as an identity between operators mapping $\cS(\R^d)$ to $L^2(\R^d)$ (see Remark \ref{rk-meaning-leibniz-op} below),

 \begin{equation}\label{eq:bony-op}
\begin{multlined}
|D|^\sigma W = \sum_{|\alpha|\le\lfloor \sigma_1\rfloor}\frac{1}{\alpha!}(\partial^\alpha W) D^{\sigma ,\alpha}+\sum_{|\beta|\le\lfloor \sigma_2\rfloor}\frac{1}{\beta!}(D^{\sigma ,\beta} W)\partial^\beta
+ \sum_{j\in\Z}|D|^\sigma W_j \tilde{P_j}\\
- \sum_{|\alpha|\le\lfloor \sigma_1\rfloor}\frac{1}{\alpha!}\sum_{j\in\Z}(\partial^\alpha W_j) D^{\sigma ,\alpha}\tilde{P_j}
-\sum_{|\alpha|\le\lfloor \sigma_1\rfloor}\frac{1}{\alpha!}\sum_{j\in\Z}(\partial^\alpha W_j) D^{\sigma ,\alpha}P_{\le j-2}\\
- \sum_{|\beta|\le\lfloor \sigma_2\rfloor}\frac{1}{\beta!}\sum_{j\in\Z}(D^{\sigma,\beta}\tilde{P_j}W)P_j\partial^\beta
-\sum_{|\beta|\le\lfloor \sigma_2\rfloor}\frac{1}{\beta!}\sum_{j\in\Z}(D^{\sigma,\beta}P_{\le j-2}W)P_j\partial^\beta+\cA_W^{(\sigma_1)}+\cB_W^{(\sigma_2)}
\end{multlined}
\end{equation}
with the notation
$$
\begin{multlined}
  (\cA_W^{(\sigma_1)}g)(x) = \sum_{j\in\Z} 2^{2jd}\int \Phi_\cA(2^jy,2^jy')\times\\
  \times\sum_{|\alpha|=\lfloor\sigma_1\rfloor+1}C_\alpha (2^jy)^{\alpha} 2^{j(\sigma-\lfloor\sigma_1\rfloor-1)}(\partial^\alpha P_{\le j-2}W)(x-y')(P_jg)(x-y)\,dy\,dy',
\end{multlined}
$$
$$
\begin{multlined}
  (\cB_W^{(\sigma_2)}g)(x) = \sum_{j\in\Z} 2^{2jd}\int \Phi_\cB(2^jy,2^jy')\times\\
  \times\sum_{|\beta|=\lfloor\sigma_2\rfloor+1}C_\alpha (2^jy)^{\beta} 2^{j(\sigma-\lfloor\sigma_2\rfloor-1)}(P_jW)(x-y')(\partial^\beta P_{\le j-2}g)(x-y)\,dy\,dy',
\end{multlined}
$$
for any $g\in\cS(\R^d)$, where $\Phi_\cA,\Phi_\cB$ are Schwartz functions and $C_\alpha$ are some real constants.

\end{lemma}

\begin{remark}\label{rk-meaning-leibniz-op}
 To emphasize the meaning of the identity \eqref{eq:bony-op}, $|D|^\sigma W$ is the operator $f\in\cS(\R^d)\mapsto |D|^\sigma (Wf)$ (we first take the product of $W$ and $f$, and then take the fractional derivative). On the right side, terms like $(\partial^\alpha W)D^{\sigma,\alpha}$ are interpreted as the operator $f\in\cS(\R^d)\mapsto (\partial^\alpha W)(D^{\sigma,\alpha}f)$. As a general rule, we make a difference between the operators $|D|^\sigma W$ and $(|D|^\sigma W)$: while we already defined the former, the latter is defined as the multiplication operator by the function $|D|^\sigma W$.
\end{remark}

\begin{remark}
 In \cite{Li-19}, only the first two sums in the right side of \eqref{eq:bony-op} are kept and the others are considered as a remainder term which is estimated for instance in \cite[Thm. 1.2(1)]{Li-19}. For our purposes, we will need to keep the precise form of these terms and this is why we state Lemma \ref{lem:bony-op} in this way.
\end{remark}

We now provide a similar decomposition for $|D|^\sigma\rho_\gamma$ where $\gamma$ is an operator on $L^2_x$, which is derived in the same way as \eqref{eq:bony-op}. Indeed, the starting point to derive \eqref{eq:bony-op} is the Bony decomposition valid for any $W,f\in\cS(\R^d)$:
$$Wf=\sum_{j\in\Z}W_j \tilde{f_j}+\sum_{j\in\Z} W_j(P_{\le j-2}f)+\sum_{j\in\Z}(P_{\le j-2}W)f_j,$$
which just follows from writing $W=\sum_j W_j$ and $f=\sum_j f_j$ and regrouping terms. For density matrices $\gamma$ that have an integral kernel $\gamma(x,y)$, we can similarly write
$$\gamma(x,x)=\sum_{j,k\in\Z} ((P_j)_x (P_k)_y\gamma)(x,x),$$
where $(P_j)_x$ acts in the $x$ variable and $(P_k)_y$ acts in the $y$ variable. Regrouping terms similarly as in the Bony decomposition, we obtain the following result.

\begin{lemma}\label{lem:bony-density-matrix-general}
 Let $d\ge1$, $\gamma$ an operator on $L^2_x$ with integral kernel $\gamma(x,y)\in\cS_{x,y}(\R^d\times\R^d)$, and $\sigma>0$. Then, for any $\sigma_1,\sigma_2\ge0$ such that $\sigma_1+\sigma_2=\sigma$ we have,
\begin{equation}\label{eq:bony-density-matrix-general}
 \begin{multlined}
 |D|^\sigma \rho_\gamma=\sum_{|\alpha|\le \sigma_1}\frac{(-1)^{|\alpha|}}{\alpha!}\rho_{\partial^\alpha\gamma D^{\sigma ,\alpha}}+\sum_{|\beta|\le \sigma_2}\frac{(-1)^{|\beta|}}{\beta!}\rho_{ D^{\sigma ,\beta}\gamma\partial^\beta}+\sum_{j\in\Z}|D|^\sigma \rho_{P_j\gamma\tilde{P_j}}\\
 -\sum_{|\alpha|\le \sigma_1}\frac{(-1)^{|\alpha|}}{\alpha!}\sum_{j\in\Z}\rho_{P_j \partial^\alpha \gamma D^{\sigma ,\alpha}\tilde{P_j}}-\sum_{|\alpha|\le \sigma_1}\frac{(-1)^{|\alpha|}}{\alpha!}\sum_{j\in\Z} \rho_{P_j\partial^\alpha\gamma D^{\sigma ,\alpha}P_{\le j-2}}+\cA_\gamma^{(\sigma_1)}\\
-\sum_{|\beta|\le \sigma_2}\frac{(-1)^{|\beta|}}{\beta!}\sum_{j\in\Z}\rho_{ D^{\sigma ,\beta}\tilde{P_j}\gamma P_j \partial^\beta}-\sum_{|\beta|\le \sigma_2}\frac{(-1)^{|\beta|}}{\beta!}\sum_{j\in\Z} \rho_{ D^{\sigma ,\beta}P_{\le j-2}\gamma P_j\partial^\beta}+\cB_\gamma^{(\sigma_2)},
\end{multlined}
\end{equation}
where $\sigma_1,\sigma_2\ge0$ verify $\sigma_1+\sigma_2=\sigma$ and where
$$
\cA_\gamma^{(\sigma_1)}(x)=\sum_{j\in\Z}\sum_{|\alpha|=\lfloor \sigma_1\rfloor+1} 2^{2jd}\int\Phi_\alpha(2^j y, 2^j y')2^{j(\sigma-\lfloor \sigma_1\rfloor-1)}(\partial^\alpha P_{\le j-2}\gamma P_j)(x-y',x-y)\,dy\,dy',
$$
$$
\cB_\gamma^{(\sigma_2)}(x)=\sum_{j\in\Z}\sum_{|\beta|=\lfloor \sigma_2\rfloor+1} 2^{2jd}\int\Phi_\beta(2^j y, 2^j y')2^{j(\sigma-\lfloor \sigma_2\rfloor-1)}(P_j\gamma \partial^\beta P_{\le j-2} )(x-y',x-y)\,dy\,dy'.
$$
\end{lemma}

\begin{remark}
 We may rewrite
$$\cA_\gamma^{(\sigma_1)}=\sum_{j\in\Z}\sum_{|\alpha|=\lfloor \sigma_1\rfloor+1} 2^{2jd}\int\Phi_\alpha(2^j y, 2^j y')2^{j(\sigma-\lfloor \sigma_1\rfloor-1)}\rho_{\partial^\alpha P_{\le j-2}\tau_{y'}\gamma\tau_y^* P_j}\,dy\,dy',$$
$$\cB_\gamma^{(\sigma_2)}=\sum_{j\in\Z}\sum_{|\beta|=\lfloor \sigma_2\rfloor+1} 2^{2jd}\int\Phi_\beta(2^j y, 2^j y')2^{j(\sigma-\lfloor \sigma_2\rfloor-1)}\rho_{P_j\tau_{y'}\gamma\tau_y^* \partial^\beta P_{\le j-2}}\,dy\,dy',$$
where $\tau_y$ is the translation operator $(\tau_y f)(x):=f(x-y)$.
\end{remark}

\begin{remark}
 The additional factors $(-1)^{|\alpha|}$ and $(-1)^{|\beta|}$ in \eqref{eq:bony-density-matrix-general} compared to \eqref{eq:bony-op} come from the fact that $\partial_{y_j}\gamma(x,y)=-(\gamma\partial_j)(x,y)$ since the operator $\partial_j$ is antisymmetric.
\end{remark}

We also gather several estimates used in \cite{Li-19} that we will need.

\begin{proposition}\label{prop:est-leibniz}
 \begin{enumerate}

  \item (Fefferman-Stein \cite{FefSte-71}) Let $1<p<+\ii$ and $1<r\le+\ii$. Then, for any sequence of functions $(f^{(k)})$ we have
  \begin{equation}\label{eq:fefferman-stein}
    \| Mf^{(k)} \|_{L^p\ell^r_k} \lesssim \| f^{(k)} \|_{L^p\ell^r_k},
  \end{equation}
 where we defined the Hardy-Littlewood maximal function
$$Mf(x)=\sup_{r>0}\frac{1}{r^d}\int_{B(x,r)}|f(y)|\,dy = \sup_{r>0}\int_{B(0,1)}|f(x+ry)|\,dy.$$

 \item For any $\psi:\R^d\to\R$ such that $|\psi(x)|\lesssim(1+|x|)^{-(d+2)}$, any $\lambda>0$, any $f\in\cS(\R^d)$ and any $x\in\R^d$ we have
 \begin{equation}\label{eq:control-convolution-maximal}
    |((\lambda^d\psi(\lambda\cdot))*f)(x)| \le C Mf(x),
 \end{equation}
 with $C>0$ that depends only on $\psi$ and not on $\lambda$, $x$, nor $f$. In particular, for all $j\in\Z$ we have
 \begin{equation}\label{eq:control-LP-maximal}
    |P_j f|\lesssim Mf.
 \end{equation}

 \item Let $1<p<+\ii$ and $1<r\le+\ii$. Then, for any sequence of functions $(f^{(k)})$ we have
 \begin{equation}\label{eq:Pkfk}
    \| P_k f^{(k)} \|_{L^p\ell^r_k} \lesssim \| f^{(k)} \|_{L^p\ell^r_k}.
 \end{equation}

 \item Let $1<p<+\ii$ and $\sigma\in\R$. Then, we have
 \begin{equation}\label{eq:bernstein}
     \| 2^{\sigma k} P_k f\|_{L^p\ell^2_k} \lesssim \| |D|^\sigma f\|_{L^p} \lesssim \| 2^{\sigma k} P_k f\|_{L^p\ell^2_k}.
 \end{equation}

 \item \cite[Lemma 2.13]{Li-19} Let $1<p<+\ii$ and $\sigma>0$. Then, we have
 \begin{equation}\label{eq:bernsteinLeq}
  \| 2^{-j\sigma}|D|^\sigma P_{\le j} f\|_{L^p\ell^2_j}\lesssim \|f\|_{L^p}.
 \end{equation}

 \end{enumerate}

\end{proposition}

\begin{remark}
 In \cite[Lemma 2.5]{Li-19}, Point (3) of Proposition \ref{prop:est-leibniz} is stated only for $r<+\ii$. The case $r=+\ii$ follows from Points (2) and (1) of Proposition \ref{prop:est-leibniz}.
\end{remark}

\subsection{Fractional Leibniz decompositions}

In this section, we derive functional properties of the remainder terms of Lemma \ref{lem:bony-op} and Lemma \ref{lem:bony-density-matrix-general} that will be key to our applications. In particular, it will be important that the remainders are factorized as $A^*B$ for some operators $A$ and $B$.

\begin{lemma}[Fractional Leibniz decomposition of $|D|^\sigma W$]\label{lem:leibniz-op-general}

 Let $d\ge1$ and $\sigma>0$. Let $\sigma_1,\sigma_2\ge0$ be such that $\sigma_1+\sigma_2=\sigma$.  Let $q\in(1,2)$, $q_1\in(1,+\ii)$, and $q_2\in(2,+\ii)$ be such that $\tfrac1q=\tfrac{1}{q_1}+\tfrac{1}{q_2}$. Then, for any $W\in W^{\sigma,q_1}(\R^d)$ we can write, as an identity between operators from $W^{\sigma,q_2}(\R^d)$ to $L^q(\R^d)$,
 \begin{equation}\label{eq:leibniz-op-lemma-general}
    |D|^\sigma  W = \sum_{|\alpha|\le\lfloor \sigma_1\rfloor}\frac{1}{\alpha!}(\partial^\alpha W) D^{\sigma ,\alpha}+\sum_{|\beta|\le\lfloor \sigma_2\rfloor}\frac{1}{\beta!}(D^{\sigma ,\beta} W)\partial^\beta + L^{(1)}(g^{(1)}(W))^* L^{(2)}(g^{(2)}(W)),
 \end{equation}
where $L^{(1)}:B^{(1)}\to\cB(L^{q'}_x,L^2(X))$ and $L^{(2)}:B^{(2)}\to\cB(\dot{W}^{\sigma_2,q_2}_x,L^2(X))$ are linear and continuous, $B^{(1)}$ and $B^{(2)}$ are some Banach space and $X$ some measure space, and $g^{(\ell)}: \dot{W}^{\sigma_1,q_1}_x\to B^{(\ell)}$ are such that
$$\|g^{(1)}(W)\|_{B^{(1)}}\lesssim \||D|^{\sigma_1}W\|_{L^{q_1}_x}^{\theta},\quad \|g^{(2)}(W)\|_{B^{(2)}}\lesssim \||D|^{\sigma_1}W\|_{L^{q_1}_x}^{1-\theta},$$
where $\theta:=\tfrac{q_1(q'-2)}{2q'}\in(0,1)$ verifies $1-\theta=\tfrac{q_1(q_2-2)}{2q_2}$.
\end{lemma}

\begin{remark}

 The constraints $q<2$ and $q_2>2$ are absent for the standard fractional Leibniz rule \cite[Thm. 1.2(1)]{Li-19}. Here, they are needed to ensure the existence of the operators $L^{(1)}$ and $L^{(2)}$. For instance, $L^{(2)}(g)$ maps $L^{q_2}$ to $L^2$ and for this we need $q_2>2$ (because $L^{(2)}(g)$ is a kind of multiplication operator). The similar remark applied to $L^{(1)}(g)$ which maps $L^{q'}$ to $L^2$, forcing $q'>2$ and hence $q<2$. This being said, notice that Lemma \ref{lem:leibniz-op-general} implies the more standard Leibniz rule of \cite[Thm. 1.2(1)]{Li-19} in this range of parameters:
 $$
 \begin{multlined}
 \left\||D|^\sigma  (Wf) - \sum_{|\alpha|\le\lfloor \sigma_1\rfloor}\frac{1}{\alpha!}(\partial^\alpha W) (D^{\sigma ,\alpha}f)-\sum_{|\beta|\le\lfloor \sigma_2\rfloor}\frac{1}{\beta!}(D^{\sigma ,\beta} W)(\partial^\beta f)\right\|_{L^q_x} \\
 \lesssim \||D|^{\sigma_1}W\|_{L^{q_1}_x}\||D|^{\sigma_2}f\|_{L^{q_2}_x},
 \end{multlined}
 $$
 which is coherent with the fact that our proof is just a reformulation of the proof of \cite{Li-19}.
\end{remark}

\begin{remark}
 The term $(L^{(1)})^* L^{(2)}$ can be interpreted as behaving like $(|D|^{\sigma_1}W)|D|^{\sigma_2}$, which can indeed be written as $(L^{(1)})^* L^{(2)}$ with $L^{(1)}=(|D|^{\sigma_1}W)^{\theta}$ which sends, as a multiplication operator, $L^{q'}(\R^d)$ to $L^2(\R^d)$ and with $L^{(2)}=(|D|^{\sigma_1}W)^{1-\theta}|D|^{\sigma_2}$ which sends $\dot{W}^{\sigma_2,q_2}(\R^d)$ to $L^2(\R^d)$.
\end{remark}

We write a similar decomposition for $|D|^\sigma\rho_\gamma$ when $\gamma$ is a density matrix.

\begin{lemma}[Fractional Leibniz decomposition of $|D|^\sigma\rho_\gamma$]\label{lem:leibniz-gamma}
Let $d\ge1$ and $\sigma>0$. Let $\sigma_1,\sigma_2\ge0$ be such that $\sigma_1+\sigma_2=\sigma$.  Let $q\in(1,+\ii)$ and $q_1,q_2\in(2,+\ii)$ be such that $\tfrac1q=\tfrac{1}{q_1}+\tfrac{1}{q_2}$. Then, for any bounded operator $\gamma$ on $L^2(\R^d)$ such that $\langle\nabla\rangle^{\sigma+d(\frac12-\frac{1}{q_1})}\gamma\langle\nabla\rangle^{\sigma+d(\frac12-\frac{1}{q_2})}$ is trace-class and for any $W\in L^{q'}(\R^d)$ we can write
\begin{multline}\label{eq:leibniz-gamma}
 \int_{\R^d}W(x)\left((|D|^\sigma\rho_\gamma)(x)-\sum_{|\alpha|\le \sigma_1}\frac{(-1)^{|\alpha|}}{\alpha!}\rho_{\partial^\alpha\gamma D^{\sigma ,\alpha}}(x) -\sum_{|\beta|\le \sigma_2}\frac{(-1)^{|\beta|}}{\beta!}\rho_{D^{\sigma ,\beta}\gamma\partial^\beta}(x)\right)\,dx \\
 =\tr_{L^2_x}A^{(2)}(h^{(2)}(W))^* A^{(1)}(h^{(1)}(W))\gamma
\end{multline}
where $A^{(1)}:B^{(1)}\to\cB(\dot{W}^{\sigma_1,q_1}_x,L^2(X))$ and $A^{(2)}:B^{(2)}\to\cB(\dot{W}^{\sigma_2,q_2}_x,L^2(X))$ are linear and continuous, $B^{(1)}$ and $B^{(2)}$ are some Banach space and $X$ some measure space, and $h^{(\ell)}: L^{q'}_x\to B^{(\ell)}$ are such that
$$\|h^{(1)}(W)\|_{B^{(1)}}\lesssim \|W\|_{L^{q'}_x}^{\theta},\quad \|h^{(2)}(W)\|_{B^{(2)}}\lesssim \|W\|_{L^{q'}_x}^{1-\theta},$$
where $\theta:=\tfrac{q'(q_1-2)}{2q_1}\in(0,1)$.
\end{lemma}

\begin{remark}
 The assumption $\langle\nabla\rangle^{\sigma+d(\frac12-\frac{1}{q_1})}\gamma\langle\nabla\rangle^{\sigma+d(\frac12-\frac{1}{q_2})}\in\gS^1(L^2(\R^d))$ implies that one can write $\gamma=\sum_j\lambda_j|u_j\rangle\langle v_j|$ with $u_j\in H^{\sigma+d(\frac12-\frac{1}{q_1})}\hookrightarrow W^{\sigma,q_1}$, $v_j\in H^{\sigma+d(\frac12-\frac{1}{q_2})}\hookrightarrow W^{\sigma,q_2}$, and $\lambda_j\in\ell^1$. In particular, for any $j$, $|D|^\sigma(u_j\bar{v_j})\in L^q_x$ by \eqref{eq:leibniz-standard} and hence $|D|^\sigma\rho_\gamma=\sum_j \lambda_j|D|^\sigma( u_j\bar{v_j})\in L^q$. By the same argument, one can show that $\rho_{\partial^\alpha \gamma D^{\sigma,\alpha}}\in L^q$ and that $\rho_{D^{\sigma,\beta}\gamma\partial^\beta}\in L^q$, so that the left side of \eqref{eq:leibniz-gamma} is well-defined. Similarly, using that $H^{\sigma+d(\frac12-\frac{1}{q_1})}\hookrightarrow W^{\sigma,q_1}$ and $H^{\sigma+d(\frac12-\frac{1}{q_2})}\hookrightarrow W^{\sigma,q_2}$, we deduce that the operator
 $$\langle\nabla\rangle^{-\sigma-d(\frac12-\frac{1}{q_2})}A^{(2)}(h^{(2)}(W))^* A^{(1)}(h^{(1)}(W))\langle\nabla\rangle^{-\sigma-d(\frac12-\frac{1}{q_1})}$$
 is bounded on $L^2$, so that the right side of \eqref{eq:leibniz-gamma} is also well-defined. In the following sections, we will exploit some Schatten class properties of the operators $A^{(1)}$ and $A^{(2)}$ to estimate $|D|^\sigma\rho_\gamma$ in $L^q$ for a larger class of $\gamma$.
\end{remark}

\begin{remark}
 The interpretation of \eqref{eq:leibniz-gamma} is that we want to estimate $|D|^\sigma\rho_\gamma$ in $L^q_x$ by testing it against $W\in L^{q'}_x$. When $\gamma=|u\rangle\langle v|$, we obtain that
 \begin{multline*}
  \int_{\R^d}W(x)\left(|D|^\sigma(uv) - \sum_{|\alpha|\le\lfloor \sigma_1\rfloor}\frac{1}{\alpha!}(\partial^\alpha u) (D^{\sigma ,\alpha}v)-\sum_{|\beta|\le\lfloor \sigma_2\rfloor}\frac{1}{\beta!}(D^{\sigma ,\beta} u)(\partial^\beta v) \right)\,dx \\
  = \langle A^{(1)}(h^{(1)}(W))u,A^{(2)}(h^{(2)}(W))v\rangle_{L^2(X)},
 \end{multline*}
 so that we recover the estimate
 $$
 \left\| |D|^\sigma(uv) - \sum_{|\alpha|\le\lfloor \sigma_1\rfloor}\frac{1}{\alpha!}(\partial^\alpha u) (D^{\sigma ,\alpha}v)-\sum_{|\beta|\le\lfloor \sigma_2\rfloor}\frac{1}{\beta!}(D^{\sigma ,\beta} u)(\partial^\beta v) \right\|_{L^q_x} \lesssim \||D|^{\sigma_1}u\|_{L^{q_1}_x}\||D|^{\sigma_2}v\|_{L^{q_2}_x}.
 $$
\end{remark}

\begin{proof}[Proof of Lemma \ref{lem:leibniz-op-general} and Lemma \ref{lem:leibniz-gamma}]
 As a first general remark to obtain the decompositions \eqref{eq:leibniz-op-lemma-general} and \eqref{eq:leibniz-gamma} from \eqref{eq:bony-op} and \eqref{eq:bony-density-matrix-general} respectively, we will use that a sum of terms of the type $(L^{(1)})^* L^{(2)}$ is still of this form. More precisely, if for $n=1,2$ one has $L_n^{(1)}:B_n^{(1)}\to\cB(L^{q'}_x,L^2(X_n))$ and $L_n^{(2)}:B_n^{(2)}\to\cB(\dot{W}^{\sigma_2,q_2}_x,L^2(X_n))$, with $B_n^{(1)}$ and $B_n^{(2)}$ some Banach space and $X_n$ some measure space, and if $g_n^{(\ell)}: \dot{W}^{\sigma_1,q_1}_x\to B_n^{(\ell)}$ satisfy the properties of Lemma \ref{lem:leibniz-op-general}, one can then write
$$(L_1^{(1)}(g_1^{(1)}(W))^* L_1^{(2)}(g_1^{(2)}(W)) + (L_2^{(1)}(g_2^{(1)}(W))^* L_2^{(2)}(g_2^{(2)}(W)) = (L^{(1)}(g^{(1)}(W))^* L^{(2)}(g^{(2)}(W)),$$
with $B^{(\ell)}:=B_1^{(\ell)}\times B_2^{(\ell)}$, $g^{(\ell)}=(g^{(\ell)}_1,g^{(\ell)}_2)$, $X=X_1\sqcup X_2$, and for any $f\in L^{q'}_x$ one has
$$L^{(1)}(g_1,g_2)(f)(x):=
\begin{cases}
 L^{(1)}_1(g_1)(f)(x) & \text{if}\ x\in X_1,\\
 L^{(1)}_2(g_2)(f)(x) & \text{if}\ x\in X_2,
\end{cases}
$$
and the same definition for $L^{(2)}$. In the rest of the proof, we will thus write each of the terms of the right side of \eqref{eq:bony-op} (except the first two) as $(L^{(1)})^* L^{(2)}$, and we can sum all these different terms to obtain a single $(L^{(1)})^* L^{(2)}$ by the argument we just mentioned. Similarly, we write each of the terms of the right side of \eqref{eq:bony-density-matrix-general} (except the first two), tested against a $W(x)$, as $\tr (A^{(1)})^*A^{(2)}\gamma$. Since there are similar terms in \eqref{eq:bony-op} and \eqref{eq:bony-density-matrix-general} (for instance $\sum_{j\in\Z}|D|^\sigma  (P_j W)\tilde{P_j}$ and $\sum_{j\in\Z}|D|^\sigma \rho_{P_j\gamma\tilde{P_j}}$), we will treat such terms together instead of separating the proofs of Lemma \ref{lem:leibniz-op-general} and Lemma \ref{lem:leibniz-gamma}.

\paragraph{\textbf{Part 1:}} We begin with showing that there exist $L^{(1)}$, $g^{(1)}$, $L^{(2)}$, $g^{(2)}$, $A^{(1)}$, $h^{(1)}$, $A^{(2)}$, $h^{(2)}$ such that we have
$$\sum_{j\in\Z}|D|^\sigma  (P_j W)\tilde{P_j}=L^{(1)}(g^{(1)}(W))^* L^{(2)}(g^{(2)}(W)),$$
as well as
$$\int_{\R^d}W(x) \sum_{j\in\Z}|D|^\sigma \rho_{P_j\gamma\tilde{P_j}}(x)\,dx=\tr_{L^2_x}A^{(2)}(h^{(2)}(W))^* A^{(1)}(h^{(1)}(W))\gamma.$$
Let us start with the construction of $L^{(1)}$, $g^{(1)}$, $L^{(2)}$, $g^{(2)}$. Since $W_j\tilde{P_j}g$ has Fourier support in $\{|\xi| \le \tfrac72 2^j\}$, $P_k (W_j\tilde{P_j}g)$ is non-zero only if $j-k>-\log_27$. Let $Q_k'$ be the Fourier multiplier by the function $\psi(\tfrac{\xi}{2^k})$ with $\psi(\xi):=|\xi|^\sigma\phi(\xi)$. Since $\phi$ is supported away from the origin, $\psi$ is also a Schwartz function.
Then, it satisfies $|D|^\sigma P_k = Q_k' 2^{\sigma k}$. Hence, we rewrite the diagonal term as
\begin{align*}
  \sum_{j\in\Z}|D|^\sigma  (P_j W)\tilde{P_j} &=\sum_{k\in\Z}\sum_{j\in\Z}P_k|D|^\sigma  (P_j W)\tilde{P_j} \\
  &= \sum_{k\in\Z}\sum_{j>k-\log_27}P_k|D|^\sigma  (P_j W)\tilde{P_j} \\
  &= \sum_{k\in\Z}\sum_{m>-\log_2 7} P_k|D|^\sigma (P_{k+m} W)\tilde{P_{k+m}}\\
  &= \sum_{m>-\log_2 7} \sum_{k\in\Z} 2^{\sigma k}Q_k' (P_{k+m} W)\tilde{P_{k+m}}\\
  &= \sum_{m>-\log_2 7} 2^{-\sigma m}\sum_{k\in\Z} Q_k' (2^{(k+m)\sigma}P_{k+m} W)\tilde{P_{k+m}}\\
  &= \sum_{m>-\log_2 7} 2^{-\sigma m}\sum_{k\in\Z} Q_{k-m}'(2^{\sigma k}P_{k} W)\tilde{P_k} \\
  &= L^{(1)}(g^{(1)}(W))^* L^{(2)}(g^{(2)}(W))
\end{align*}
where we defined
$$[g^{(1)}(W)](x,k):=\frac{2^{\sigma_1 k}(P_kW)(x)}{\|2^{\sigma_1 j}(P_jW)(x)\|_{\ell^2_j}^{1-\theta}},\quad [g^{(2)}(W)](x):=\|2^{\sigma_1 j}(P_jW)(x)\|_{\ell^2_j}^{1-\theta},$$
$$[L^{(1)}(g)f](x,k,m):=2^{-\sigma m/2}\1_{m>-\log_27}g(x,k)(Q_{k-m}'f)(x),$$
$$[L^{(2)}(g)f](x,k,m):= 2^{-\sigma m/2}\1_{m>-\log_27}g(x)(2^{\sigma_2k}\tilde{P_k}f)(x).$$
Since $|Q_{k-m'}' f|\lesssim Mf$ by \eqref{eq:control-convolution-maximal}, we deduce that
$$\|L^{(1)}(g)f\|_{L^2_x\ell^2_{k,m}} \lesssim \|g\|_{L^{q_1/\theta}_x\ell^2_k}\|Q_{k-m}' f\|_{L^{q'}_x\ell^\ii_{k,m}}\lesssim \|g\|_{L^{q_1/\theta}_x\ell^2_k}\|f\|_{L^{q'}_x},$$
and hence for any $g\in L^{q_1/\theta}_x\ell^2_k=:B^{(1)}$, $L^{(1)}(g)\in\cB(L^{q'}_x,L^2_x\ell^2_{k,m})$. Here, we thus take $L^2(X):=L^2_x\ell^2_{k,m}$. Similarly, we have by \eqref{eq:bernstein}
$$\|L^{(2)}(g)f\|_{L^2_x\ell^2_{k,m}}\le \|g\|_{L^{q_1/(1-\theta)}_x}\|2^{\sigma_2k}\tilde{P_k}f\|_{L^{q_2}_x\ell^2_k}\lesssim\|g\|_{L^{q_1/(1-\theta)}_x} \||D|^{\sigma_2}f\|_{L^{q_2}_x},$$
and hence for any $g\in L^{q_1/(1-\theta)}_x=:B^{(2)}$, we have $L^{(2)}(g)\in\cB(\dot{W}^{\sigma_2,q_2}_x,L^2_x\ell^2_{k,m})$. Finally, notice that we have by \eqref{eq:bernstein}
$$\|g^{(1)}(W)\|_{L^{q_1/\theta}_x\ell^2_k} \lesssim \| |D|^{\sigma_1}  W\|_{L^{q_1}_x}^{\theta},\quad \|g^{(2)}(W)\|_{L^{q_1/(1-\theta)}_x} \lesssim \||D|^{\sigma_1}  W\|_{L^{q_1}_x}^{1-\theta}.$$
We now turn to the construction of $A^{(1)}$, $h^{(1)}$, $A^{(2)}$, $h^{(2)}$. Using the formula
$$\hat{\rho_\gamma}(\xi)=\frac{1}{(2\pi)^{d/2}}\int_{\R^d}\hat{\gamma}(\eta,\xi-\eta)\,d\eta,$$
we deduce that
$$\supp\hat{\rho_\gamma} \subset \bigcup_{q\in\R^d}\supp\hat{\gamma}(\cdot,q)+\bigcup_{p\in\R^d}\supp\hat{\gamma}(p,\cdot),$$
a similar Fourier support property as the product of two functions. Hence, we deduce as above that
$$\sum_{j\in\Z}|D|^\sigma \rho_{P_j\gamma\tilde{P_j}} =
\sum_{m>-\log_2 7} 2^{-\sigma m}\sum_{k\in\Z} 2^{\sigma k}Q_{k-m}' \rho_{P_k \gamma \tilde{P_k}},
$$
and hence
\begin{align*}
 \int_{\R^d}W(x) \sum_{j\in\Z}|D|^\sigma \rho_{P_j\gamma\tilde{P_j}}(x)\,dx &= \sum_{m>-\log_2 7} 2^{-\sigma m}\sum_{k\in\Z} 2^{\sigma k}\tr_{L^2_x}\tilde{P_k}(Q_{k-m}'W)P_k \gamma \\
 &= \tr_{L^2_x}A^{(2)}(h^{(2)}(W))^* A^{(1)}(h^{(1)}(W))\gamma,
\end{align*}
where
$$[h^{(1)}(W)](x,k,m)=|(Q_{k-m}'W)(x)|^{\theta},\quad [h^{(2)}(W)](x,k,m)=\frac{(Q_{k-m}'W)(x)}{|(Q_{k-m}'W)(x)|^{\theta}},$$
 $$[A^{(1)}(h)f](x,k)=2^{-\sigma m/2}\1_{m>-\log_27}2^{k\sigma_1}h(x,k,m)(P_kf)(x),$$
 $$[A^{(2)}(h)f](x,k)=2^{-\sigma m/2}\1_{m>-\log_27}2^{k\sigma_2}h(x,k,m)(\tilde{P_k}f)(x).$$
 Since we have
 $$\|A^{(1)}(h)f\|_{L^2_x\ell^2_{k,m}}\lesssim \|h\|_{L^{q'/\theta}_x\ell^\ii_{k,m}} \||D|^{\sigma_1}f\|_{L^{q_1}_x}$$
we deduce that for any $h\in L^{q'/\theta}_x\ell^\ii_{k,m}=:B^{(1)}$, we have $A^{(1)}(h)\in \cB(\dot{W}^{\sigma_1,q_1}_x,L^2(X))$ with $L^2(X):=L^2_x\ell^2_{k,m}$. We also have
 $$\|A^{(2)}(h)f\|_{L^2_x\ell^2_{k,m}}\lesssim \|h\|_{L^{q'/(1-\theta)}_x\ell^\ii_{k,m}} \||D|^{\sigma_2}f\|_{L^{q_2}_x},$$
 so that for any $h\in L^{q'/(1-\theta)}_x\ell^\ii_{k,m}=:B^{(2)}$, we have $A^{(2)}(h)\in\cB(\dot{W}^{\sigma_2,q_2}_x,L^2(X))$. Finally, we have
 $$\|h^{(1)}(W)\|_{L^{q'/\theta}_x\ell^\ii_{k,m}}\lesssim\| W\|_{L^{q'}_x}^{\theta},\quad\|h^{(2)}(W)\|_{L^{q'/(1-\theta)}_x\ell^\ii_{k,m}}\lesssim \| W\|_{L^{q'}_x}^{1-\theta},$$
 which concludes Part 1 of the proof.

 \paragraph{\textbf{Part 2:}} For fixed $\alpha\in\N^d$ with $|\alpha|=\lfloor \sigma_1\rfloor+1$, let us define $F_{\alpha,j} W:=2^{j(\sigma_1-\lfloor \sigma_1\rfloor-1)}\partial^\alpha P_{\le j-2}W$ and
 $$\cA^{(\sigma_1)}_{W,\alpha}(f)(x)=\sum_{j\in\Z} 2^{2jd}\int\Phi_\alpha(2^j y, 2^j y')(F_{\alpha,j}W)(x-y')(2^{j\sigma_2}P_jf)(x-y)\,dy\,dy',$$
 so that the term $\cA_W^{(\sigma_1)}$ in \eqref{eq:bony-op} is $\cA_W^{(\sigma_1)}=\sum_{|\alpha|=\lfloor \sigma_1\rfloor+1}\cA^{(\sigma_1)}_{W,\alpha}$ for some Schwartz function $\Phi_\alpha$ (in this identity, we absorbed the term $C_\alpha(2^jy)^\alpha$ into the function $\Phi$, which does not change its fast decay). Let us show that we can write
 $$\cA^{(\sigma_1)}_{W,\alpha}=L^{(1)}(g^{(1)}(W))^* L^{(2)}(g^{(2)}(W)).$$
 We define
$$[g^{(1)}(W)](x,j)=\frac{(F_{\alpha,j} W)(x)}{\| (F_{\alpha,k} W)(x) \|_{\ell^2_k}^{1-\theta}},\quad [g^{(2)}(W)](x)=\| (F_{\alpha,k} W)(x) \|_{\ell^2_k}^{1-\theta},$$
$$[L^{(1)}(g)f](x,y,y',j)= 2^{jd}\sqrt{|\Phi_\alpha|}(2^j y,2^j y')g(x-y',j)f(x),$$
$$[L^{(2)}(g)f](x,y,y',j)=2^{jd} \sqrt{\Phi_\alpha}(2^j y,2^j y')g(x-y') (2^{j\sigma_2}P_j f)(x-y).$$
We have by \eqref{eq:control-convolution-maximal}
\begin{align*}
  \|L^{(1)}(g)f\|_{L^2_{x,y,y'}\ell^2_{j}}^2 &= \sum_{j\in\Z}\int 2^{2jd}|\Phi_\alpha(2^jy,2^jy')||g(x-y',j)|^2|f(x)|^2\,dx\,dy\,dy'\\
  &= \sum_{j\in\Z}\int 2^{jd}|\Phi_\alpha(y,2^jy')|g(x,j)|^2|f(x+y')|^2\,dx\,dy\,dy' \\
  &\lesssim \int \|g(x)\|_{\ell^2_{j}}^2 M(|f|^2)(x)\,dx \\
  &\lesssim \|g\|_{L^{q_1/\theta}_x\ell^2_{j,\alpha}}^{2}\|f\|_{L^{q'}_x}^2
\end{align*}
and hence $L^{(1)}\in\cB(B^{(1)},\cB(L^{q'}_x,L^2(X)))$ with $B^{(1)}:=L^{q_1/\theta}_x\ell^2_{j}$ and $L^2(X):=L^2_{x,y,y'}\ell^2_{j}$. We also have by \eqref{eq:fefferman-stein} and  \eqref{eq:bernstein}
\begin{align*}
 \|L^{(2)}(g)f\|_{L^2_{x,y,y'}\ell^2_{j}}^2 &= \sum_{j\in\Z} \int 2^{2jd}|\Phi_\alpha(2^j y,2^jy')||g(x-y')|^2 |2^{j\sigma_2}(P_jf)(x-y)|^2\,dx\,dy\,dy' \\
 &= \sum_{j\in\Z} \int 2^{2jd}|\Phi_\alpha(2^j y,2^jy')|g(x+y-y')|^2 |2^{j\sigma_2}(P_jf)(x)|^2\,dx\,dy\,dy'\\
 &\lesssim \int (M|g|^2)(x)\|2^{j\sigma_2}(P_jf)(x)\|_{\ell^2_j}^2\,dx\\
 &\lesssim \|g\|_{L^{q_1/(1-\theta)}_x}^{2}\||D|^{\sigma_2}f\|_{L^{q_2}_x}^2,
\end{align*}
We estimated the $(y,y')$ integral above by changing variables $(c,r)=(\tfrac{y+y'}{2},y'-y)$ so that it can be written as $((2^{jd}\zeta(2^j\cdot))*_r|g|^2)(x)$, with $\zeta(r):=\int |\Phi_\alpha(c-r/2,c+r/2)|\,dc$. Then, we have $|((2^{jd}\zeta(2^j\cdot))*_r|g|^2)(x)|\lesssim M|g|^2(x)$ by \eqref{eq:control-convolution-maximal}. We thus proved that $L^{(2)}\in\cB(B^{(2)},\cB(\dot{W}^{\sigma_2,q_2}_x,L^2(X)))$ with $B^{(2)}:=L^{q_1/(1-\theta)}_x$. Notice that
we indeed have $\cA^{(\sigma_1)}_{W,\alpha}=L^{(1)}(g^{(1)}(W))^* L^{(2)}(g^{(2)}(W))$. Finally, by \eqref{eq:bernsteinLeq} and the boundedness of $\partial^\alpha|D|^{-\lfloor\sigma_1\rfloor-1}$ on $L^{q_1}_w$ we have the estimates
$$\|g^{(1)}(W)\|_{L^{q_1/\theta}_x\ell^2_{j}}\lesssim \|\partial^\alpha|D|^{-\lfloor\sigma_1\rfloor-1}|D|^{\sigma_1}  W\|_{L^{q_1}_x}^{\theta}\lesssim \||D|^{\sigma_1}  W\|_{L^{q_1}_x}^{\theta},$$
$$
\|g^{(2)}(W)\|_{L^{q_1/(1-\theta)}_x}\lesssim \||D|^{\sigma_1}  W\|_{L^{q_1}_x}^{1-\theta}.
$$
The operators coming from the $\cB_W^{(\sigma_2)}$ piece in \eqref{eq:bony-op} can be treated in the same way as the $\cA$ piece.

The analogous term for density matrices is the term $\cA_\gamma^{(\sigma_1)}$ in \eqref{eq:bony-density-matrix-general}. For this term (where again we treat only a fixed $\alpha$), we have
\begin{align*}
\int_{\R^d}W(x)\cA_{\gamma,\alpha}^{(\sigma_1)}(x)\,dx &=
\sum_{j\in\Z}\int \,dy\,dy' 2^{2jd}\Phi_\alpha(2^j y, 2^j y')2^{j(\sigma-\lfloor \sigma_1\rfloor-1)}\tr_{L^2_x} P_j\tau_y^* W \tau_{y'} \partial^\alpha P_{\le j-2} \gamma \\
&= \tr_{L^2_x}A^{(2)}(h^{(2)}(W))^* A^{(1)}(h^{(1)}(W))\gamma,
\end{align*}
where
$$[h^{(1)}(W)](x)=|W(x)|^\theta,\quad [h_2(W)](x)=\frac{W(x)}{|W(x)|^\theta},$$
$$[A^{(2)}(h)f](x,y,y',j)=2^{jd}2^{j\sigma_2}\sqrt{|\Phi_\alpha|}(2^jy,2^jy')h(x)(P_jf)(x-y),$$
 $$[A^{(1)}(h)f](x,y,y',j)=2^{jd}2^{j(\sigma_1-\lfloor \sigma_1\rfloor-1)}\sqrt{\Phi_\alpha}(2^jy,2^jy')h(x)(\partial^\alpha P_{\le j-2}f)(x-y').$$
 Since we have
 \begin{align*}
 \| A^{(2)}(h)f\|_{L^2_{x,y,y'}\ell^2_j}^2 &= \sum_{j}\int dx\,dy\,dy' 2^{2jd}|\Phi_\alpha(2^jy,2^jy')||h(x)|^2|2^{j\sigma_2} (P_jf)(x-y)|^2 \\
  &= \sum_{j}\int dx\,dy\,dy' 2^{2jd}|\Phi_\alpha(2^jy,2^jy')||h(x+y)|^2|2^{j\sigma_2} (P_jf)(x)|^2 \\
  &\lesssim \int dx (M|h|^2)(x) \| 2^{j\sigma_2} (P_jf)(x) \|_{\ell^2_{j}}^2 \\
  &\lesssim \|h\|_{L^{q'/(1-\theta)}_{x}}^2\||D|^{\sigma_2}f\|_{L^{q_2}_{x}}^2,
\end{align*}
 and a similar estimate for $A^{(1)}$, we deduce that $A^{(1)}$ and $A^{(2)}$ have the desired properties with $B^{(1)}:=L^{q'/\theta}_x$, $B^{(2)}=L^{q'/(1-\theta)}_x$, and $L^2(X)=L^2_{x,y,y,'}\ell^2_j$. The term $\cB_\gamma^{(\sigma_2)}$ is treated in the same way.

 \paragraph{\textbf{Part 3:}} Let us now show that we can write, for any $|\alpha|\le\lfloor \sigma_1\rfloor$,
 \begin{equation}\label{eq:leibniz-part3}
    \sum_{j\in\Z}(\partial^\alpha W_j) D^{\sigma ,\alpha}\tilde{P_j}=L^{(1)}(g^{(1)}(W))^* L^{(2)}(g^{(2)}(W)).
 \end{equation}
 Define
$$[g^{(1)}(W)](x,j):=\frac{2^{j(\sigma_1-|\alpha|)}(\partial^\alpha W_j)(x)}{\| 2^{k(\sigma_1-|\alpha|)}(\partial^\alpha W_k)(x)\|_{\ell^2_k}^{1-\theta}},$$
$$[g^{(2)}(W)](x):=\| 2^{k(\sigma_1-|\alpha|)}(\partial^\alpha W_k)(t,x)\|_{\ell^2_k}^{1-\theta},$$
$$[L^{(1)}(g)f](x,j):=g(x,j)f(x),$$
$$[L^{(2)}(g)f](x,j):=g(x)2^{j(|\alpha|-\sigma_1)}(D^{\sigma ,\alpha}\tilde{P_j}f)(x).$$
We then have by \eqref{eq:bernstein}
$$\|L^{(1)}(g)f\|_{L^2_x\ell^2_{j}}\lesssim \|g\|_{L^{q_1/\theta}_x\ell^2_{j}}\|f\|_{L^{q'}_x},
$$
\begin{align*}
\|L^{(2)}(g)f\|_{L^2_x\ell^2_{j}} &\lesssim \|g\|_{L^{q_1/(1-\theta)}_x}\|2^{j(|\alpha|-\sigma_1)}D^{\sigma ,\alpha}\tilde{P_j}f\|_{L^{q_2}_x\ell^2_{j}} \\
&\lesssim \|g\|_{L^{q_1/(1-\theta)}_x}\||D|^{\sigma_2}f\|_{L^{q_2}_x}.
\end{align*}
We thus have proved that $L^{(1)}\in\cB(B^{(1)},\cB(L^{q'}_x,L^2(X)))$ with $B^{(1)}:=L^{q_1/\theta}_x\ell^2_{j}$ and $L^2(X):=L^2_x\ell^2_{j}$, and that $L^{(2)}\in\cB(B^{(2)},\cB(\dot{W}^{\sigma_2,q_2}_x,L^2(X)))$ with $B^{(2)}:=L^{q_1/(1-\theta)}_x$. With these definitions, we indeed have \eqref{eq:leibniz-part3}. Finally, we also have by \eqref{eq:bernstein} that
$$\|g^{(1)}(W)\|_{L^{q_1/\theta}_x\ell^2_{j}}\lesssim \||D|^{\sigma_1} W\|_{L^{q_1}_x}^{\theta},\quad \|g^{(2)}(W)\|_{L^{q_1/(1-\theta)}_x}\lesssim \||D|^{\sigma_1} W\|_{L^{q_1}_x}^{1-\theta}.$$
Summing over $\alpha$ and using the argument as the beginning of the proof, we infer that also have
$$\sum_{|\alpha|\le\lfloor \sigma_1\rfloor}\frac{1}{\alpha!}\sum_{j\in\Z}(\partial^\alpha W_j) D^{\sigma ,\alpha}\tilde{P_j}=L^{(1)}(g^{(1)}(W))^* L^{(2)}(g^{(2)}(W)).$$
By the same argument, we can write
$$\sum_{|\beta|\le\lfloor \sigma_2\rfloor}\frac{1}{\beta!}\sum_{j\in\Z}(D^{s,\beta}\tilde{P_j}W)P_j\partial^\beta=L^{(1)}(g^{(1)}(W))^* L^{(2)}(g^{(2)}(W)).$$
For the density matrix version, we write
\begin{align*}
 \int_{\R^d} W(x)\sum_{j\in\Z}\rho_{P_j \partial^\alpha \gamma D^{\sigma ,\alpha}\tilde{P_j}}(x)\,dx &= \sum_{j\in\Z}\tr_{L^2_x} D^{\sigma ,\alpha}\tilde{P_j}WP_j\partial^\alpha \gamma \\
 &= \tr_{L^2_x}A^{(2)}(h^{(2)}(W))^* A^{(1)}(h^{(1)}(W))\gamma
\end{align*}
where
$$[h^{(2)}(W)](x)=|W(x)|^{1-\theta},\quad [h^{(1)}(W)](x)=\frac{W(x)}{|W(x)|^{1-\theta}},$$
$$[A^{(2)}(h)f](x,j)=2^{j(|\alpha|-\sigma_1)}h(x)(\tilde{P_j}D^{\sigma,\alpha}f)(x),\quad [A^{(1)}(h)f](x,j)=2^{j(\sigma_1-|\alpha|)}h(x)(P_j\partial^\alpha f)(x)$$
and we conclude in the same way with $B^{(1)}:=L^{q'/\theta}_x$, $B^{(2)}=L^{q'/(1-\theta)}_x$, and $L^2(X)=L^2_{x}\ell^2_j$.

\paragraph{\textbf{Part 4:}}

We finally show that, for fixed $|\alpha|\le\lfloor \sigma_1\rfloor$, we can write
$$\sum_{j\in\Z}(\partial^\alpha W_j) D^{\sigma ,\alpha}P_{\le j-2}=L^{(1)}(g^{(1)}(W))^* L^{(2)}(g^{(2)}(W)).$$
Since $(\partial^\alpha W_{j})(D^{\sigma ,\alpha}P_{\le j-2}g)$ has Fourier support on $\{ \tfrac{5}{24}2^j \le |\xi| \le \tfrac{35}{24}2^j\}$, $P_k(\partial^\alpha W_{j}D^{\sigma ,\alpha}P_{\le j-2}g)$ is non-zero only if $k+\log_2\tfrac{12}{35} \le j \le k+\log_2\tfrac{28}{5}$. Hence, we have
$$\sum_{j\in\Z}(\partial^\alpha W_j) D^{\sigma ,\alpha}P_{\le j-2}=\sum_{m\in M}\sum_{j\in\Z} P_j (\partial^\alpha W_{j+m}) D^{\sigma ,\alpha}P_{\le j+m-2}$$
for some finite set $M$. Hence, we may treat these terms by fixing $\alpha$ and $m$. We define
$$[g^{(1)}(W)](x):=\|2^{k(\sigma_1-|\alpha|)}(\partial^\alpha W_k)(x)\|_{\ell^2_k}^{\theta},
$$
$$[g^{(2)}(W)](x,j):=\frac{2^{j(\sigma_1-|\alpha|)}(\partial^\alpha W_{j+m})(x)}{\|2^{k(\sigma_1-|\alpha|)}(\partial^\alpha W_k)(x)\|_{\ell^2_k}^{\theta}},$$
$$[L^{(1)}(g)f](x,j):=\frac{1}{\alpha!}g(x)(P_jf)(x),$$
$$[L^{(2)}(g)f](x,j):=g(x,j)2^{j(|\alpha|-\sigma_1)}(D^{\sigma,\alpha}P_{\le j+m-2}f)(x).$$
We have by \eqref{eq:bernstein}, \eqref{eq:bernsteinLeq}, and \eqref{eq:control-convolution-maximal} that
$$\|L^{(1)}(g)f\|_{L^2_x\ell^2_{j}} \lesssim \|g\|_{L^{q_1/\theta}_x}\|f\|_{L^{q'}_x},$$
\begin{align*}
\|L^{(2)}(g)f\|_{L^2_x\ell^2_{j}} &\lesssim \|g\|_{L^{q_1/(1-\theta)}_x\ell^2_{j}}\| 2^{j(|\alpha|-\sigma_1)}D^{\sigma ,\alpha}P_{\le j+m-2}f\|_{L^{q_2}_x\ell^\ii_{j}} \\
&\lesssim \|g\|_{L^{q_1/(1-\theta)}_x\ell^2_{j}}\||D|^{\sigma_2}f\|_{L^{q_2}_x}.
\end{align*}
To obtain the last inequality, we have distinguished the case $|\alpha|<\sigma_1$, for which we estimate the $\ell^\ii_j$ norm by the $\ell^2_j$ norm and apply \eqref{eq:bernsteinLeq}, and the case $|\alpha|=\sigma_1$, for which we apply \eqref{eq:control-convolution-maximal} to deduce that $|(D^{\sigma ,\alpha}P_{\le j+m-2}f)(x)|\lesssim M(D^{\sigma ,\alpha}f)(x)$. We thus have proved that $L^{(1)}\in\cB(B^{(1)},\cB(L^{q'}_x,L^2(X)))$ and $L^{(2)}\in\cB(B^{(2)},\cB(\dot{W}^{\sigma_2,q_2}_x,L^2(X)))$ with $B^{(1)}:=L^{q_1/\theta}_x$, $B^{(2)}:=L^{q_1/(1-\theta)}_x\ell^2_{j}$ and $L^2(X):=L^2_x\ell^2_{j}$. Notice that we also have
$$\sum_{j\in\Z} P_j (\partial^\alpha W_{j+m}) D^{\sigma ,\alpha}P_{\le j+m-2}=L^{(1)}(g^{(1)}(W))^* L^{(2)}(g^{(2)}(W)).$$
Finally, we have by \eqref{eq:bernsteinLeq}
$$\|g^{(1)}(W)\|_{L^{q_1/\theta}_x}\lesssim \||D|^{\sigma_1}  W\|_{L^{q_1}_x}^{\theta},\quad \|g^{(2)}(W)\|_{L^{q_1/(1-\theta)}_x\ell^2_{j}}\lesssim \||D|^{\sigma_1}W\|_{L^{q_1}_x}^{1-\theta}.$$
Summing over $\alpha$, we also obtain that
$$\sum_{|\alpha|\le\lfloor \sigma_1\rfloor}\frac{1}{\alpha!}\sum_{j\in\Z}(\partial^\alpha W_j) D^{\sigma ,\alpha}P_{\le j-2}=L^{(1)}(g^{(1)}(W))^* L^{(2)}(g^{(2)}(W)).$$
Using the same method, we can also write
$$\sum_{|\beta|\le\lfloor \sigma_2\rfloor}\frac{1}{\beta!}\sum_{j\in\Z}(D^{\sigma,\beta}P_{\le j-2}W)P_j\partial^\beta=L^{(1)}(g^{(1)}(W))^* L^{(2)}(g^{(2)}(W)).$$
For the density matrix version, we first write that
$$
\sum_{|\alpha|\le \sigma_1}\frac{1}{\alpha!}\sum_{j\in\Z} \rho_{P_j\partial^\alpha\gamma D^{\sigma ,\alpha}P_{\le j-2}} = \sum_{|\alpha|\le \sigma_1}\frac{1}{\alpha!}\sum_{j\in\Z}\sum_{m\in M} P_j\rho_{P_{j+m}\partial^\alpha\gamma D^{\sigma ,\alpha}P_{\le j+m-2}},
$$
$$
\sum_{|\beta|\le \sigma_2}\frac{1}{\beta!}\sum_{j\in\Z} \rho_{ D^{\sigma ,\beta}P_{\le j-2}\gamma P_j\partial^\beta}
=\sum_{|\beta|\le \sigma_2}\frac{1}{\beta!}\sum_{j\in\Z}\sum_{m\in M} P_j\rho_{ D^{\sigma ,\beta}P_{\le j+m-2}\gamma P_{j+m}\partial^\beta},
$$
and hence we have
\begin{align*}
 \int_{\R^d}W(x)\sum_{j\in\Z} (P_j\rho_{P_{j+m}\partial^\alpha\gamma D^{\sigma ,\alpha}P_{\le j+m-2}})(x)\,dx &= \sum_{j\in\Z}\tr_{L^2_x}P_{\le j+m-2}D^{\sigma ,\alpha}(P_jW)P_{j+m}\partial^\alpha\gamma \\
 &= \tr_{L^2_x}A^{(2)}(h^{(2)}(W))^* A^{(1)}(h^{(1)}(W))\gamma,
\end{align*}
where
 $$[h_2(W)](x,j)=\frac{(P_jW)(x)}{\|(P_k W)(x)\|_{\ell^2_k}^{\theta}},\quad [h_1(W)](x)=\|(P_k W)(x)\|_{\ell^2_k}^{\theta},$$
 $$[A^{(2)}(h)f](x,j)=2^{j(|\alpha|-\sigma_1)}h(x,j)(P_{\le j+m-2}D^{s,\alpha}f)(x),$$
 $$[A^{(1)}(h)f](x,j)=2^{j(\sigma_1-|\alpha|)}h(x)(P_{j+m}\partial^\alpha f)(x).$$
 Since we have
 $$\|A^{(2)}(h)f\|_{L^2_x\ell^2_j}\lesssim\|h\|_{L^{q'/(1-\theta)}_x\ell^2_j} \||D|^{\sigma_2}f\|_{L^{q_2}_x},\quad\|A^{(1)}(h)f\|_{L^2_x\ell^2_j}\lesssim \|h\|_{L^{q'/\theta}_x} \||D|^{\sigma_1}f\|_{L^{q_1}_x},$$
 and
 $$\|h_2(W)\|_{L^{q'/(1-\theta)}_x\ell^2_j}\lesssim \| W\|_{L^{q'}_x}^{1-\theta},\quad\|h_1(W)\|_{L^{q'/\theta}_x}\lesssim\| W\|_{L^{q'}_x}^{\theta},$$
 we have the desired property for the sum over $\alpha$. The sum over $\beta$ is treated in the same way.

\end{proof}

\subsection{A consequence}

In our applications, we will use \eqref{eq:leibniz-op-lemma-general} to get information on the operator $\sd^s W \sd^{-s}$ when $W\in W^{s,\nu}_x$. By Lemma \ref{lem:leibniz-op-general} applied to $\sigma_1=s$ and $\sigma_2=0$, we may think of the operator $\sd^s W \sd^{-s}$ as a sum of terms behaving like $(\sd^\sigma W)\sd^{-\sigma}$ for $0\le\sigma\le s$. We will then write such terms as $(\sd^\sigma W)^{1-\theta}(\sd^\sigma W)^\theta\sd^{-\sigma}$ for various choices of $\theta$. Since these terms are not exactly $(\sd^\sigma W)\sd^{-\sigma}$ but rather have the same mapping properties on Sobolev spaces, we first state a general decomposition which can be interpreted as including all the various possibilities for the parameter $\theta$. We will then apply this general decomposition to the specific values of $\theta$ that will be useful to us.

\begin{lemma}\label{lem:leibniz-decomp-general}

 Let $d\ge1$, $s\ge0$, $\nu\in(1,+\ii)$. For all $a\in\{0,1,\ldots,\lfloor s\rfloor,s\}$, let $q_a^{(1)}, q_a^{(2)}\in(1,+\ii)$ be such that
 $$\frac1\nu -\frac{s-a}{d}\le\frac{1}{\nu_a}:=\frac{1}{2(q_a^{(1)})'}+\frac{1}{2(q_a^{(2)})'}\le\frac1\nu.$$
 Then, for all $a\in\{0,1,\ldots,\lfloor s\rfloor,s\}$ there exists $L_a^{(1)}:B_a^{(1)}\to\cB(L^{2q_a^{(1)}}_x,L^2(X_a))$ and $L_a^{(2)}:B_a^{(2)}\to\cB(L^{2q_a^{(2)}}_x,L^2(X_a))$ linear and continuous, where $B_a^{(1)}$ and $B_a^{(2)}$ are some Banach space and $X_a$ some measure space, and $g_a^{(\ell)}: W^{s,\nu}_x\to B_a^{(\ell)}$ such that for all $W\in W^{s,\nu}_x$ we have
 \begin{equation}\label{eq:leibniz-decomp-general}
 \sd^s W\sd^{-s}=\sum_{a\in\{0,1,\ldots,\lfloor s\rfloor,s\}} L_a^{(1)}(g_a^{(1)}(W))^* L_a^{(2)}(g_a^{(2)}(W))\sd^{-a},
 \end{equation}
 and
 $$\|g_a^{(1)}(W)\|_{B_a^{(1)}}\lesssim \|W\|_{W^{s,\nu}_x}^{\nu_a/2(q_a^{(1)})'},\quad \|g_a^{(2)}(W)\|_{B_a^{(2)}}\lesssim \|W\|_{W^{s,\nu}_x}^{\nu_a/2(q_a^{(2)})'}.$$
 \end{lemma}

\begin{remark}
 As we said before, the interpretation of \eqref{eq:leibniz-decomp-general} is the following: the term $(L_a^{(1)})^* L_a^{(2)}$ behaves as the multiplication operator by the function $\sd^{a} W\in W^{s-a,\nu}\hookrightarrow L^{\nu_a}$. Indeed, since $(\sd^{a} W)^{\nu_a/2(q_a^{(1)})'}\in L^{2(q_a^{(1)})'}$, the multiplication by $(\sd^{a} W)^{\nu_a/2(q_a^{(1)})'}$ maps $L^{2q_a^{(1)}}(\R^d)$ to $L^2(\R^d)$, and since $(\sd^{a} W)^{\nu_a/2(q_a^{(2)})'}\in L^{2(q_a^{(2)})'}$,the multiplication by $(\sd^{a} W)^{\nu_a/2(q_a^{(2)})'}$ maps $L^{2q_a^{(2)}}(\R^d)$ to $L^2(\R^d)$.
\end{remark}

\begin{proof}
 The case $s=0$ just corresponds to the splitting $W=|W|^{\nu/2(q_0^{(1)})'}W^{\nu/2(q_0^{(2)})'}$. If $s>0$, we apply Lemma \ref{lem:leibniz-op-general} to $\sigma=s=\sigma_1$, $\sigma_2=0$, $q_1=\nu$, $q_2=2q_s^{(2)}$ (notice that in this case, $q'=2q_s^{(1)}$) and we get
$$
  \begin{multlined}
    \sd^s W\sd^{-s} = \sum_{|\alpha|\le \lfloor s\rfloor}\frac{1}{\alpha!}\frac{\sd^s}{1+|D_x|^s}(\partial^\alpha_x W) D^{s,\alpha}_x\sd^{-s}+\frac{\sd^s}{1+|D_x|^s}((|D|^sW)+W)\sd^{-s}\\
    +  \frac{\sd^s}{1+|D_x|^s}L^{(1)}(g^{(1)}(W))^* L^{(2)}(g^{(2)}(W))\sd^{-s}.
  \end{multlined}
$$
Since $\frac{\sd^s}{1+|D_x|^s}$ is a multiplier on $L^{2q_s^{(1)}}_x$, the term involving $(L^{(1)})^* L^{(2)}$ has the right form for $a=s$ if we replace $L^{(1)}(g)$ by $L^{(1)}(g)\frac{\sd^s}{1+|D_x|^s}$. For the prior term, we split
$$\frac{\sd^s}{1+|D_x|^s}((|D|^sW)+W)\sd^{-s}=L^{(1)}(g^{(1)}(W))^*L^{(2)}(g^{(2)}(W))\sd^{-s}$$
with
$L^{(1)}(g)=\frac{\sd^s}{1+|D_x|^s}g$, $L^{(2)}(g)=g$, $g^{(1)}(W)=((|D|^sW)+W)^{\tfrac{\nu}{2(q_s^{(1)})'}}$, $g^{(2)}(W)=|(|D|^sW)+W|^{\tfrac{\nu}{2(q_s^{(2)})'}}$, which has the right form for $a=s$ since $(|D|^sW)+W\in L^\nu_x$. Notice that several terms corresponding to a fixed value of $a$ can be regrouped as a single term by changing the Banach spaces $B_a$ and the measure space $X_a$, as in the proof of Lemma \ref{lem:leibniz-op-general}. The terms with the sum over $\alpha$ are treated in the sawe way: those with a fixed value of $|\alpha|\le \lfloor s\rfloor$ correspond to $a=|\alpha|$, since one can write
$$\frac{1}{\alpha!}\frac{\sd^s}{1+|D_x|^s}(\partial^\alpha_x W) D^{s,\alpha}_x\sd^{-s}=L^{(1)}(g^{(1)}(W))^*L^{(2)}(g^{(2)}(W))\sd^{-|\alpha|}$$
with $L^{(1)}(g)=\frac{1}{\alpha!}\frac{\sd^s}{1+|D_x|^s}g$, $L^{(2)}(g)=gD^{s,\alpha}\sd^{-(s-|\alpha|)}$, $g^{(1)}(W)=(\partial^\alpha W)^{\tfrac{\nu_{|\alpha|}}{2(q_{|\alpha|}^{(1)})'}}$, $g^{(2)}(W)=|\partial^\alpha W|^{\tfrac{\nu_{|\alpha|}}{2(q_{|\alpha|}^{(2)})'}}$. Since $\partial^\alpha W\in W^{s-|\alpha|,\nu}\hookrightarrow L^{\nu_{|\alpha|}}$, we get the result.

\end{proof}

\section{The Christ-Kiselev lemma in Schatten spaces}\label{sec:CK}

We now provide the other key tool to prove our result, which will allow to estimate operators which are defined as nested integrals $\int_0^t\,dt_1\,\int_0^{t_1}\,dt_2(\cdots)$. It is here that the importance of the factorization of the operators $\sd^s W\sd^{-s}$ obtained in Lemma \ref{lem:leibniz-decomp-general} will become clear (see Corollary \ref{coro:CK} and its first application below in Proposition \ref{prop:UV-bez}).

\subsection{Fully orthogonal version}

  \begin{lemma}[Christ-Kiselev lemma in Schatten spaces, first version]\label{lem:CK}

   Let $\gH_1$, $\gH_2$ be Hilbert spaces, and $X_1$, $X_2$ be Banach spaces. Let $(B(t,s))_{t,s\in\R}$ be a (measurable) family of bilinear maps from $X_1\times X_2$ to $\cB(\gH_1,\gH_2)$. For any $t,s\in\R$ and $x_1\in X_1$, $x_2\in X_2$, we denote by $B(t,s,x_1,x_2)\in\cB(\gH_1,\gH_2)$ the image of $(x_1,x_2)$ by $B(t,s)$ (which is bilinear in $(x_1,x_2)$). For any measurable $g_1:\R\to X_1$ and $g_2:\R\to X_2$, let us denote by $T(g_1,g_2)$ and $T_<(g_1,g_2)$ the linear maps from $L^2_t\gH_1$ to $L^2_t\gH_2$ defined by: for any $u\in L^2_t\gH_1$ and for any $t\in\R$,
   $$(T(g_1,g_2)u)(t) = \int_\R B(t,s,g_1(t),g_2(s))u(s)\,ds\in \gH_2,$$
   $$(T_<(g_1,g_2)u)(t) = \int_{-\ii}^t B(t,s,g_1(t),g_2(s))u(s)\,ds\in \gH_2,$$
   Assume that for some $\alpha\in[1,+\ii]$, $p_1, p_2\in[1,+\ii)$ with $1/\alpha<1/p_1+1/p_2$ there exists $C>0$ such that for any $g_1$, $g_2$ we have
   $$\| T(g_1,g_2) \|_{\gS^\alpha(L^2_t\gH_1,L^2_t\gH_2)}\le C\|g_1\|_{L^{p_1}X_1} \|g_2\|_{L^{p_2} X_2}.$$
   Then, there exists $C>0$ such that for any $g_1$, $g_2$ we have
   $$\| T_<(g_1,g_2) \|_{\gS^\alpha(L^2_t\gH_1,L^2_t\gH_2)}\le C\|g_1\|_{L^{p_1}X_1} \|g_2\|_{L^{p_2} X_2}.$$

  \end{lemma}

  The proof is an adaptation of the proof of \cite[Lemma 3.1]{Tao-00} to operators in Schatten spaces.

  \begin{proof}

   By bilinearity we may assume that $\|g_1\|_{L^{p_1}X_1}^{p_1}=1/2=\|g_2\|_{L^{p_2}X_2}^{p_2}$. Define then
   $$F(t)=\int_{-\ii}^t ( \| g_1(s)\|_{X_1}^{p_1} + \|g_2(s)\|_{X_2}^{p_2}) \,ds,$$
   so that $F:\R\to[0,1]$ is non-decreasing. As in \cite[Lemma 3.1]{Tao-00}, we may assume that $F$ is increasing. For dyadic intervals $I,J\subset[0,1]$, we denote by $I\sim J$ if $I$ and $J$ have the same length, are not adjacent but have adjacent parents, and if $I$ is on the right of $J$. Then, almost everywhere on $t,s$ we have
   $$\1(t-s>0) = \sum_{I\sim J}\1_{F^{-1}(J)}(s)\1_{F^{-1}(I)}(t)$$
   and hence
   $$T_<(g_1,g_2)=\sum_{I\sim J}T_{I,J}, \quad T_{I,J}:= T(\1_{F^{-1}(I)}g_1,\1_{F^{-1}(J)}g_2)$$
   so that
   $$\| T_<(g_1,g_2) \|_{\gS^\alpha(L^2_t\gH_1,L^2_t\gH_2)}
   \lesssim \sum_{k\ge2} \left\|\sum_{\substack{I\sim J \\ |I|=2^{-k}}}T_{I,J}\right\|_{\gS^\alpha(L^2_t\gH_1,L^2_t\gH_2)}.$$
   We now argue that, by an orthogonality argument, we have
   \begin{equation}\label{eq:ortho-CK}
   \left\|\sum_{\substack{I\sim J \\ |I|=2^{-k}}}T_{I,J}\right\|_{\gS^\alpha(L^2_t\gH_1,L^2_t\gH_2)}^\alpha \lesssim \sum_{\substack{I\sim J \\ |I|=2^{-k}}} \left\|T_{I,J}\right\|_{\gS^\alpha(L^2_t\gH_1,L^2_t\gH_2)}^\alpha.
   \end{equation}
   To see this property, we partition the family of dyadic intervals $I\subset[0,1]$ with $|I|=2^{-k}$ as two families,
   $$\cI_{\rm even}:=\left\{\left[\frac{2n}{2^k},\frac{2n+1}{2^k} \right],\ n=1,\ldots,2^{k-1}-1 \right\},$$
   $$\cI_{\rm odd}:=\left\{\left[\frac{2n+1}{2^k},\frac{2n+2}{2^k} \right],\ n=1,\ldots,2^{k-1}-1\right\}.$$
   Notice that we start with $n=1$ because the intervals $I=[0,2^{-k}]$ and $I=[2^{-k},2^{-(k-1)}]$ which would correspond to $n=0$ have no $J$ such that $I\sim J$ (they are on the extreme left of $[0,1]$ so there can be no $J$ on their left). Hence, we deduce that
   $$\left\|\sum_{\substack{I\sim J \\ |I|=2^{-k}}}T_{I,J}\right\|_{\gS^\alpha(L^2_t\gH_1,L^2_t\gH_2)} \le \left\|\sum_{\substack{I\sim J \\ I\in\cI_{\rm even}}}T_{I,J}\right\|_{\gS^\alpha(L^2_t\gH_1,L^2_t\gH_2)}
   +\left\|\sum_{\substack{I\sim J \\ I\in\cI_{\rm odd}}}T_{I,J}\right\|_{\gS^\alpha(L^2_t\gH_1,L^2_t\gH_2)}$$
   For any $I\in\cI_{\rm even}$, there exists a unique $J$ such that $I\sim J$, which is $J=I-2\times2^{-k}$. For any $I\in\cI_{\rm odd}$, there are exactly two $J$ such that $I\sim J$, namely $J=I-2\times2^{-k}$ and $J=I-3\times2^{-k}$. Hence, we have
   $$\sum_{\substack{I\sim J \\ I\in\cI_{\rm even}}}T_{I,J}=\sum_{I\in\cI_{\rm even}}T_{I,I-2\times2^{-k}},$$
   $$\left\|\sum_{\substack{I\sim J \\ I\in\cI_{\rm odd}}}T_{I,J}\right\|_{\gS^\alpha(L^2_t\gH_1,L^2_t\gH_2)}\le
   \left\|\sum_{I\in\cI_{\rm odd}}T_{I,I-2\times2^{-k}} \right\|_{\gS^\alpha(L^2_t\gH_1,L^2_t\gH_2)}
   +\left\|\sum_{I\in\cI_{\rm odd}}T_{I,I-3\times2^{-k}} \right\|_{\gS^\alpha(L^2_t\gH_1,L^2_t\gH_2)}.$$
   For any $I\in\cI_{\rm even}$, let us abbreviate $T_I:=T_{I,I-2\times2^{-k}}$. By definition, $(T_Iu)(t)=0$ for all $t\notin F^{-1}(I)$, and hence $T_I^*v=0$ if $v(t)=0$ for all $t\in F^{-1}(I)$. If $I\neq I'$, we have $F^{-1}(I)\cap F^{-1}(I')=\emptyset$ and hence $T_I^* T_{I'}=0$. As a consequence, we deduce that $A:=\sum_{I\in\cI_{\rm even}}T_I$ satisfies $A^*A=\sum_{I\in\cI_{\rm even}}T_I^* T_I$. Similarly, $T_Iu=0$ if $u(s)=0$ for all $s\in F^{-1}(I-2\times2^{-k})$, so that $(T_I^*v)(s)=0$ for all $s\notin F^{-1}(I-2\times2^{-k})$. This means that in the orthogonal decomposition
   $$L^2_t\gH_1=\bigoplus_{K\ {\rm dyadic},\ |K|=2^{-k}}^\perp \cH_K,$$
   where $\cH_K:=\{u\in L^2_t \gH_1,\ u(s)=0\ {\rm for\   all}\ s\notin F^{-1}(K)\}$, the operator $A^*A$ is block-diagonal (since $T_I^*T_I$ stabilizes $\cH_{I-2\times2^{-k}}$ and $(T_I^*T_I)_{|\cH_K}=0$ if $K\neq I-2\times2^{-k}$). As a consequence, we have
   $$\tr(A^*A)^{\alpha/2}=\sum_{I\in\cI_{\rm even}}\tr(T_I^* T_I)^{\alpha/2}=\sum_{I\in\cI_{\rm even}}\|T_I\|_{\gS^\alpha(L^2_t\gH_1,L^2_t\gH_2)}^\alpha.$$
   Performing the same argument for the intervals in $\cI_{\rm odd}$, we obtain \eqref{eq:ortho-CK}. Now using that
   $$\left\|T_{I,J}\right\|_{\gS^\alpha(L^2_t\gH_1,L^2_t\gH_2)}\lesssim \|\1_{F^{-1}(I)}g_1\|_{L^{p_1}X_1}\|\1_{F^{-1}(J)}g_2\|_{L^{p_2}X_2}\lesssim 2^{-k(1/p_1+1/p_2)},$$
   we get that
   $$\| T_<(g_1,g_2) \|_{\gS^\alpha(L^2_t\gH_1,L^2_t\gH_2)}
   \lesssim \sum_{k\ge2}2^{k(1/\alpha-1/p_1-1/p_2)},
   $$
   which is finite if $1/\alpha<1/p_1+1/p_2$.

  \end{proof}

  \begin{remark}\label{rk:reversed-CK}
   Of course, the previous proof is exactly the same if one considers $\int_t^\ii ds$ in the definition of $T_<$ instead of $\int_{-\ii}^tds$.
  \end{remark}

  \begin{remark}
   The same statement holds for operators that map $L^2_t(I,\gH_1)$ to $L^2_t(I,\gH_2)$ for any time interval $I$ (in this case, the integral defining $T$ is restricted to $s\in I$, and the integral defining $T_<$ is restricted to $s\in I$ and $s<t$. The functions $g_1,g_2$ are also defined only on $I$).
  \end{remark}

  \begin{corollary}\label{coro:CK}

   Let $\gH_1$, $\gH_2$, $\gK$ be Hilbert spaces and $X_1$, $X_2$ be Banach spaces.
   For $j\in\{1,2\}$, let $(T_j(t))_{t\in\R}$ be a family of linear maps from $X_j$ to $\cB(\gK,\gH_j)$ (again, and for any $x_j\in X_j$, we use the notation $T_j(t,x_j):=T_j(t)(x_j)\in\cB(\gK,\gH_j)$). For any $g_j:\R\to X_j$, define the operator $[T_1(t,g_1(t))T_2(s,g_2(s))^*]_<:L^2_t\gH_2\to L^2_t\gH_1$ by: for any $u\in L^2_t\gH_2$ and any $t\in\R$,
   $$([T_1(t,g_1(t))T_2(s,g_2(s)^*]_<u)(t) := T_1(t,g_1(t))\int_{-\ii}^t T_2(s,g_2(s))^*u(s)\,ds.$$
   Define also the operators $\tilde{T_j}(g_j):\gK\to L^2_t\gH_j$ by: for any $u\in\gK$ and any $t\in\R$,
   $$(\tilde{T_j}(g_j)u)(t)=T_j(t,g_j(t))u.$$
   Assume that for some $\alpha_j\in[1,+\ii]$, $p_j\in[1,+\ii)$ with $1/\alpha_1+1/\alpha_2<1/p_1+1/p_2$, there exists $C_j>0$ such that for all $g_j$ we have
   $$\|\tilde{T_j}(g_j)\|_{\gS^{\alpha_j}(\gK,L^2_t\gH_j)} \le C\|g_j\|_{L^{p_j}X_j}.$$
   Then, for $1/\alpha=1/\alpha_1+1/\alpha_2$ there exists $C>0$ such that for all $g_1$, $g_2$ we have
   $$\|[T_1(t,g_1(t))T_2(s,g_2(s))^*]_<\|_{\gS^\alpha(L^2_t\gH_2,L^2_t\gH_1)} \le C\|g_1\|_{L^{p_1}X_1}\|g_2\|_{L^{p_2}X_2}.$$

  \end{corollary}

  \begin{proof}

   We apply Lemma \ref{lem:CK} to $B(t,s,x_1,x_2)=T_1(t,x_1)T_2(s,x_2)^*\in\cB(\gH_2,\gH_1)$. In the notations of Lemma \ref{lem:CK}, we then have
   $$T(g_1,g_2)=\tilde{T_1}(g_1)\tilde{T_2}(g_2)^*,$$
   so that the result follows from the Hölder inequality in Schatten spaces.

  \end{proof}

  \begin{remark}
   The assumption $1/\alpha_1+1/\alpha_2<1/p_1+1/p_2$ is in particular satisfied if $\alpha_1>p_1$ and $\alpha_2>p_2$, a situation we will often encounter.

  \end{remark}

  \begin{remark}\label{rk:reversed-coro-CK}
   From Remark \ref{rk:reversed-CK}, we also deduce that the same result holds if $\int_{-\ii}^t ds$ is replaced by $\int_t^\ii ds$ in the definition of $T_<$.
  \end{remark}

  \subsection{Partially orthogonal version}

  The above version of the Christ-Kiselev will not always be enough for our purposes. The version below will be key to treat the most singular term in the reaction of the background (see the proof of Proposition \ref{prop:rho22} of below).

  \begin{lemma}[Christ-Kiselev lemma in Schatten spaces, second version]

   Let $\gH_1$, $\gH_2$ be Hilbert spaces, and $X_1$, $X_2$ be Banach spaces. Let $(B(t,s))_{t,s\in\R}$ be a (measurable) family of bilinear maps from $X_1\times X_2$ to $\cB(\gH_1,\gH_2)$. For any $t,s\in\R$ and $x_1\in X_1$, $x_2\in X_2$, we denote by $B(t,s,x_1,x_2)\in\cB(\gH_1,\gH_2)$ the image of $(x_1,x_2)$ by $B(t,s)$ (which is bilinear in $(x_1,x_2)$). For any measurable $g_1:\R\to X_1$ and $g_2:\R\to X_2$, let us denote by $T(g_1,g_2)$ and $T_<(g_1,g_2)$ the linear maps from $L^2_t\gH_1$ to $\gH_2$ defined by: for any $u\in L^2_t\gH_1$,
   $$T(g_1,g_2)u = \int_\R\int_\R B(t,s,g_1(t),g_2(s))u(s)\,ds\,dt\in \gH_2,$$
   $$T_<(g_1,g_2)u = \int_\R\int_{-\ii}^t B(t,s,g_1(t),g_2(s))u(s)\,ds\,dt\in \gH_2,$$
   Assume that for some $\alpha\in[1,+\ii]$, $p_1, p_2\in[1,+\ii)$ with $\max(1/2,1/\alpha)<1/p_1+1/p_2$ there exists $C>0$ such that for all $g_1$, $g_2$ we have
   $$\| T(g_1,g_2) \|_{\gS^\alpha(L^2_t\gH_1,\gH_2)}\le C\|g_1\|_{L^{p_1}X_1} \|g_2\|_{L^{p_2} X_2}.$$
   Then, there exists $C>0$ such that for all $g_1$, $g_2$ we have
   $$\| T_<(g_1,g_2) \|_{\gS^\alpha(L^2_t\gH_1,\gH_2)}\le C\|g_1\|_{L^{p_1}X_1} \|g_2\|_{L^{p_2} X_2}.$$

  \end{lemma}

  \begin{proof}

   By bilinearity we may assume that $\|g_1\|_{L^{p_1}X_1}^{p_1}=1/2=\|g_2\|_{L^{p_2}X_2}^{p_2}$. Define then
   $$F(t)=\int_{-\ii}^t ( \| g_1(s)\|_{X_1}^{p_1} + \|g_2(s)\|_{X_2}^{p_2}) \,ds,$$
   so that $F:\R\to[0,1]$ is non-decreasing. For dyadic intervals $I,J\subset[0,1]$, we denote by $I\sim J$ if $I$ and $J$ have the same length, are not adjacent but have adjacent parents, and if $I$ is on the right of $J$. Then, almost everywhere on $t,s$ we have
   $$\1(t-s>0) = \sum_{I\sim J}\1_{F^{-1}(J)}(s)\1_{F^{-1}(I)}(t)$$
   and hence
   $$T_<(g_1,g_2)=\sum_{I\sim J}T_{I,J},\quad T_{I,J}:= T(\1_{F^{-1}(I)}g_1,\1_{F^{-1}(J)}g_2)$$
   so that
   $$\| T_<(g_1,g_2) \|_{\gS^\alpha(L^2_t\gH_1,\gH_2)}
   \lesssim \sum_{k\ge2} \left\|\sum_{\substack{I\sim J \\ |I|=2^{-k}}}T_{I,J}\right\|_{\gS^\alpha(L^2_t\gH_1,\gH_2)}.$$
   Since $T_{I,J}T_{I',J'}^*=0$ if $|J|=|J'|$ and $J\neq J'$, we deduce that
   $$\left\|\sum_{\substack{I\sim J \\ |I|=2^{-k}}}T_{I,J}\right\|_{\gS^\alpha(L^2_t\gH_1,\gH_2)}=\left\|\sum_{\substack{I\sim J \\ |I|=2^{-k}}}T_{I,J}T_{I,J}^*\right\|_{\gS^{\alpha/2}(\gH_2,\gH_2)}^{1/2}.$$
   If $\alpha\ge2$, we have
   $$\left\|\sum_{\substack{I\sim J \\ |I|=2^{-k}}}T_{I,J}T_{I,J}^*\right\|_{\gS^{\alpha/2}(\gH_2,\gH_2)}\le \sum_{\substack{I\sim J \\ |I|=2^{-k}}}\|T_{I,J}T_{I,J}^*\|_{\gS^{\alpha/2}(\gH_2,\gH_2)}
   =\sum_{\substack{I\sim J \\ |I|=2^{-k}}}\|T_{I,J}\|_{\gS^\alpha(L^2_t\gH_1,\gH_2)}^2$$
   and using that
   $$\left\|T_{I,J}\right\|_{\gS^\alpha(L^2_t\gH_1,\gH_2)}\lesssim \|\1_{F^{-1}(I)}g_1\|_{L^{p_1}X_1}\|\1_{F^{-1}(J)}g_2\|_{L^{p_2}X_2}\lesssim 2^{-k(1/p_1+1/p_2)},$$
   we get that
   $$\| T_<(g_1,g_2) \|_{\gS^\alpha(L^2_t\gH_1,\gH_2)}
   \lesssim \sum_{k\ge2}2^{k(1/2-1/p_1-1/p_2)},
   $$
   which is finite if $1/2<1/p_1+1/p_2$. If $\alpha<2$, we use that
   $$\left\|\sum_{\substack{I\sim J \\ |I|=2^{-k}}}T_{I,J}T_{I,J}^*\right\|_{\gS^{\alpha/2}(\gH_2,\gH_2)}^{\alpha/2} = \tr_{\gH_2}\left(\sum_{\substack{I\sim J \\ |I|=2^{-k}}}T_{I,J}T_{I,J}^*\right)^{\alpha/2}
   $$
   and since for any non-negative operators $A$ and $B$ we have $\tr (A+B)^{\alpha/2}\le \tr A^{\alpha/2}+\tr B^{\alpha/2}$ (see for instance \cite[Eq. (7)]{BhaKit-00}), we deduce that
   $$
   \tr_{\gH_2}\left(\sum_{\substack{I\sim J \\ |I|=2^{-k}}}T_{I,J}T_{I,J}^*\right)^{\alpha/2} \le \sum_{\substack{I\sim J \\ |I|=2^{-k}}} \tr_{\gH_2}(T_{I,J}T_{I,J}^*)^{\alpha/2} = \sum_{\substack{I\sim J \\ |I|=2^{-k}}} \|T_{I,J}\|_{\gS^\alpha(L^2_t\gH_1,\gH_2)}^\alpha.
   $$
   We conclude as in the fully orthogonal case that
   $$\| T_<(g_1,g_2) \|_{\gS^\alpha(L^2_t\gH_1,L^2_t\gH_2)}
   \lesssim \sum_{k\ge2}2^{k(1/\alpha-1/p_1-1/p_2)},
   $$
   which is finite if $1/\alpha<1/p_1+1/p_2$.

  \end{proof}

  \begin{remark}

   The term ``partially orthogonal'' refers to the fact that in the first version of the Christ-Kiselev lemma we have both $T_{I,J}T_{I',J'}^*=0$ if $J\neq J'$ and $T_{I',J'}^*T_{I,J}=0$ if $I\neq I'$, so that $\sum_{I\sim J} T_{I,J}$ has a block diagonal structure and is ``fully orthogonal''. In the second version, we only have $T_{I,J}T_{I',J'}^*=0$ if $J\neq J'$ and hence we have weaker Schatten bounds, but they are still strong enough for our applications.

  \end{remark}

  \begin{corollary}\label{coro:CK-2}

   Let $\gH_1$, $\gH_2$ be Hilbert spaces and $X_1$, $X_2$ be Banach spaces. For $j\in\{1,2\}$, let $(T_j(t))_{t\in\R}$ be a family of linear maps from $X_j$ to $\cB(\gH_1,\gH_2)$ (again, for any $x_j\in X_j$ we use the notation $T_j(t,x_j)\in\cB(\gH_1,\gH_2)$). For any $g_j:\R\to X_j$, define the operator $T_<(g_1,g_2):L^2_t\gH_2\to\gH_2$ by: for any $u\in L^2_t\gH_2$,
   $$T_<(g_1,g_2)u = \int_\R T_1(t,g_1(t))\int_{-\ii}^t T_2(s,g_2(s))^*u(s)\,ds\,dt.$$
   Define also the operators $\tilde{T_1}(g_1):\gH_1\to \gH_2$ and $\tilde{T_2}(g_2):\gH_1\to L^2_t\gH_2$ by: for any $u\in\gH_1$ and any $t\in\R$,
   $$\tilde{T_1}(g_1)u = \int_\R T_1(t',g_1(t')) u \,dt',\quad (\tilde{T_2}(g_2)u)(t)=T_2(t,g_2(t))u.$$
   Assume that for some $\alpha_j\in[1,+\ii]$, $p_j\in[1,+\ii)$ there exists $C_j>0$ such that for all $g_j$ we have
   $$\|\tilde{T_1}(g_1)\|_{\gS^{\alpha_1}(\gH_1,\gH_2)} \le C_1\|g_1\|_{L^{p_1}X_1},\quad \|\tilde{T_2}(g_2)\|_{\gS^{\alpha_2}(\gH_1,L^2_t\gH_2)} \le C_2\|g_2\|_{L^{p_2}X_2}.$$
   If $\max(1/2,1/\alpha)<1/p_1+1/p_2$ and $1/\alpha\le1/\alpha_1+1/\alpha_2$, then there exists $C>0$ such that for all $g_1$, $g_2$ we have
   $$\|T_<(g_1,g_2)\|_{\gS^\alpha(L^2_t\gH_2\to\gH_2)} \le C\|g_1\|_{L^{p_1}X_1}\|g_2\|_{L^{p_2}X_2}.$$

  \end{corollary}

\section{Strichartz estimates for $U_V$}\label{sec:strichartz-UV}

In this section, we extend the Strichartz estimates of Section \ref{sec:free} which were valid for the free propagator $e^{it\Delta_x}$ to the propagator of $-\Delta+V(t)$. To prove them, we rely on the fractional Leibniz rules of Section \ref{sec:leibniz} as well as the Christ-Kiselev lemma in Schatten spaces of Section \ref{sec:CK}. Furthermore, we not only estimate $\rho$ in $L^p_t L^q_x$ but also in $L^p_t W^{s,q}_x$, generalizing Remark \ref{rk:integer-derivatives} to fractional derivatives. It is thanks to these estimates that we will be able to treat the first term $\rho_{U_V(t,0)Q_{\text{in}}U_V(t,0)^*}$ in \eqref{eq:Phi}. As in the introduction, we will use the notation
$$U_V(t):=U_V(t,0).$$
The first result that we state is the fact that $U_V(t)$ satisfies the same Strichartz estimates for functions as $e^{it\Delta_x}$. In the case $s_0=0$, the proposition below has been proved for instance in \cite[Theorem 1.1]{Yajima-87}, \cite{AncPieVis-05}, and in the case $s_0\ge0$ in \cite[Lemma 3.9]{Hadama-23}. The proof we propose is standard and relies on inhomogeneous Strichartz estimates as well as Leibniz rules to estimate a product of functions in Sobolev spaces. We provide the proof to explain the origin of the conditions on the various exponents.

\begin{proposition}\label{prop:UV-Sonae}

 Let $d\ge2$, $s\ge0$ and $\mu,\nu\in(1,+\ii)$ be such that $2\le 2/\mu+d/\nu\le2+s$, $1/\nu<1/2+1/d$, and $2/\mu+d/\nu<2+d/2$. Let $p,q\in[1,+\ii]$ be such that $2/p+d/q=d$, with furthermore $(p,q)\neq(1,+\ii)$ if $d=2$. Let $s_0\in[0,s]$. Then, there exists $\delta>0$ and $C>0$ such that for any $V,\tilde{V}\in L^\mu_t W^{s,\nu}_x$ such that $\|V\|_{L^\mu_t W^{s,\nu}_x}\le\delta$, $\|\tilde{V}\|_{L^\mu_t W^{s,\nu}_x}\le\delta$ and any $u_0\in L^2_x$ we have
$$\| \sd^{s_0} U_V(t)\sd^{-s_0} u_0\|_{L^{2p}_t L^{2q}_x}\le C\|u_0\|_{L^2_x},$$
$$\| \sd^{s_0} (U_V(t)-U_{\tilde{V}}(t))\sd^{-s_0} u_0\|_{L^{2p}_t L^{2q}_x}\le C\|u_0\|_{L^2_x}\|V-\tilde{V}\|_{L^\mu_t W^{s,\nu}_x}.$$

\end{proposition}

\begin{proof}

Under our assumptions on $(d,s,\mu,\nu,s_0)$, there exists $\nu_0\in(1,+\ii)$ such that
$$2\le\frac{2}{\mu}+\frac{d}{\nu_0}\le2+s_0,\
\frac{1}{\nu_0}<\frac12+\frac1d,\
\frac{2}{\mu}+\frac{d}{\nu_0}<2+\frac d2,\
\frac{1}{\nu}-\frac{s-s_0}{d}\le\frac{1}{\nu_0}\le\frac{1}{\nu}.
$$
Let us then define $\tilde{\nu}$ by the relation
$$\frac{2}{\mu}+\frac{d}{\tilde{\nu}}=2,$$
equivalently $\tilde{\nu}=d\mu/(2(\mu-1))$. Since $\mu\in(1,+\ii)$ and $d\ge2$, we deduce that $\tilde{\nu}\in(1,+\ii)$. Defining $\zeta_0=2/\mu+d/\nu_0-2\in[0,s_0]$, we have the relation
$$\frac{1}{\tilde{\nu}}=\frac{1}{\nu_0}-\frac{\zeta_0}{d},$$
so that in particular $W^{s_0,\nu_0}\hookrightarrow L^{\tilde{\nu}}$. Let us now show that there exists $(\tilde{p},\tilde{q})\in(2,+\ii)$ such that
$$\frac{2}{\tilde{p}}+\frac{d}{\tilde{q}}=\frac d2,\quad \frac12<\frac{1}{\mu}+\frac{1}{\tilde{p}}<1,\quad \frac{1}{\tilde{p}}<\frac{d}{4}+1-\frac1\mu-\frac{d}{2\nu_0}.
$$
To do so, we just pick any $\tilde{p}\in(2,+\ii)$ such that
$$\frac12-\frac{1}{\mu}<\frac{1}{\tilde{p}}< 1-\frac{1}{\mu},\quad \frac{1}{\tilde{p}}<\frac{d}{4}+1-\frac1\mu-\frac{d}{2\nu_0}.$$
This is possible due to the bounds $1/\nu_0<1/2+1/d$ and $2/\mu+d/\nu_0<2+d/2$. Now given such a $(\tilde{p},\tilde{q})$, define $(r,c)\in(2,+\ii)$ by
$$\frac{1}{r'}=\frac{1}{\mu}+\frac{1}{\tilde{p}},\quad \frac{1}{c'}=\frac{1}{\tilde{\nu}}+\frac{1}{\tilde{q}}.$$
Notice that $2/r+d/c=d/2$. Finally, let us define $k$ by the relation $1/k+1/\nu_0=1/c'$, that is
$$\frac1k=\frac{1}{\tilde{q}}+\frac{1}{\tilde{\nu}}-\frac{1}{\nu_0}.$$
The fact that $1/k>0$ follows from the relation $1/\tilde{p}<d/4+1-1/\mu-d/(2\nu_0)$. Since $\tilde{\nu}\ge\nu_0$, we deduce that $k\in[\tilde{q},+\ii)$. Since $1/\nu_0-1/\tilde{\nu}\le s_0/d$, we also deduce that $W^{s_0,\tilde{q}}_x\hookrightarrow L^k_x$. Now by the Duhamel formula, we have
$$
\langle\nabla\rangle^{s_0} U_V(t) \langle\nabla\rangle^{-s_0}u_0 =
e^{it\Delta_x}u_0
- i \int_0^{t}
e^{i(t-t')\Delta_x}\langle\nabla_x\rangle^{s_0} V(t')U_V(t')\langle\nabla_x\rangle^{-s_0}u_0\,dt'.
$$
By the homogeneous \eqref{eq:stri-hom} and inhomogeneous \eqref{eq:stri-inhom} Strichartz estimates, we have
$$\|\langle\nabla\rangle^{s_0} U_V(t) \langle\nabla\rangle^{-s_0}u_0\|_{L^{\tilde{p}}_t L^{\tilde{q}}_x} \lesssim \|u_0\|_{L^2_x}+ \|V(t')U_V(t')\langle\nabla_x\rangle^{-s_0}u_0\|_{L^{r'}_{t'} W^{s_0,c'}_x}
$$
By the Leibniz rule \eqref{eq:leibniz-standard} and using that $W^{s_0,\nu_0}_x\hookrightarrow L^{\tilde{\nu}}_x$ as well as $W^{s_0,\tilde{q}}\hookrightarrow L^k_x$, we have
$$\|V(t')U_V(t')\langle\nabla_x\rangle^{-s_0}u_0\|_{L^{r'}_{t'} W^{s_0,c'}_x }\lesssim \|V\|_{L^\mu_t W^{s_0,\nu_0}_x} \| \langle\nabla_x\rangle^{s_0} U_V(t')\langle\nabla_x\rangle^{-s_0}u_0\|_{L^{\tilde{p}}_t L^{\tilde{q}}_x}.$$
Note that $W^{s,\nu}_x\hookrightarrow W^{s_0,\nu_0}_x$, hence choosing $C\|V\|_{L^\mu_t W^{s,\nu}_x}\le1/2$, we get that
$$\|\langle\nabla\rangle^{s_0} U_V(t) \langle\nabla\rangle^{-s_0}u_0\|_{L^{\tilde{p}}_t L^{\tilde{q}}_x} \lesssim \|u_0\|_{L^2_x}.$$
The estimate in $L^{2p}_tL^{2q}_x$ follows from the one in $L^{\tilde{p}}_t L^{\tilde{q}}_x$ using the same inhomogeneous Strichartz estimate with a right side estimated in $L^{r'}_{t'} W^{s_0,c'}_x$, for the same value of $(r,c)$ defined in terms of $(\tilde{p},\tilde{q})$. The proof of the estimate for the difference $U_V-U_{\tilde{V}}$ follows from the same proof, using the identity
$$V(t')U_V(t')-\tilde{V}(t')U_{\tilde{V}}(t')=(V(t')-\tilde{V}(t'))U_V(t')+\tilde{V}(t')(U_V(t')-U_{\tilde{V}}(t')).$$

\end{proof}

\begin{remark}
 The above proof should be understood as a construction of $U_V(t)$. Indeed, since we did not appeal to any prior result, an expression as $\|\langle\nabla\rangle^{s_0} U_V(t) \langle\nabla\rangle^{-s_0}u_0\|_{L^{\tilde{p}}_t L^{\tilde{q}}_x}$ is merely formal and could very well be $+\ii$. Hence, the above proof gives a construction of $U_V(t)$ in the following sense: if we define for any $A\in\cB(H^{s_0}_x\to L^{\tilde{p}}_tW^{s_0,\tilde{q}}_x)$ the map
 $$\Psi(A):u_0\in H^{s_0}_x\mapsto e^{it\Delta_x}u_0-i\int_0^t e^{i(t-t')\Delta}V(t')(Au_0)(t')\,dt',$$
 the above proof shows that $\Psi$ is a contraction on the Banach space $\cB(H^{s_0}_x\to L^{\tilde{p}}_tW^{s_0,\tilde{q}}_x)$ and hence as a unique fixed point on it, that we define as $U_V(t)$. Using Strichartz estimates, we then deduce that $U_V(t)\in \cB(H^{s_0}_x\to L^{2p}_tW^{s_0,2q}_x)$ for any $(p,q)$ as in Proposition \ref{prop:UV-Sonae}, and that for any $u_0$, $u(t):=\Psi(u_0)(t)$ solves $i\partial_tu=(-\Delta+V(t))u$.
\end{remark}

\begin{remark}\label{rk:negative-sobolev}
 The conclusion of Proposition \ref{prop:UV-Sonae} still holds if $\sd^{s_0} U_V(t)\sd^{-s_0}$ is replaced by $\sd^{-s_0}U_V(t)\sd^{s_0}$. Since there is no Leibniz rule for the derivative of negative order of a product, the above proof has to be modified. One can do so by rather working on its dual $\sd^{s_0} U_V(t)^*\sd^{-s_0}$, which by the (dual) Duhamel formulation satisfies for any $F\in L^{(2p)'}_t L^{(2q)'}_x$,
 \begin{align*}
  \int_0^\ii U_V(t)^* F(t)\,dt &= \int_0^\ii e^{-it\Delta_x}F(t)\,dt + i\int_0^\ii \int_0^t U_V(t')^* V(t')e^{i(t'-t)\Delta}F(t)\,dt'\,dt \\
  &= \int_0^\ii e^{-it\Delta_x}F(t)\,dt + i\int_0^\ii U_V(t')^* V(t') \left(\int_{t'}^\ii e^{i(t'-t)\Delta}F(t)\,dt\right) \,dt'.
 \end{align*}
 With the same estimates as in the proof of Proposition \ref{prop:UV-Sonae}, one can use a fixed point argument in $\cB(L^{r'}_t W^{s_0,c'}_x\to H^{s_0}_x)$ to construct $U_V(t)^*$ (where $(r,c)$ is given in the proof of Proposition \ref{prop:UV-Sonae}).
\end{remark}

\begin{remark}
 As explained in the proof of Lemma 3.9 in \cite{Hadama-23}, the smallness condition $\|V\|_{L^\mu_t W^{s,\nu}_x}\le\delta$ in Proposition \ref{prop:UV-Sonae} can be removed by iterating this result on a finite number of time intervals where we have the smallness condition. If $\|V\|_{L^\mu_t W^{s,\nu}_x}\le R$ for some $R>0$, the constant $C>0$ appearing in Proposition \ref{prop:UV-Sonae} then depends only on $R$.
\end{remark}

\begin{proposition}\label{prop:UV-bez}

 Let $d\ge2$, $s\ge0$, and $\mu,\nu\in(1,+\ii)$. Define $\zeta:=2/\mu+d/\nu-2$ and assume that $0\le\zeta\le s$, $1/\nu<\min(1/2+1/d,(\zeta+2)/(2(\zeta+1)))$, and $(\zeta+2)/(2(d+1))<1/\mu<(\zeta+2)/(2(\zeta+1))$.
 Let $s_0\in[-s,s]$ and $\sigma\ge0$. Let $p,q,\alpha\in(1,+\ii)$ be such that
 $$\frac2p+\frac dq=d-2\sigma,\quad\frac{1}{\alpha}\ge\frac{1}{dp}+\frac1q,\quad \alpha<p.$$
 Let $m:\R^d\to\R$ be such that $|m(\xi)|\lesssim\langle\xi\rangle^{s_0-\sigma}$. Then, there exists $\delta>0$ and $C>0$ such that for any $V,\tilde{V}\in L^\mu_t W^{s,\nu}_x$ such that $\|V\|_{L^\mu_t W^{s,\nu}_x}\le\delta$, $\|\tilde{V}\|_{L^\mu_t W^{s,\nu}_x}\le\delta$, and any $g\in L^{2p'}_t(\R_+,L^{2q'}_x)$ we have
 $$\|g(t) m(-i\nabla_x)U_V(t)\sd^{-s_0}\|_{\gS^{2\alpha'}(L^2_x\to L^2(\R_+,L^2_x))}\le C\|g\|_{L^{2p'}_t(\R_+,L^{2q'}_x)},$$
 $$\|g(t) m(-i\nabla_x)(U_V(t)-U_{\tilde{V}}(t))\sd^{-s_0}\|_{\gS^{2\alpha'}(L^2_x\to L^2(\R_+,L^2_x))}\le C\|g\|_{L^{2p'}_t(\R_+,L^{2q'}_x)}\|V-\tilde{V}\|_{L^\mu_t W^{s,\nu}_x}.$$
\end{proposition}

To prove Proposition \ref{prop:UV-bez} we will need the following version of fractional Leibniz rule, whose proof will be done after the one of Proposition \ref{prop:UV-bez}.

\begin{lemma}\label{lem:leibniz-without-loss}
 Let $d\ge2$, $s_0\ge0$ and $\mu,\nu_0\in(1,+\ii)$. Define $\zeta_0:=2/\mu+d/\nu_0-2$ and assume that $0\le\zeta_0\le s_0$, $(\zeta_0+2)/(2(d+1))<1/\mu<(\zeta_0+2)/(2(\zeta_0+1))$, $1/\nu_0<(\zeta_0+2)/(2(\zeta_0+1))$. Then, there exists $N\in\N\setminus\{0\}$ and for all $n=1,\ldots,N$ and $\ell\in\{1,2\}$ there exist  $p_n^{(\ell)},q_n^{(\ell)}\in(1,+\ii)$, $J_n^{(\ell)}:B_n^{(\ell)}\to\cB(L^{2q_n^{(\ell)}}_x,L^2(X_n))$ linear and continuous, where $B_n^{(\ell)}$ is some Banach space and $X_n$ some measure space, $g_n^{(\ell)}: W^{s_0,\nu_0}_x\to B_n^{(\ell)}$, and $\theta_n^{(\ell)}\in[0,1]$ with $\theta_n^{(1)}+\theta_n^{(2)}=1$ such that for all $V\in L^\mu_t W^{s_0,\nu_0}_x$ one can decompose
 \begin{equation}\label{eq:decomp-leibniz-schatten-J}
    \sd^{s_0} V(t)\sd^{-s_0} = \sum_{n=1}^N J_n^{(1)}(g_n^{(1)}(V(t)))^*J_n^{(2)}(g_n^{(2)}(V(t))),
 \end{equation}
with for all $n=1,\ldots,N$ and all $\ell\in\{1,2\}$,
 \begin{equation}\label{eq:est-gnj}
    \|g_n^{(\ell)}(V(t))\|_{L^{2(p_n^{(\ell)})'}_t B_n^{(\ell)}}\lesssim \|V\|_{L^\mu_t W^{s_0,\nu_0}_x}^{\theta_n^{(\ell)}}
 \end{equation}
 $$\frac{2}{p_n^{(\ell)}}+\frac{d}{q_n^{(\ell)}}=d,$$
 and defining $\alpha_n^{(\ell)}$ by $1/\alpha_n^{(\ell)}=1/(dp_n^{(\ell)})+1/q_n^{(\ell)}$, we have
 $$\alpha_n^{(\ell)}<p_n^{(\ell)},\quad\frac{1}{2(\alpha_n^{(1)})'}+\frac{1}{2(\alpha_n^{(2)})'}=\frac{1}{d\mu'}.$$
 In particular, note that
 \begin{equation}\label{eq:est-Jng}
    \|J_n^{(\ell)}(g(t))\|_{L^{2(p_n^{(\ell)})'}_t\cB(L^{2q_n^{(\ell)}}_x\to L^2(X_n))}\lesssim \|g\|_{L^{2(p_n^{(\ell)})'}_t B_n^{(\ell)}}.
 \end{equation}

\end{lemma}

\begin{proof}[Proof of Proposition \ref{prop:UV-bez} assuming Lemma \ref{lem:leibniz-without-loss}]
 Assume first that $s_0\ge0$. Notice that the parameters $(d,s,\mu,\nu,s_0)$ satisfy the assumptions of Proposition \ref{prop:UV-Sonae}: the only non-explicit condition is $\zeta<d/2$, which follow from the assumptions $1/\mu,1/\nu<(\zeta+2)/(2(\zeta+1))$. Let $g\in L^{2p'}_t(\R_+,L^{2q'}_x)$ and define
 $$\cU(t):=\sd^{s_0} U_V(t)\sd^{-s_0}.$$
 By the Duhamel formula, we have
$$
\begin{multlined}
 g(t)m(-i\nabla_x)U_V(t)\sd^{-s_0} = g(t)e^{it\Delta_x}m(-i\nabla_x)\sd^{-s_0} \\
 -i g(t)e^{it\Delta_x}m(-i\nabla_x)\sd^{-s_0}\int_0^t e^{-it_1\Delta}\sd^{s_0} V(t_1)\sd^{-s_0}\cU(t_1)\,dt_1.
\end{multlined}
$$
As in the proof of Proposition \ref{prop:UV-Sonae}, let $\nu_0\in(1,+\ii)$ be such that
$$2\le\frac{2}{\mu}+\frac{d}{\nu_0}\le2+s_0,\
\frac{1}{\nu_0}<\frac12+\frac1d,\
\frac{2}{\mu}+\frac{d}{\nu_0}<2+\frac d2,\
\frac{1}{\nu}-\frac{s-s_0}{d}\le\frac{1}{\nu_0}\le\frac{1}{\nu}.
$$
In particular, we have $W^{s,\nu}_x\hookrightarrow W^{s_0,\nu_0}_x$. Define $\zeta_0:=2/\mu+d/\nu_0-2$. Since $1/\nu_0\le1/\nu$, we have $\zeta_0\le\zeta$.
Since $(\zeta,\nu)$ satisfies that
$$\frac1\nu<\frac{\zeta+2}{2(\zeta+1)},\quad \frac{\zeta+2}{2(d+1)}<\frac1\mu<\frac{\zeta+2}{2(\zeta+1)},$$
and since $1/\nu_0\le1/\nu$ and $\zeta_0\le\zeta$, we have by monotoniticy that
$$\frac{1}{\nu_0}<\frac{\zeta_0+2}{2(\zeta_0+1)},\quad \frac{\zeta_0+2}{2(d+1)}<\frac1\mu<\frac{\zeta_0+2}{2(\zeta_0+1)}.$$
We can thus apply Lemma \ref{lem:leibniz-without-loss} with the parameters $(d,s_0,\mu,\nu_0)$ to $V\in L^\mu_t W^{s_0,\nu_0}_x$ to infer that
$$
\begin{multlined}
 g(t)e^{it\Delta_x}m(-i\nabla_x)\sd^{-s_0}\int_0^t e^{-it_1\Delta}\sd^{s_0} V(t_1)\sd^{-s_0}\cU(t_1)\,dt_1 \\
 =\sum_{n=1}^N [g(t)e^{it\Delta_x}m(-i\nabla_x)\sd^{-s_0}(J_n^{(1)}(g_n^{(1)}(V(t_1)))e^{it_1\Delta})^*]_<J_n^{(2)}(g_n^{(2)}(V(t_1)))\cU(t_1),
\end{multlined}
$$
where the notation $[\cdot]_<$ is the same as in Corollary \ref{coro:CK}. By \eqref{eq:stri-bez-xtotx}, we have
\begin{align*}
\|g(t)e^{it\Delta_x}m(-i\nabla_x)\sd^{-s_0}\|_{\gS^{2\alpha'}(L^2_x\to L^2_t(\R_+,L^2_x))}&\lesssim\|g(t)e^{it\Delta_x}\sd^{-\sigma}\|_{\gS^{2\alpha'}(L^2_x\to L^2_t(\R_+,L^2_x))}\\
&\lesssim \|g\|_{L^{2p'}_t(\R_+,L^{2q'}_x)},
\end{align*}
and combining Corollary \ref{coro:general-schatten} and \eqref{eq:est-Jng}, we have
$$\|J_n^{(1)}(g_1)e^{it_1\Delta}\|_{L^2_x\to L^2_{t_1}(\R_+,L^2(X_n))} \lesssim  \|g_1\|_{L^{2(p_n^{(1)})'}_{t_1}(\R_+,B_n^{(1)})}.$$
Both Schatten exponents are bigger than the Lebesgue time exponent on the right side. Hence, by Corollary \ref{coro:CK} we deduce that
\begin{align*}
    \|[g(t)e^{it\Delta_x}m(-i\nabla_x)\sd^{-s_0}(J_n^{(1)}(g_n^{(1)}(V(t_1)))&e^{it_1\Delta})^*]_<\|_{\gS^{2\alpha'}(L^2_{t_1}(\R_+,L^2(X_n))\to L^2_t(\R_+,L^2_x))}\\
    &\lesssim \|g\|_{L^{2p'}_t(\R_+,L^{2q'}_x)}\|g_n^{(1)}(V(t_1))\|_{L^{2(p_n^{(1)})'}_{t_1}(\R_+, B_n^{(1)})} \\
    &\lesssim \|g\|_{L^{2p'}_t(\R_+,L^{2q'}_x)} \|V\|_{L^\mu_t(\R_+,W^{s,\nu}_x)}^{\theta_n^{(1)}}
\end{align*}
where in the last inequality we used \eqref{eq:est-gnj} and $W^{s,\nu}_x\hookrightarrow W^{s_0,\nu_0}_x$. Combining Proposition \ref{prop:UV-Sonae}, Lemma \ref{lem:strichartz-A}, \eqref{eq:est-Jng}, and \eqref{eq:est-gnj} we also have
$$
 \|J_n^{(2)}(g_n^{(2)}(V(t_1)))\cU(t_1)\|_{L^2_x\to L^2_{t_1}(\R_+,L^2(X_n))} \lesssim \|V\|_{L^\mu_t(\R_+,W^{s,\nu}_x)}^{\theta_n^{(2)}}.
$$
We deduce that
$$\|g(t)m(-i\nabla_x)U_V(t)\sd^{-s_0}\|_{\gS^{2\alpha'}(L^2_x\to L^2_t(\R_+,L^2_x))}\lesssim \|g\|_{L^{2p'}_t(\R_+, L^{2q'}_x)},$$
which is the first desired result. The estimate for $U_V-U_{\tilde{V}}$ follows from the same argument, using the same identity as in the end of the proof of Proposition \ref{prop:UV-Sonae}. Finally, in the case $s_0\le0$, one can use the same argument by working on $U_V(t)^*$ as in Remark \ref{rk:negative-sobolev}.
\end{proof}

\begin{proof}[Proof of Lemma \ref{lem:leibniz-without-loss}]
For any $a\in\{0,1,\ldots,\lfloor s_0\rfloor,s_0\}$, define $z_a:=\min(\zeta_0,a)$ and
$$\frac{1}{(q_a^{(1)})'}:=\frac2d\left(1-\frac{2}{\mu(z_a+2)}\right),\quad \frac{1}{(q_a^{(2)})'}:=\frac2d\left(1-\frac{2(z_a+1)}{\mu(z_a+2)}\right).$$
Since $2/\mu<(\zeta_0+2)/(\zeta_0+1)$ and $z_a\le\zeta_0$, we deduce that $0<1/(q_a^{(\ell)})'<2/d$ and thus defining $p_a^{(\ell)}$ by the relation
$$\frac{2}{p_a^{(\ell)}}+\frac{d}{q_a^{(\ell)}}=d,$$
we have $p_a^{(\ell)}\in(1,+\ii)$. We also define
$$\frac{1}{(\tilde{q_a}^{(2)})'}:=\frac{1}{(q_a^{(2)})'}+\frac{2z_a}{d}>0.$$
Notice that we have
$$\frac{1}{q_a}:=\frac{1}{2(q_a^{(1)})'}+\frac{1}{2(\tilde{q_a}^{(2)})'}=\frac{1}{\nu_0}-\frac{\zeta_0-z_a}{d}\in\left[\frac{1}{\nu_0}-\frac{s_0-a}{d}, \frac{1}{\nu_0}\right].$$
The condition $1/\nu_0<(\zeta_0+2)/(2(\zeta_0+1))$ implies that $\tilde{q_a}^{(2)}\in(1,+\ii)$: $1/(\tilde{q_a}^{(2)})'$ is an increasing function of $z_a$ and its value at $z_a=\zeta_0$ is $<1$. We thus apply Lemma \ref{lem:leibniz-decomp-general} to the parameters $(q_a^{(1)},\tilde{q_a}^{(2)})$, and we take $J_a^{(1)}:=L_a^{(1)}$,  $J_a^{(2)}:=L_a^{(2)}\sd^{-a}$. Notice that if $\theta_a^{(1)}:=q_a/(2(q_a^{(1)})')\in[0,1]$ and $\theta_a^{(2)}:=q_a/(2(\tilde{q_a}^{(2)})')\in[0,1]$, we have $\theta_a^{(1)}+\theta_a^{(2)}=1$ and $2(p_a^{(\ell)})'\theta_a^{(\ell)}=\mu$, so that
$$\|g_a^{(\ell)}(V(t))\|_{L^{2(p_a^{(\ell)})'}_t B_a^{(\ell)}}\lesssim \|V\|_{L^\mu_t W^{s,\nu}_x}^{\theta_a^{(\ell)}}.$$
Furthermore, since
$$\|J_a^{(1)}(g)\|_{L^{2q_a^{(1)}}_x\to L^2(X_a)} =\|L_a^{(1)}(g)\|_{L^{2q_a^{(1)}}_x\to L^2(X_a)} \lesssim \|g\|_{B_a^{(1)}},$$
we deduce that
$$\|J_a^{(1)}(g(t))\|_{L^{2(p_a^{(1)})'}_t\cB(L^{2q_a^{(1)}}_x\to L^2(X_a))} \lesssim \|g\|_{L^{2(p_a^{(1)})'}_tB_a^{(1)}}.$$
Furthermore, since $W^{z_a,2q_a^{(2)}}_x \hookrightarrow L^{2\tilde{q_a}^{(2)}}_x$ and $z_a\le a$, we have
\begin{align*}
 \|J_a^{(2)}(g)\|_{L^{2q_a^{(2)}}_x\to L^2(X_a)} &= \|L_a^{(2)}(g)\sd^{-a}\|_{L^{2q_a^{(2)}}_x\to L^2(X_a)} \\
 &= \|L_a^{(2)}(g)\|_{W^{a,2q_a^{(2)}}_x\to L^2(X_a)} \\
 &\lesssim \|L_a^{(2)}(g)\|_{L^{2\tilde{q_a}^{(2)}}_x\to L^2(X_a)} \\
 &\lesssim \|g\|_{B_a^{(2)}},
\end{align*}
and hence
$$\|J_a^{(2)}(g(t))\|_{L^{2(p_a^{(2)})'}_t\cB(L^{2q_a^{(2)}}_x\to L^2(X_a))} \lesssim \|g\|_{L^{2(p_a^{(2)})'}_tB_a^{(2)}}.$$
Finally, defining $\alpha_a^{(\ell)}$ by
$$\frac{1}{\alpha_a^{(\ell)}}=\frac{1}{dp_a^{(\ell)}}+\frac{1}{q_a^{(\ell)}},$$
we have $(\alpha_a^{(\ell)})'=2(q_a^{(\ell)})'$ and hence
$$\frac{1}{2(\alpha_a^{(1)})'}+\frac{1}{2(\alpha_a^{(2)})'}=\frac{1}{d\mu'}.$$
The condition $1/\mu>(\zeta+2)/(2(d+1))$ ensures that $\alpha_a^{(\ell)}<p_a^{(\ell)}$.

\end{proof}

\begin{corollary}
  Let $d\ge2$, $s\ge0$, and $\mu,\nu\in(1,+\ii)$. Define $\zeta:=2/\mu+d/\nu-2$ and assume that $0\le\zeta\le s$, $1/\nu<\min(1/2+1/d,(\zeta+2)/(2(\zeta+1)))$, and $(\zeta+2)/(2(d+1))<1/\mu<(\zeta+2)/(2(\zeta+1))$. Let $s_0\in[0,s]$. Let $p,q,\alpha\in(1,+\ii)$ be such that $1/\alpha\ge1/(dp)+1/q$, $\alpha<p$. Let $\sigma_1,\sigma_2\ge0$ be such that $\sigma_1+\sigma_2=d-2/p-d/q$ and
 $$\sigma_1,\sigma_2<\min(d/2,d/2-2/p+1).$$
  Let $m_1,m_2:\R^d\to\R$ be such that $|m_j(\xi)|\lesssim\langle\xi\rangle^{s_0-\sigma_j}$
 Then, there exists $\delta>0$ and $C>0$ such that for any $V_1,V_2,V_3\in L^\mu_t W^{s,\nu}_x$ such that $\|V_j\|_{L^\mu_t W^{s,\nu}_x}\le\delta$, and any $\gamma\in\cH^{s_0,\alpha}$ we have
 \begin{equation}\label{eq:bez-UV-m1m2}
  \|\rho_{m_1(-i\nabla_x)U_{V_1}(t)\gamma U_{V_1}(t)^*m_2(-i\nabla_x)}\|_{L^p_t L^q_x} \lesssim \|\sd^{s_0}\gamma\sd^{s_0}\|_{\gS^\alpha},
 \end{equation}
 \begin{equation}
  \|\rho_{m_1(-i\nabla_x)U_{V_1}(t)\gamma (U_{V_2}(t)-U_{V_3}(t))^*m_2(-i\nabla_x)}\|_{L^p_t L^q_x} \lesssim \|\sd^{s_0}\gamma\sd^{s_0}\|_{\gS^\alpha}\|V_2-V_3\|_{L^\mu_tW^{s,\nu}_x}.
 \end{equation}

\end{corollary}

\begin{proof}
 The proof is the same as the one of Corollary \ref{coro:deriv-distrib}, starting from Proposition \ref{prop:UV-bez}.
\end{proof}

 \begin{theorem}\label{thm:stri-potential}

 Let $d\ge2$, $s\ge0$ and $\mu,\nu\in(1,+\ii)$. Define $\zeta:=2/\mu+d/\nu-2$ and assume that $0\le\zeta\le s$, $1/\nu<\min(1/2+1/d,(\zeta+2)/(2(\zeta+1)))$, and $(\zeta+2)/(2(d+1))<1/\mu<(\zeta+2)/(2(\zeta+1))$. Let $p,q,\alpha\in(1,+\ii)$ be such that
 $$\frac 2p+\frac dq=d-s,\ s<\min\left(\frac d2, \frac d2 - \frac2p + 1 \right),\ \frac 1\alpha \ge \frac{1}{dp}+\frac 1q,\ \alpha<p.$$
Then, there exists $\delta>0$ and $C>0$ such that for any $V_1,V_2,V_3\in L^\mu_t W^{s,\nu}_x$ such that $\|V_j\|_{L^\mu_t W^{s,\nu}_x}\le\delta$ and for all $\gamma\in\cH^{s,\alpha}$ we have
 $$\| \rho_{U_{V_1}(t)\gamma U_{V_1}(t)^*} \|_{L^p_t W^{s,q}_x} \le C\| \langle \nabla\rangle^s \gamma \langle\nabla\rangle^s\|_{\gS^\alpha},$$
 $$\| \rho_{U_{V_1}(t)\gamma (U_{V_2}(t)-U_{V_3}(t))^*} \|_{L^p_t W^{s,q}_x} \le C\| \langle \nabla\rangle^s \gamma \langle\nabla\rangle^s\|_{\gS^\alpha}\|V_2-V_3\|_{L^\mu_t W^{s,\nu}_x}.$$

\end{theorem}

\begin{proof}
 Let us write $V=V_1$ for brevity. By \eqref{eq:bez-UV-m1m2} applied to $m_j=1$ and $\sigma_j=s/2=s_0$, we have
 $$\|\rho_{U_V(t)\gamma U_V(t)^*}\|_{L^p_t L^q_x} \le C\|\sd^{s/2}\gamma\sd^{s/2}\|_{\gS^\alpha}\le C\|\sd^{s}\gamma\sd^{s}\|_{\gS^\alpha}.$$
It thus remains to estimate $|D|^s\rho_{U_V(t)\gamma U_V(t)^*}$ in $L^p_t L^q_x$. Let $W\in L^{p'}_t L^{q'}_x$. By Lemma \ref{lem:leibniz-gamma} applied with $\sigma_1=\sigma_2=s/2$ and $q_1=q_2=2q$, we have
\begin{multline*}
 \int W(t,x)(|D_x|^s\rho_{U_V(t)\gamma U_V(t)^*})(x)\,dx\,dt = \sum_{|\alpha|\le s/2}\frac{1}{\alpha!}\int W(t,x)\rho_{\partial^\alpha U_V(t)\gamma U_V(t)^*D^{s ,\alpha}}(x)\,dx\,dt\\ +\sum_{|\beta|\le s/2}\frac{1}{\beta!}\int W(t,x)\rho_{ D^{s ,\beta}U_V(t)\gamma U_V(t)^*\partial^\beta}(x)\,dx\,dt
 \\
 +\tr_{L^2_x} \left(\int_\R U_V(t)^*A^{(2)}(h^{(2)}(W(t)))^*A^{(1)}(h^{(1)}(W(t)))U_V(t)\,dt\right)\gamma
\end{multline*}
For any $|\alpha|\le s/2$ we have, applying \eqref{eq:bez-UV-m1m2} to $m_1(\xi)=\xi^\alpha$, $m_2(\xi)=\partial^\alpha(|\xi|^s)$, $s_0=s$, $\sigma_1=s-|\alpha|$, $\sigma_2=|\alpha|$,
\begin{align*}
 \left| \int W(t,x)\rho_{\partial^\alpha U_V(t)\gamma U_V(t)^*D^{s ,\alpha}}(x)\,dx\,dt \right| &\le \|W\|_{L^{p'}_t L^{q'}_x}\|\rho_{\partial^\alpha U_V(t)\gamma U_V(t)^*D^{s ,\alpha}}\|_{L^p_tL^q_x}\\
 &\le C\|W\|_{L^{p'}_t L^{q'}_x}\|\sd^{s}\gamma\sd^{s}\|_{\gS^\alpha}.
\end{align*}
A similar estimate holds for the terms in the sum over $\beta$. For the last term, we use the estimate
\begin{multline*}
 \left|\tr_{L^2_x} \left(\int_\R U_V(t)^*A^{(2)}(h^{(2)}(W(t)))^*A^{(1)}(h^{(1)}(W(t)))U_V(t)\,dt\right)\gamma \right| \\
 \le \prod_{j=1}^2\left\| A^{(j)}(h^{(j)}(W(t)))U_V(t)\sd^{-s}\right\|_{\gS^{2\alpha'}(L^2_x\to L^2_{t,x})}\|\sd^{s}\gamma\sd^{s}\|_{\gS^\alpha}
\end{multline*}
Defining $\cU(t):=\sd^{s/2}U_V(t)\sd^{-s}$, we have by Proposition \ref{prop:UV-bez} applied to $m(\xi)=\langle\xi\rangle^{s/2}$, $\sigma=s/2$, and $s_0=s$ that
$$\|g(t)\cU(t)\|_{\gS^{2\alpha'}(L^2_x\to L^2(\R_+,L^2_x))}\lesssim \|g\|_{L^{2p'}_t(\R_+,L^{2q'}_x)},$$
hence applying Lemma \ref{lem:strichartz-A} we have for $j=1,2$
\begin{align*}
 \left\| A^{(j)}(h^{(j)}(W(t)))U_V(t)\sd^{-s}\right\|_{\gS^{2\alpha'}(L^2_x\to L^2_{t,x})}^{2p'} &\le C\int_\R\|A^{(j)}(h^{(j)}(W(t)))\sd^{-s/2}\|_{L^{2q}_x\to L^2(X)}^{2p'}\,dt \\
 &\le C\int_\R\|A^{(j)}(h^{(j)}(W(t)))\|_{W^{s/2,2q}_x\to L^2(X)}^{2p'}\,dt \\
 &\le C\int_\R \|W(t)\|_{L^{q'}_x}^{p'}\,dt.
\end{align*}
Regrouping all the previous estimates, we deduce that
$$\||D|^s\rho_{U_V(t)\gamma U_V(t)^*}\|_{L^p_t L^q_x} \le C\|\sd^{s}\gamma\sd^{s}\|_{\gS^\alpha},$$
which proves the result.
\end{proof}

We will apply Theorem \ref{thm:stri-potential} to $s=d/2-1$, $\mu=2=\nu$, $p=2=q$, and $\alpha=2d/(d+1)$. We obtain
$$\| \rho_{U_V(t)\gamma U_V(t)^*} \|_{L^2_t H^s_x} \le C\| \langle \nabla\rangle^s \gamma \langle\nabla\rangle^s\|_{\gS^{2d/(d+1)}}$$
if $\|V\|_{L^2_t H^s_x}$ is small enough and if $s=d/2-1$.

\section{Estimating the reaction of the background}\label{sec:reaction}

  We now estimate the term
  $$
  \begin{multlined}
   \rho(V)(t):=\rho\left[ -i\int_0^te^{i(t-\tau)\Delta_x}[V(\tau),g(-i\nabla_x)]D_V(t,\tau)^*\,d\tau \right]\\
   +\rho\left[ -i\int_0^tD_V(t,\tau)[V(\tau),g(-i\nabla_x)]e^{i(\tau-t)\Delta_x}\,d\tau \right]\\
   +\rho\left[ -i\int_0^tD_V(t,\tau)[V(\tau),g(-i\nabla_x)]D_V(t,\tau)^*\,d\tau \right]
  \end{multlined}
  $$
  that arise in \eqref{eq:Phi}.
  \begin{theorem}\label{thm:reaction}
   Let $d\ge2$ and $s=d/2-1$. If $\|V\|_{L^2_t H^s_x}+\|\tilde{V}\|_{L^2_t H^s_x}\le\delta$, then
   $$\left\| \rho(V) \right\|_{L^2_t H^s_x} \le C,$$
   $$\left\| \rho(V) - \rho(\tilde{V}) \right\|_{L^2_t H^s_x} \le \frac12 \|V-\tilde{V}\|_{L^2_t H^s_x}.$$
  \end{theorem}
  To prove Theorem \ref{thm:reaction}, we write the three terms composing $\rho(V)$ as
  $$\rho(V)=\bar{\rho_2(V)}+\rho_2(V)+\rho_3(V),$$
  and we estimate $\rho_2(V)$ and $\rho_3(V)$. Expanding the commutator in the second term, we have
  \begin{multline*}
   \rho_2(V)(t)=\rho\left[ -i\int_0^tD_V(t,\tau)[V(\tau),g(-i\nabla_x)]e^{i(\tau-t)\Delta_x}\,d\tau \right] \\
   = \rho\left[ -i\int_0^tD_V(t,\tau)V(\tau)g(-i\nabla_x)e^{i(\tau-t)\Delta_x}\,d\tau \right]-\rho\left[ -i\int_0^tD_V(t,\tau)g(-i\nabla_x)V(\tau)e^{i(\tau-t)\Delta_x}\,d\tau \right].
  \end{multline*}
  Recall that $D_V$ was defined as
  $$D_V(t,\tau)=U_V(t,\tau)-e^{i(t-\tau)\Delta_x}=-i\int_\tau^t e^{i(t-t_1)\Delta_x}V(t_1)U_V(t_1,\tau)\,dt_1.$$
  Taking the adjoint of this relation and exchanging $t$ and $\tau$ we also have that
  $$D_V(t,\tau)=-i\int_\tau^t U_V(t,t_1)V(t_1)e^{i(t_1-\tau)\Delta}\,dt_1.$$
  Using these relations, we obtain that
  \begin{multline*}
   \rho\left[ -i\int_0^tD_V(t,\tau)V(\tau)g(-i\nabla_x)e^{i(\tau-t)\Delta_x}\,d\tau \right] \\
   =-\rho\left[\int_0^t\,dt_1\int_0^{t_1}\,dt_2\, e^{i(t-t_1)\Delta_x}V(t_1)U_V(t_1,t_2)V(t_2)g(-i\nabla_x)e^{i(t_2-t)\Delta}\right],
  \end{multline*}
  \begin{multline*}
   \rho\left[ -i\int_0^tD_V(t,\tau)g(-i\nabla_x)V(\tau)e^{i(\tau-t)\Delta_x}\,d\tau \right] \\
   =-\rho\left[\int_0^t\,dt_1\int_0^{t_1}\,dt_2\, U_V(t,t_1)V(t_1)e^{i(t_1-t_2)\Delta}g(-i\nabla_x)V(t_2)e^{i(t_2-t)\Delta}\right].
  \end{multline*}
  To extract the most singular terms in this expression, we expand $U_V$ once more:
  \begin{multline*}
   \rho\left[ -i\int_0^tD_V(t,\tau)V(\tau)g(-i\nabla_x)e^{i(\tau-t)\Delta_x}\,d\tau \right] \\
   =-\rho\left[\int_0^t\,dt_1\int_0^{t_1}\,dt_2\, e^{i(t-t_1)\Delta_x} V(t_1)e^{i(t_1-t_2)\Delta}V(t_2)g(-i\nabla_x)e^{i(t_2-t)\Delta}\right]\\
   +\rho\left[i\int_0^t\,dt_1\int_0^{t_1}\,dt_2\int_0^{t_2}\,dt_3\, e^{i(t-t_1)\Delta_x} V(t_1)U_V(t_1,t_2)V(t_2)e^{i(t_2-t_3)\Delta}V(t_3)g(-i\nabla_x)e^{i(t_3-t)\Delta}\right],
  \end{multline*}
  \begin{multline*}
   \rho\left[ -i\int_0^tD_V(t,\tau)g(-i\nabla_x)V(\tau)e^{i(\tau-t)\Delta_x}\,d\tau \right] \\
   =-\rho\left[\int_0^t\,dt_1\int_0^{t_1}\,dt_2\, e^{i(t-t_1)\Delta_x} V(t_1)e^{i(t_1-t_2)\Delta}g(-i\nabla_x)V(t_2)e^{i(t_2-t)\Delta}\right]\\
   +\rho\left[i\int_0^t\,dt_1\int_0^{t_1}\,dt_2\int_0^{t_2}\,dt_3\, e^{i(t-t_1)\Delta_x} V(t_1)U_V(t_1,t_2)V(t_2)e^{i(t_2-t_3)\Delta}g(-i\nabla_x)V(t_3)e^{i(t_3-t)\Delta}\right],
  \end{multline*}
  The terms with only free propagators (and no $U_V$) are the most singular and will be treated separately. For the other two terms, what is important in their expression is that the time integrals are ordered and that $g(-i\nabla_x)$ has a free propagator next to it (and not a $U_V$). To sum up, we have decomposed
  $$\rho_2(V)=-\rho_{2,1}(V)+i\rho_{3,1}(V)+\rho_{2,2}(V)-i\rho_{3,2}(V),$$
  with
   $$\rho_{2,1}(V)(t) := \rho\left[\int_0^t\,dt_1\int_0^{t_1}\,dt_2\, e^{i(t-t_1)\Delta_x} V(t_1)e^{i(t_1-t_2)\Delta}g(-i\nabla_x)V(t_2)e^{i(t_2-t)\Delta}\right],$$
   $$\rho_{2,2}(V)(t) := \rho\left[\int_0^t\,dt_1\int_0^{t_1}\,dt_2\, e^{i(t-t_1)\Delta_x} V(t_1)e^{i(t_1-t_2)\Delta}V(t_2)g(-i\nabla_x)e^{i(t_2-t)\Delta}\right],$$
   $$
   \begin{multlined}
   \rho_{3,1}(V)(t):=\\
   \rho\left[\int_0^t\,dt_1\int_0^{t_1}\,dt_2\int_0^{t_2}\,dt_3\, e^{i(t-t_1)\Delta_x} V(t_1)U_V(t_1,t_2)V(t_2)e^{i(t_2-t_3)\Delta}V(t_3)g(-i\nabla_x)e^{i(t_3-t)\Delta}\right],
   \end{multlined}
   $$
   $$
   \begin{multlined}
   \rho_{3,2}(V)(t):=\\
   \rho\left[\int_0^t\,dt_1\int_0^{t_1}\,dt_2\int_0^{t_2}\,dt_3\, e^{i(t-t_1)\Delta_x} V(t_1)U_V(t_1,t_2)V(t_2)e^{i(t_2-t_3)\Delta}g(-i\nabla_x)V(t_3)e^{i(t_3-t)\Delta}\right].
   \end{multlined}
   $$
   We estimate $\rho_{2,1}$ in Proposition \ref{prop:rho21}, $\rho_{2,2}$ in Proposition \ref{prop:rho22}, $\rho_{3,1}$ in Proposition \ref{prop:rho31}, $\rho_{3,2}$ in Proposition \ref{prop:rho32}, and $\rho_{3}$ in Proposition \ref{prop:rho33}. Together, they imply Theorem \ref{thm:reaction}.

  \subsection{The less singular terms}

  In this section, we estimate $\rho_{3,1}$, $\rho_{3,2}$, and $\rho_3$. They all have in common that three $V(t)$ appear in their definition, which makes them the easiest to estimate.

  \begin{proposition}\label{prop:rho31}

   Let $d\ge2$ and $s=d/2-1$. If $\|V\|_{L^2_t H^s_x}+\|\tilde{V}\|_{L^2_t H^s_x}\le\delta$, then
   $$\left\| \rho_{3,1}(V) \right\|_{L^2_t H^s_x} \le C,$$
   $$\left\| \rho_{3,1}(V) - \rho_{3,1}(\tilde{V}) \right\|_{L^2_t H^s_x} \le \frac12 \|V-\tilde{V}\|_{L^2_t H^s_x}.$$

\end{proposition}

\begin{lemma}\label{lem:leibniz-op-2}

Let $d\ge2$ and $s=d/2-1$. Then, there exist $L^{(1)}:B^{(1)}\to\cB(L^{2+4/d}_x,L^2(X))$, $L^{(2)}:B^{(2)}\to\cB(L^{d+2}_x,L^2(X))$ linear and continuous, where $X$ is some measure space, and for $\ell\in\{1,2\}$, $B^{(\ell)}$ are some Banach spaces and $g^{(\ell)}: H^s_x\to B^{(\ell)}$ are such that for all $W\in H^s_x$ one can decompose
 \begin{equation}\label{eq:decomp-leibniz-schatten}
    \nsd^s  W\nsd^{-s} = L^{(1)}(g^{(1)}(W))^* L^{(2)}(g^{(2)}(W)),
 \end{equation}
 \begin{equation}\label{eq:est-gnL}
  \|g^{(1)}(W)\|_{B^{(1)}}\lesssim \|W\|_{H^s_x}^{\tfrac{2}{d+2}},\quad \|g^{(2)}(W)\|_{B^{(2)}}\lesssim \|W\|_{H^s_x}^{\tfrac{d}{d+2}}.
\end{equation}
In particular, we have
\begin{equation}\label{eq:est-Lg}
\begin{aligned}
  \|L^{(1)}(g(t))e^{it\Delta_x}\|_{\gS^{2(d+2)}(L^2_x\to L^2_{t}L^2(X))}&\lesssim \|g\|_{L^{d+2}_t B^{(1)}},\\
  \|L^{(2)}(g(t))e^{it\Delta_x}\sd^{-s}\|_{\gS^{2d(d+2)/(d^2-2)}(L^2_x\to L^2_{t}L^2(X))}&\lesssim \|g\|_{L^{2(d+2)/d}_t B^{(2)}}.
\end{aligned}
\end{equation}

\end{lemma}

\begin{proof}
We apply Lemma \ref{lem:leibniz-decomp-general} with $\nu=2$, $q_a^{(1)}=(d+2)/d$ and $q_a^{(2)}=(d+2)/2$. We also replace $L_a^{(2)}(g)\sd^{-a}$ by $L_a^{(2)}(g)$, using that $\sd^{-a}$ is bounded on $L^{d+2}_x$. Then, \eqref{eq:est-Lg} follows from Corollary \ref{coro:general-schatten} applied to $\sigma=0$, $p=q=1+2/d$, $\alpha=(d+2)/(d+1)$ for $L_n^{(1)}$ and to $\sigma=s$, $p=q=1+d/2$, $\alpha=d(d+2)/(2d+2)$ for $L_n^{(2)}$.
\end{proof}

\begin{lemma}
\label{lem:leibniz-gamma-derivative}
 Let $d\ge2$ and $s=d/2-1$. Then, there exist  $A^{(1)}:B^{(1)}\to\cB(L^{2+4/d}_x,L^2(X))$, $A^{(2)}:B^{(2)}\to\cB(L^{d+2}_x,L^2(X))$ linear and continuous, where $X$ is some measure space, and for $\ell\in\{1,2\}$, $B^{(\ell)}$ are some Banach spaces and $g^{(\ell)}: L^2_x\to B^{(\ell)}$ are such that for all $W\in L^2_{x}$ and all $\gamma$ operator on $L^2_x$ one can decompose
 \begin{equation}
    \int W(x)(\sd^s\rho_{\gamma})(x)\,dx = \tr_{L^2_x}A^{(2)}(g^{(2)}(W(t)))^*A^{(1)}(g^{(1)}(W(t)))\sd^s\gamma\sd^s,
 \end{equation}
with
 \begin{equation}
  \label{eq:est-g-gamma-derivative}
    \|g^{(1)}(W)\|_{B^{(1)}}\lesssim \|W\|_{L^2_x}^{\tfrac{2}{d+2}},\quad \|g^{(2)}(W)\|_{B^{(2)}}\lesssim \|W\|_{L^2_x}^{\tfrac{d}{d+2}}.
 \end{equation}
 In particular, note that
 \begin{equation}
 \label{eq:est-Ag-gamma-derivative}
    \begin{aligned}
  \|A^{(1)}(g(t))e^{it\Delta_x}\|_{\gS^{2(d+2)}(L^2_x\to L^2_{t}L^2(X))}\lesssim \|g\|_{L^{d+2}_t B^{(1)}},\\
  \|A^{(2)}(g(t))e^{it\Delta_x}\sd^{-s}\|_{\gS^{2d(d+2)/(d^2-2)}(L^2_x\to L^2_t L^2(X))}&\lesssim \|g\|_{L^{2+4/d}_t B^{(2)}}.
\end{aligned}
 \end{equation}
\end{lemma}

\begin{proof}
 Let us first show that $\int W(x)(|D_x|^s\rho_{\gamma})(x)\,dx$ can be written in this way. To do so, we apply Lemma \ref{lem:leibniz-gamma} to $\sigma=s$, $\sigma_1=s$, $\sigma_2=0$ (here the choice of $\sigma_1$ and $\sigma_2$ does not matter), $q=2$, $q_1=2(d+2)/d$, and $q_2=d+2$. We obtain
 \begin{multline*}
 \int_{\R^d}W(x)(|D_x|^s\rho_\gamma)(x)\,dx = \sum_{|\alpha|\le s}\frac{1}{\alpha!}\tr_{L^2_x}D^{s,\alpha}W\partial^\alpha\gamma +\tr_{L^2_x} W |D_x|^{s}\gamma \\
 +\tr_{L^2_x}A^{(2)}(h^{(2)}(W))^* A^{(1)}(h^{(1)}(W))\gamma.
\end{multline*}
The last term can be written as
\begin{multline*}
  \tr_{L^2_x}A^{(2)}(h^{(2)}(W))^* A^{(1)}(h^{(1)}(W))\gamma
  \\
  =\tr_{L^2_x}\sd^{-s}A^{(2)}(h^{(2)}(W))^* A^{(1)}(h^{(1)}(W))\sd^{-s}\sd^s\gamma\sd^s,
\end{multline*}
where $A^{(1)}\sd^{-s}:B^{(1)}\to\cB(L^{2(d+2)/d}_x,L^2(X))$ and $A^{(2)}\sd^{-s}:B^{(2)}\to\cB(L^{d+2}_x,L^2(X))$ have the desired properties. We then write
$$\tr_{L^2_x}D^{s,\alpha}W\partial^\alpha\gamma = \tr_{L^2_x} A^{(2)}(g^{(2)}(W(t)))^*A^{(1)}(g^{(1)}(W(t)))\sd^s\gamma\sd^s,$$
with $A^{(1)}(g)=g\partial^\alpha\sd^{-s}$ which is bounded from $L^{2(d+2)/d}_x$ to $L^2_x$ for $g\in B^{(1)}:=L^{d+2}_x$ by the Hölder inequality and the fact that $\partial^\alpha\sd^{-s}$ is bounded on $L^{2(d+2)/d}_x$,  $g^{(1)}(W):= W^{\tfrac{2}{d+2}}\in B^{(1)}$, $A^{(2)}(g)=gD^{s,\alpha}\sd^{-s}$ which is bounded from $L^{d+2}_x$ to $L^2_x$ for $g\in B^{(2)}:=L^{2(d+2)/d}_x$, and $g^{(2)}(W):= W^{\tfrac{d}{d+2}}\in B^{(2)}$. The term $\tr_{L^2_x} W |D_x|^{s}\gamma$ can be treated in the same way. Now that we have proved that $\int W(x)(|D_x|^s\rho_{\gamma})(x)\,dx$ has this form, we notice that $\int W(x)(1+|D_x|^s)\rho_{\gamma})(x)\,dx$ can be written in the same way since $\int W\rho_\gamma = \tr_{L^2_x} W \gamma$ and we can use the same idea. Finally, we use that
$$\int W(x)(\sd^s\rho_{\gamma})(x)\,dx = \int ((1+|D_x|^s)^{-1}\sd^s W)(x)((1+|D_x|^s)\rho_\gamma \,dx,$$
and the fact that $(1+|D_x|^s)^{-1}\sd^s$ is bounded on all $L^q_x$ to deduce the result.
\end{proof}

\begin{proof}[Proof of Proposition \ref{prop:rho31}]
 Similarly as in the proof of Proposition \ref{prop:UV-bez}, we estimate only $\rho_{3,1}$ for positive times and the estimate for negative times is done in the same way. Also, we will just write $L^2_t$ for $L^2_t(\R_+)$ everywhere not to overload notations. Let $W\in L^2_{t,x}$. By Lemma \ref{lem:leibniz-gamma-derivative}, we have
 \begin{multline*}
  \int W(t,x)(\sd^s\rho_{3,1}(V))(t,x) \,dx\,dt \\
  =  \tr_{L^2_x}\int_0^\ii dt\, e^{-it\Delta_x}A^{(2)}(g^{(2)}(W(t)))^*A^{(1)}(g^{(1)}(W(t)))e^{it\Delta_x}\times\\
  \times \int_0^t dt_1\, e^{-it_1\Delta}\sd^s V(t_1)\sd^{-s}\sd^sU_V(t_1)\sd^{-s}\times\\
  \times\int_0^{t_1} dt_2\, \sd^s U_V(t_2)^*\sd^{-s}\sd^s V(t_2)\sd^{-s}e^{it_2\Delta}\times\\
  \times\int_0^{t_2} dt_3\, e^{-it_3\Delta}\sd^s V(t_3)\sd^{-s}\sd^s g(-i\nabla_x)\sd^s e^{it_3\Delta}
 \end{multline*}
 We write $\sd^sV(t_3)\sd^{-s}$ according to Lemma \ref{lem:leibniz-op-2} (since it is next to $g(-i\nabla_x)$ which can absorb an additional $\sd^s$) and $\sd^s V(t_1)\sd^{-s}$, $\sd^s V(t_2)\sd^{-s}$ according to Lemma \ref{lem:leibniz-without-loss} (since they do not have an additional $\sd^{-s}$ around them) applied to $s=d/2-1$ and $\mu=\nu=2$. We obtain
 \begin{multline*}
  \int W(t,x)(\sd^s\rho_{3,1}(V))(t,x) \,dx\,dt \\
  = \sum_{n_1,n_2} \tr_{L^2_x}\int_0^\ii dt\, \sd^{-s}e^{-it\Delta_x}A^{(2)}(g^{(2)}(W(t)))^*A^{(1)}(g^{(1)}(W(t)))e^{it\Delta_x}\times\\
  \times \int_0^t dt_1\, e^{-it_1\Delta}J_{n_1}^{(1)}(g_{n_1}^{(1)}(V(t_1)))^*J_{n_1}^{(2)}(g_{n_1}^{(2)}(V(t_1)))\sd^sU_V(t_1)\sd^{-s}\times\\
  \times\int_0^{t_1} dt_2\, \sd^s U_V(t_2)^*\sd^{-s}J_{n_2}^{(1)}(g_{n_2}^{(1)}(V(t_2)))^*J_{n_2}^{(2)}(g_{n_2}^{(2)}(V(t_2)))e^{it_2\Delta}\times\\
  \times\int_0^{t_2} dt_3\, e^{-it_3\Delta}L^{(1)}(g^{(1)}(V(t_3)))^* L^{(2)}(g^{(2)}(V(t_3)))e^{it_3\Delta}\sd^{-s}\sd^{2s} g(-i\nabla_x)\sd^{2s}.
 \end{multline*}
 We rewrite this big operator from $L^2_x$ to $L^2_x$ as a composition of operators
 $$L^2_x \to L^2_{t_3,x} \to L^2_{t_2,x} \to L^2_{t_1,x} \to L^2_{t,x} \to L^2_x$$
 including the time ordered integrals, using the notations of Corollary \ref{coro:CK}:
 \begin{multline*}
  \int W(t,x)(\sd^s\rho_{3,1}(V))(t,x) \,dx\,dt \\
  = \sum_{n_1,n_2} \tr_{L^2_x}(A^{(2)}(g^{(2)}(W(t)))e^{it\Delta_x}\sd^{-s})[A^{(1)}(g^{(1)}(W(t)))e^{it\Delta_x}(J_{n_1}^{(1)}(g_{n_1}^{(1)}(V(t_1)))e^{it_1\Delta})^*]_<\times\\
  \times[J_{n_1}^{(2)}(g_{n_1}^{(2)}(V(t_1)))\sd^sU_V(t_1)\sd^{-s}(J_{n_2}^{(1)}(g_{n_2}^{(1)}(V(t_2)))\sd^{-s}U_V(t_2)\sd^{s})^*]_<\times\\
  \times[J_{n_2}^{(2)}(g_{n_2}^{(2)}(V(t_2)))e^{it_2\Delta}
  (L^{(1)}(g^{(1)}(V(t_3)))e^{it_3\Delta})^*]_<\times\\
  \times L^{(2)}(g^{(2)}(V(t_3)))e^{it_3\Delta}\sd^{-s}\sd^{2s} g(-i\nabla_x)\sd^{2s}.
  \end{multline*}
  By \eqref{eq:est-Ag-gamma-derivative} and \eqref{eq:est-g-gamma-derivative}, we have
  $$\|A^{(2)}(g^{(2)}(W(t)))e^{it\Delta_x}\sd^{-s}\|_{\gS^{2d(d+2)/(d^2-2)}(L^2_x\to L^2_{t}L^2(X))}\lesssim \|W\|_{L^2_{t,x}}^{\tfrac{d}{d+2}}.$$
  By \eqref{eq:est-Ag-gamma-derivative}, we also have
  $$\|A^{(1)}(g(t))e^{it\Delta_x}\|_{\gS^{2(d+2)}(L^2_x\to L^2_{t}L^2(X))}\lesssim \|g\|_{L^{d+2}_t B^{(1)}}.
  $$
  Similarly, by Lemma \ref{lem:strichartz-A} combined with Lemma \ref{lem:leibniz-without-loss} we have
  $$\|J_{n_1}^{(1)}(g(t_1))e^{it_1\Delta}\|_{\gS^{2(\alpha_{n_1}^{(1)})'}(L^2_x\to L^2_{t_1}L^2(X_{n_1}))}\lesssim \|g\|_{L^{2(p_{n_1}^{(1)})'}_{t_1} B_{n_1}^{(1)}}.
  $$
  Combining these last two estimates together with  Corollary \ref{coro:CK}, we get
  \begin{multline*}
    \|[A^{(1)}(g^{(1)}(W(t)))e^{it\Delta_x}(J_{n_1}^{(1)}(g_{n_1}^{(1)}(V(t_1)))e^{it_1\Delta})^*]_<\|_{\gS^{a_1}(L^2_{t_1}L^2(X_{n_1})\to L^2_{t}L^2(X))}\\
    \lesssim \|g^{(1)}(W(t))\|_{L^{d+2}_t B^{(1)}}\|g_{n_1}^{(1)}(V(t_1))\|_{L^{2(p_{n_1}^{(1)})'}_{t_1} B_{n_1}^{(1)}},
  \end{multline*}
  with
  $$\frac{1}{a_1} = \frac{1}{2(d+2)}+\frac{1}{2(\alpha_{n_1}^{(1)})'}.$$
  By \eqref{eq:est-g-gamma-derivative} and Lemma \ref{lem:leibniz-without-loss}, we deduce that
  $$
    \|[A^{(1)}(g^{(1)}(W(t)))e^{it\Delta_x}(J_{n_1}^{(1)}(g_{n_1}^{(1)}(V(t_1)))e^{it_1\Delta})^*]_<\|_{\gS^{a_1}(L^2_{t_1}L^2(X_{n_1})\to L^2_{t}L^2(X))}
    \lesssim \|W\|_{L^2_{t,x}}^{\tfrac{2}{d+2}}\|V\|_{L^2_t H^s_x}^{\theta_{n_1}^{(1)}}.
  $$
  By the same argument, and using that $\sd^s U_V(t)\sd^{-s}$ satisfies the same Strichartz estimates in Schatten spaces as $e^{it\Delta_x}$ by Proposition \ref{prop:UV-bez}, we also have
  \begin{multline*}
    \|[J_{n_1}^{(2)}(g_{n_1}^{(2)}(V(t_1)))\sd^sU_V(t_1)\sd^{-s}\times\\
    \times(J_{n_2}^{(1)}(g_{n_2}^{(1)}(V(t_2)))\sd^{-s}U_V(t_2)\sd^{s})^*]_<\|_{\gS^{a_2}(L^2_{t_2}L^2(X_{n_2})\to L^2_{t_1}L^2(X_{n_1}))}
        \lesssim \|V\|_{L^2_t H^s_x}^{\theta_{n_1}^{(2)}+\theta_{n_2}^{(1)}},
  \end{multline*}
  with
  $$\frac{1}{a_2}=\frac{1}{2(\alpha_{n_1}^{(2)})'}+\frac{1}{2(\alpha_{n_2}^{(1)})'},$$
  as well as
  $$
  \|[J_{n_2}^{(2)}(g_{n_2}^{(2)}(V(t_2)))e^{it_2\Delta}
  (L^{(1)}(g^{(1)}(V(t_3)))e^{it_3\Delta})^*]_<\|_{\gS^{a_3}(L^2_{t_3}L^2(X)\to L^2_{t_2}L^2(X_{n_2}))}
  \lesssim
  \|V\|_{L^2_t H^s_x}^{\theta_{n_2}^{(2)}+\tfrac{2}{d+2}},
  $$
  with
  $$\frac{1}{a_3}=\frac{1}{2(\alpha_{n_2}^{(2)})'}+\frac{1}{2(d+2)}.$$
  Finally, by \eqref{eq:est-Lg} we have
  $$
  \|L^{(2)}(g^{(2)}(V(t_3)))e^{it_3\Delta}\sd^{-s}\|_{\gS^{2d(d+2)/(d^2-2)}(L^2_x\to L^2_{t_3}L^2(X))}\lesssim \|V\|_{L^2_tH^s_x}^{\tfrac{d}{d+2}}.
  $$
  Using that
  $$\frac{d^2-2}{2d(d+2)}+\frac{1}{a_1}+\frac{1}{a_2}+\frac{1}{a_3}+\frac{d^2-2}{2d(d+2)}=\frac{d-1}{2d}+\frac{1}{2d}+\frac{1}{2d}+\frac{d-1}{2d}=1,$$
  we deduce by the Hölder inequality in Schatten spaces that
  \begin{multline*}
   |\tr_{L^2_x}(A^{(2)}(g^{(2)}(W(t)))e^{it\Delta_x}\sd^{-s})[A^{(1)}(g^{(1)}(W(t)))e^{it\Delta_x}(J_{n_1}^{(1)}(g_{n_1}^{(1)}(V(t_1)))e^{it_1\Delta})^*]_<\times\\
  \times[J_{n_1}^{(2)}(g_{n_1}^{(2)}(V(t_1)))\sd^sU_V(t_1)\sd^{-s}(J_{n_2}^{(1)}(g_{n_2}^{(1)}(V(t_2)))\sd^{-s}U_V(t_2)\sd^{s})^*]_<\times\\
  \times[J_{n_2}^{(2)}(g_{n_2}^{(2)}(V(t_2)))e^{it_2\Delta}
  (L^{(1)}(g^{(1)}(V(t_3)))e^{it_3\Delta})^*]_<\times\\
  \times L^{(2)}(g^{(2)}(V(t_3)))e^{it_3\Delta}\sd^{-s}\sd^{2s} g(-i\nabla_x)\sd^{2s}|\lesssim \|W\|_{L^2_{t,x}}\|V\|_{L^2_t H^s_x}^3,
  \end{multline*}
  and hence using that $\|V\|_{L^2_t H^s_x}\le\delta$, we have
  $$
  \left|\int W(t,x)(\sd^s\rho_{3,1}(V))(t,x) \,dx\,dt\right|\lesssim \|W\|_{L^2_{t,x}},
  $$
  proving the first estimate of Proposition \ref{prop:rho31}. The second estimate is obtained in the same exact way, using furthermore the estimate on $\sd^s(U_V(t)-U_{\tilde{V}}(t))\sd^{-s}$ from Proposition \ref{prop:UV-bez}.

\end{proof}

\begin{proposition}\label{prop:rho32}

   Let $d\ge2$ and $s=d/2-1$. If $\|V\|_{L^2_t H^s_x}+\|\tilde{V}\|_{L^2_t H^s_x}\le\delta$, then
   $$\left\| \rho_{3,2}(V) \right\|_{L^2_t H^s_x} \le C,$$
   $$\left\| \rho_{3,2}(V) - \rho_{3,2}(\tilde{V}) \right\|_{L^2_t H^s_x} \le \frac12 \|V-\tilde{V}\|_{L^2_t H^s_x}.$$

\end{proposition}

\begin{lemma}
\label{lem:leibniz-gamma-no-derivative}
 Let $d\ge2$ and $s=d/2-1$. Then, there exists $N\in\N\setminus\{0\}$ and for all $n=1,\ldots,N$ and $\ell\in\{1,2\}$ there exist  $p_n^{(\ell)},q_n^{(\ell)}\in(1,+\ii)$, $A_n^{(\ell)}:B_n^{(\ell)}\to\cB(L^{2q_n^{(\ell)}}_x,L^2(X_n))$ linear and continuous, where $B_n^{(\ell)}$ is some Banach space and $X_n$ some measure space, $g_n^{(\ell)}: H^s_x\to B_n^{(\ell)}$, and $\theta_n^{(\ell)}\in[0,1]$ with $\theta_n^{(1)}+\theta_n^{(2)}=1$ such that for all $W\in L^2_{t,x}$ and all $\gamma(t)$ operator on $L^2_x$ one can decompose
 \begin{equation}
 \begin{multlined}
    \int W(t,x)(\sd^s\rho_{\gamma(t)})(x)\,dx\,dt \\
    = \sum_{n=1}^N \int \tr_{L^2_x}A_n^{(2)}(g_n^{(2)}(W(t)))^*A_n^{(1)}(g_n^{(1)}(W(t)))\sd^s\gamma(t)\sd^s\,dt,
 \end{multlined}
 \end{equation}
with for all $n=1,\ldots,N$ and all $\ell\in\{1,2\}$,
 \begin{equation}
    \|g_n^{(\ell)}(W(t))\|_{L^{2(p_n^{(\ell)})'}_t B_n^{(\ell)}}\lesssim \|W\|_{L^2_{t,x}}^{\theta_n^{(\ell)}}
 \end{equation}
 $$\frac{2}{p_n^{(\ell)}}+\frac{d}{q_n^{(\ell)}}=d,$$
 and defining $\alpha_n^{(\ell)}$ by $1/\alpha_n^{(\ell)}=1/(dp_n^{(\ell)})+1/q_n^{(\ell)}$, we have
 $$\alpha_n^{(\ell)}<p_n^{(\ell)},\quad\frac{1}{2(\alpha_n^{(1)})'}+\frac{1}{2(\alpha_n^{(2)})'}=\frac{1}{2d}.$$
 In particular, note that
 \begin{equation}
    \|A_n^{(\ell)}(g(t))\|_{L^{2(p_n^{(\ell)})'}_t\cB(L^{2q_n^{(\ell)}}_x\to L^2(X_n))}\lesssim \|g\|_{L^{2(p_n^{(\ell)})'}_t B_n^{(\ell)}}.
 \end{equation}
\end{lemma}

We can prove Lemma \ref{lem:leibniz-gamma-no-derivative} from Lemma \ref{lem:leibniz-gamma} in the same way that Lemma \ref{lem:leibniz-without-loss} was deduced from Lemma \ref{lem:leibniz-op-general} so we omit the proof.

\begin{proof}[Proof of Proposition \ref{prop:rho32}]
 Let $W\in L^2_{t,x}$. Using Lemma \ref{lem:leibniz-gamma-no-derivative}, we have
 \begin{multline*}
  \int W(t,x)(\sd^s\rho_{3,2}(V))(t,x) \,dx\,dt \\
  =  \sum_{n_1=1}^N\tr_{L^2_x}\int_0^\ii dt\, e^{-it\Delta_x}A_{n_1}^{(2)}(g_{n_1}^{(2)}(W(t)))^*A_{n_1}^{(1)}(g_{n_1}^{(1)}(W(t)))e^{it\Delta_x}\times\\
  \times \int_0^t dt_1\, e^{-it_1\Delta}\sd^s V(t_1)\sd^{-s}\sd^sU_V(t_1)\sd^{-s}\times\\
  \times\int_0^{t_1} dt_2\, \sd^s U_V(t_2)^*\sd^{-s}\sd^s V(t_2)\sd^{-s}e^{it_2\Delta}\times\\
  \times\sd^s g(-i\nabla_x)\sd^s\int_0^{t_2} dt_3\, e^{-it_3\Delta}\sd^{-s} V(t_3)\sd^{s} e^{it_3\Delta}
 \end{multline*}
 We then expand $\sd^s V(t_1)\sd^{-s}$ according to Lemma \ref{lem:leibniz-without-loss} (applied to $s=d/2-1$ and $\mu=\nu=2$), $\sd^s V(t_2)\sd^{-s}$ according to Lemma \ref{lem:leibniz-op-2}, and $\sd^{-s}V(t_3)\sd^s$ according to the dual of Lemma \ref{lem:leibniz-op-2} (writing that $\sd^{-s}V(t_3)\sd^s=(\sd^s V(t_3)\sd^{-s})^*$). We obtain
 \begin{multline*}
  \int W(t,x)(\sd^s\rho_{3,2}(V))(t,x) \,dx\,dt \\
  = \sum_{n_1,n_2} \tr_{L^2_x}\int_0^\ii dt\, e^{-it\Delta_x}A_{n_1}^{(2)}(g_{n_1}^{(2)}(W(t)))^*A_{n_1}^{(1)}(g_{n_1}^{(1)}(W(t)))e^{it\Delta_x}\times\\
  \times \int_0^t dt_1\, e^{-it_1\Delta}J_{n_2}^{(1)}(g_{n_2}^{(1)}(V(t_1)))^*J_{n_2}^{(2)}(g_{n_2}^{(2)}(V(t_1)))\sd^sU_V(t_1)\sd^{-s}\times\\
  \times\int_0^{t_1} dt_2\, \sd^s U_V(t_2)^*\sd^{-s}L^{(1)}(g^{(1)}(V(t_2)))^* L^{(2)}(g^{(2)}(V(t_2)))e^{it_2\Delta}\sd^{-s}\times\\
  \times\sd^{2s} g(-i\nabla_x)\sd^{2s}\int_0^{t_2} dt_3\, \sd^{-s}e^{-it_3\Delta}L^{(2)}(g^{(2)}(V(t_3)))^* L^{(1)}(g^{(1)}(V(t_3)))e^{it_3\Delta}.
 \end{multline*}
 We estimate in the same way as in the proof of Proposition \ref{prop:rho31}.
\end{proof}

\begin{proposition}\label{prop:rho33}

   Let $d\ge2$ and $s=d/2-1$. If $\|V\|_{L^2_t H^s_x}+\|\tilde{V}\|_{L^2_t H^s_x}\le\delta$, then
   $$\left\| \rho_{3}(V) \right\|_{L^2_t H^s_x} \le C,$$
   $$\left\| \rho_{3}(V) - \rho_{3}(\tilde{V}) \right\|_{L^2_t H^s_x} \le \frac12 \|V-\tilde{V}\|_{L^2_t H^s_x}.$$

\end{proposition}

\begin{proof}
 We have
 $$\rho_{3}(V) = \rho_{3,3,1}(V) - \bar{\rho_{3,3,1}(V)}$$
 where
 $$\rho_{3,3,1}(V)(t):=\rho\left[\int_0^tD_V(t,\tau)V(\tau)g(-i\nabla_x)D_V(t,\tau)^*\,d\tau\right].$$
 By definition of $D_V(t,\tau)$, we have
 $$
 \begin{multlined}
 \rho_{3,3,1}(V)(t)\\
 =\rho\left[
 \int_0^td\tau\,\int_\tau^t dt_1\,\int_\tau^t\,dt_2 e^{i(t-t_1)\Delta_x} V(t_1)U_V(t_1,\tau)V(\tau)g(-i\nabla_x)e^{i(\tau-t_2)\Delta} V(t_2)U_V(t_2,t)
 \right].
 \end{multlined}
 $$
 Let $W\in L^2_{t,x}$. Using Lemma \ref{lem:leibniz-gamma-no-derivative}, we deduce that
 \begin{multline*}
  \int W(t,x)(\sd^s\rho_{3,3,1}(V))(t,x) \,dx\,dt \\
  =  \sum_{n_1=1}^N\int_0^\ii dt\,\int_0^t d\tau\,\int_\tau^{t}\,dt_1\int_\tau^{t}dt_2\,\tr_{L^2_x} U_V(t)^*\sd^s A_{n_1}^{(2)}(g_{n_1}^{(2)}(W(t)))^*A_{n_1}^{(1)}(g_{n_1}^{(1)}(W(t)))e^{it\Delta_x}\times\\
  \times e^{-it_1\Delta}\sd^s V(t_1)U_V(t_1) U_V(\tau)^* V(\tau)e^{i\tau\Delta_x}g(-i\nabla_x)e^{-it_2\Delta} V(t_2)U_V(t_2)
 \end{multline*}
 Compared to the proofs of Proposition \ref{prop:rho31} and Proposition \ref{prop:rho32}, the time integrals are not nested, which was important to apply Corollary \ref{coro:CK} iteratively. To reduce to nested integrals, we use that
 \begin{align*}
    \int_0^\ii dt\,\int_0^t d\tau\,\int_\tau^{t}\,dt_1\int_\tau^{t}dt_2 &= \int_0^\ii dt\,\int_0^t\,dt_1\,\int_0^{t_1}d\tau\,\int_\tau^t\,dt_2 \\
    &= \int_0^\ii dt\,\int_0^t\,dt_1\,\int_0^{t_1}d\tau\,\int_\tau^\ii\,dt_2-\int_0^\ii dt\,\int_0^t\,dt_1\,\int_0^{t_1}d\tau\,\int_t^\ii\,dt_2.
 \end{align*}
 The first integral has the nested structure (in the last term where we have $\int_\tau^\ii dt_2$ instead of $\int_0^\tau dt_2$, we use Corollary \ref{coro:CK} in the version of Remark \ref{rk:reversed-coro-CK}), so that it can be treated as in the proof of Proposition \ref{prop:rho31} (expanding $\sd^s V(t_1)\sd^{-s}$ according to Lemma \ref{lem:leibniz-without-loss} with $s=d/2-1$ and $\mu=\nu=2$, expanding $\sd^sV(\tau)\sd^{-s}$ according to Lemma \ref{lem:leibniz-op-2}, and expanding $\sd^{-s}V(t_2)\sd^s$ according to the dual of Lemma \ref{lem:leibniz-op-2}). For the second integral, we use that
 $$\int_0^\ii dt\,\int_0^t\,dt_1\,\int_0^{t_1}d\tau\,\int_t^\ii\,dt_2 = \int_0^\ii dt_2\,\int_0^{t_2}\,dt\,\int_0^t\,dt_1\,\int_0^{t_1}d\tau,$$
 as well as the cyclicity of the trace to write
 \begin{multline*}
  \int_0^\ii dt_2\,\int_0^{t_2}\,dt\,\int_0^t\,dt_1\,\int_0^{t_1}d\tau\,\tr_{L^2_x} U_V(t)^*\sd^s A_{n_1}^{(2)}(g_{n_1}^{(2)}(W(t)))^*A_{n_1}^{(1)}(g_{n_1}^{(1)}(W(t)))e^{it\Delta_x}\times\\
  \times e^{-it_1\Delta}\sd^s V(t_1)U_V(t_1) U_V(\tau)^* V(\tau)e^{i\tau\Delta_x}g(-i\nabla_x)e^{-it_2\Delta} V(t_2)U_V(t_2)\\
  = \int_0^\ii dt_2\,\int_0^{t_2}\,dt\,\int_0^t\,dt_1\,\int_0^{t_1}d\tau \tr_{L^2_x} e^{-it_2\Delta} V(t_2)U_V(t_2)U_V(t)^*\sd^s \times\\
  \times A_{n_1}^{(2)}(g_{n_1}^{(2)}(W(t)))^*A_{n_1}^{(1)}(g_{n_1}^{(1)}(W(t)))e^{it\Delta_x}e^{-it_1\Delta}\sd^s V(t_1)U_V(t_1) U_V(\tau)^* V(\tau)e^{i\tau\Delta_x}g(-i\nabla_x),
 \end{multline*}
 which again has the nested structure and can be estimated as the first integral.
\end{proof}

 \subsection{The first quadratic term}

 \begin{proposition}\label{prop:rho21}

   Let $d\ge3$ and $s=d/2-1$. If $\|V\|_{L^2_t H^s_x}+\|\tilde{V}\|_{L^2_t H^s_x}\le\delta$, then
   $$\left\| \rho_{2,1}(V) \right\|_{L^2_t H^s_x} \le C,$$
   $$\left\| \rho_{2,1}(V) - \rho_{2,1}(\tilde{V}) \right\|_{L^2_t H^s_x} \le \frac12 \|V-\tilde{V}\|_{L^2_t H^s_x}.$$

\end{proposition}

\begin{proof}
Let $W\in L^2_{t,x}$. We have
$$
\begin{multlined}
 \int W(t,x)\rho_{2,1}(V)(t,x)\,dx\,dt= \tr_{L^2_x}\int_0^\ii dt\,\int_0^t dt_1\,\int_0^{t_1}dt_2\,e^{-it\Delta_x}W(t)e^{it\Delta_x}\sd^{-s}\times \\
 \times e^{-it_1\Delta} \sd^sV(t_1)\sd^{-s}e^{i(t_1-t_2)\Delta}\sd^{-s}\sd^{2s}g(-i\nabla_x)\sd^{s}\sd^{-s}V(t_2)e^{it_2\Delta}.
\end{multlined}
$$
We split $W(t)=|W(t)|^{\tfrac{2}{d+2}}W(t)^{\tfrac{d}{d+2}}$, $\sd^sV(t_1)\sd^{-s}$ according to Lemma \ref{lem:leibniz-op-2}, and $V(t_2)=|V(t_2)|^{\tfrac{d}{d+2}}V(t_2)^{\tfrac{2}{d+2}}$. Combining Corollary \ref{coro:general-schatten} and Corollary \ref{coro:CK} as in the previous proofs, we deduce that
$$\left|\int W(t,x)\rho_{2,1}(V)(t,x)\,dx\,dt\right|\lesssim \|W\|_{L^2_{t,x}}\|V\|_{L^2_t H^s_x}^2.$$
Let us now estimate $|D_x|^s\rho_{2,1}(V)$ in $L^2_{t,x}$. By Lemma \ref{lem:leibniz-gamma} applied to $\sigma=s$, $\sigma_1=\sigma_2=s/2$, $q=2$, and $q_1=q_2=4$, we have
\begin{multline*}
 \int W(t,x)(|D_x|^s\rho_{2,1}(V))(t,x)\,dx\,dt = \tr_{L^2_x}\int_0^\ii dt\,\int_0^t dt_1\,\int_0^{t_1}dt_2\,\times\\
 \times e^{-it\Delta_x}\left(\sum_{|\alpha|\le s/2}\frac{1}{\alpha!}(D^{s,\alpha}_xW(t)\partial^\alpha_x+\partial^\alpha_x W(t) D^{s,\alpha}_x)+A^{(2)}(h^{(2)}(W(t)))^* A^{(1)}(h^{(1)}(W(t)))\right)e^{it\Delta_x}\times\\
 \times e^{-it_1\Delta} V(t_1)e^{i(t_1-t_2)\Delta}g(-i\nabla_x)V(t_2)e^{it_2\Delta}.
\end{multline*}
 Let us first treat the term
 \begin{multline*}
 \tr_{L^2_x}\int_0^\ii dt\,\int_0^t dt_1\,\int_0^{t_1}dt_2\,e^{-it\Delta_x}A^{(2)}(h^{(2)}(W(t)))^* A^{(1)}(h^{(1)}(W(t)))e^{it\Delta_x}\times\\
 \times e^{-it_1\Delta} V(t_1)e^{i(t_1-t_2)\Delta}g(-i\nabla_x)V(t_2)e^{it_2\Delta}\\
 =\tr_{L^2_x}\int_0^\ii dt\,\int_0^t dt_1\,\int_0^{t_1}dt_2\,\sd^{-s/2}e^{-it\Delta_x}\sd^{-s/2}A^{(2)}(h^{(2)}(W(t)))^* \times\\
 \times A^{(1)}(h^{(1)}(W(t)))\sd^{-s/2}e^{it\Delta_x}\sd^{-s/2} e^{-it_1\Delta} \sd^s V(t_1)\sd^{-s}e^{it_1\Delta}\sd^{-s}\times\\
 \times\sd^{2s}g(-i\nabla_x)\sd^{2s} \sd^{-s}e^{-it_2\Delta}\sd^{-s}V(t_2)\sd^s e^{it_2\Delta}.
 \end{multline*}
 We expand $\sd^s V(t_1)\sd^{-s}$ according to Lemma \ref{lem:leibniz-op-2} and $\sd^{-s}V(t_2)\sd^s$ according to the dual of \eqref{eq:decomp-leibniz-schatten}. We obtain
 \begin{multline*}
 \tr_{L^2_x}\int_0^\ii dt\,\int_0^t dt_1\,\int_0^{t_1}dt_2\,e^{-it\Delta_x}A^{(2)}(h^{(2)}(W(t)))^* A^{(1)}(h^{(1)}(W(t)))e^{it\Delta_x}\times\\
 \times e^{-it_1\Delta} V(t_1)e^{i(t_1-t_2)\Delta}g(-i\nabla_x)V(t_2)e^{it_2\Delta}\\
 =\tr_{L^2_x}\int_0^\ii dt\,\int_0^t dt_1\,\int_0^{t_1}dt_2\,\sd^{-s/2}e^{-it\Delta_x}\sd^{-s/2}A^{(2)}(h^{(2)}(W(t)))^* \times\\
 \times A^{(1)}(h^{(1)}(W(t)))\sd^{-s/2}e^{it\Delta_x}\sd^{-s/2} e^{-it_1\Delta} L^{(1)}(g^{(1)}(V(t_1)))^* L^{(2)}(g^{(2)}(V(t_1)))e^{it_1\Delta}\sd^{-s}\times\\
 \times\sd^{2s}g(-i\nabla_x)\sd^{2s} \sd^{-s}e^{-it_2\Delta}L^{(2)}(g^{(2)}(V(t_2)))^* L^{(1)}(g^{(1)}(V(t_1))) e^{it_2\Delta}.
 \end{multline*}
 From Lemma \ref{lem:leibniz-gamma} we have that
 $$\|A^{(1)}(h)\sd^{-s/2}\|_{L^4_x\to L^2(X)} \lesssim \|h\|_{B^{(1)}}$$
 hence by Corollary \ref{coro:general-schatten} applied to $\sigma=s/2$, $p=q=2$, and $\alpha=2d/(d+1)$ we deduce that
 $$\|A^{(1)}(h(t))\sd^{-s/2}e^{it\Delta_x}\sd^{-s/2}\|_{\gS^{4d/(d-1)}(L^2_x\to L^2_t L^2(X))}\lesssim
 \int_\R\|h(t)\|_{B^{(1)}}^4\,dt.$$
 In the same way, we have
 $$\|A^{(2)}(h(t))\sd^{-s/2}e^{it\Delta_x}\sd^{-s/2}\|_{\gS^{4d/(d-1)}(L^2_x\to L^2_t L^2(X))}\lesssim
 \int_\R\|h(t)\|_{B^{(2)}}^4\,dt.$$
 Together with \eqref{eq:est-Lg} and
 $$\|L^{(1)}(g(t))e^{it\Delta_x}\|_{\gS^{2(d+2)}(L^2_x\to L^2_{t}L^2(X))}\lesssim \|g\|_{L^{d+2}_t B^{(1)}},
  $$
 we can apply Corollary \ref{coro:CK} used as in the proof of Proposition \ref{prop:rho31}, as well as the Hölder inequality in Schatten spaces that we can use since
 $$\frac{d-1}{4d}+\frac{d-1}{4d}+\frac{1}{2(d+2)}+\frac{d^2-2}{2d(d+2)}+\frac{1}{2(d+2)}+\frac{d^2-2}{2d(d+2)}=3\frac{d-1}{2d}\ge1
 $$
 to deduce that
 \begin{multline*}
 |\tr_{L^2_x}\int_0^\ii dt\,\int_0^t dt_1\,\int_0^{t_1}dt_2\,e^{-it\Delta_x}A^{(2)}(h^{(2)}(W(t)))^* A^{(1)}(h^{(1)}(W(t)))e^{it\Delta_x}\times\\
 \times e^{-it_1\Delta} V(t_1)e^{i(t_1-t_2)\Delta}g(-i\nabla_x)V(t_2)e^{it_2\Delta}|\lesssim \|W\|_{L^2_{t,x}}\|V\|_{L^2_t H^s_x}^2.
 \end{multline*}
 The terms $D^{s,\alpha}_xW(t)\partial^\alpha_x+\partial^\alpha_x W(t) D^{s,\alpha}_x$ for $|\alpha|\le s/2$ are treated in the same way: we split them as
 $$\sd^{-s}\partial^\alpha_x W(t) D^{s,\alpha}_x\sd^{-s}
 =\sd^{-s}\partial^\alpha_x|W(t)|^{\tfrac{d-2|\alpha|}{d+2}}W(t)^{\tfrac{2(1+|\alpha|)}{d+2}}D^{s,\alpha}_x\sd^{-s}.
 $$
 By \eqref{eq:stri-bez-xtotx}, we have
 \begin{align*}
  \|W(t)^{\tfrac{2(1+|\alpha|)}{d+2}}e^{it\Delta_x}D^{s,\alpha}_x\sd^{-s}\|_{\gS^{a_1}(L^2_x\to L^2_{t,x})} &\lesssim \|W(t)^{\tfrac{2(1+|\alpha|)}{d+2}}e^{it\Delta_x}\sd^{-|\alpha|}\|_{\gS^{a_1}(L^2_x\to L^2_{t,x})} \\
  &\lesssim \|W\|_{L^2_{t,x}}^{\tfrac{2(1+|\alpha|)}{d+2}},
 \end{align*}
 \begin{align*}
     \||W(t)|^{\tfrac{d-2|\alpha|}{d+2}}e^{it\Delta_x}\sd^{-s}\partial^\alpha_x\|_{\gS^{a_2}(L^2_x\to L^2_{t,x})} &\lesssim \||W(t)|^{\tfrac{d-2|\alpha|}{d+2}}e^{it\Delta_x}\sd^{-(s-|\alpha|)}\|_{\gS^{a_2}(L^2_x\to L^2_{t,x})}\\
     &\lesssim \|W\|_{L^2_{t,x}}^{\tfrac{d-2|\alpha|}{d+2}},
 \end{align*}
 with
 $$\frac{1}{a_1}=\frac{(d+1)(1+|\alpha|)}{d(d+2)}-\frac{1}{2d},\quad \frac{1}{a_2}=\frac{(d+1)(d-2|\alpha|)}{2d(d+2)}-\frac{1}{2d}.$$
 Since
 $$\frac{1}{a_1}+\frac{1}{a_2} = \frac{d-1}{2d},$$
 we can conclude as for the previous terms, to get the same estimate. This leads to the desired estimate of $|D_x|^s\rho_{2,1}(V)$ in $L^2_{t,x}$, and the estimate of $\rho_{2,1}(V)$ follows from similar arguments so we omit it.
\end{proof}

 \subsection{The second quadratic term}

 \begin{proposition}\label{prop:rho22}

   Let $d\ge3$ and $s=d/2-1$. If $\|V\|_{L^2_t H^s_x}+\|\tilde{V}\|_{L^2_t H^s_x}\le\delta$, then
   $$\left\| \rho_{2,2}(V) \right\|_{L^2_t H^s_x} \le C,$$
   $$\left\| \rho_{2,2}(V) - \rho_{2,2}(\tilde{V}) \right\|_{L^2_t H^s_x} \le \frac12 \|V-\tilde{V}\|_{L^2_t H^s_x}.$$

\end{proposition}

Following the proof of Lemma \ref{lem:leibniz-op-2} but reversing the roles of $q_n^{(1)}$ and $q_n^{(2)}$, we obtain:

\begin{lemma}\label{lem:leibniz-op-3}

Let $d\ge2$ and $s=d/2-1$. Then, there exist  $K^{(1)}:B^{(1)}\to\cB(L^{d+2}_x,L^2(X))$, $K^{(2)}:B^{(2)}\to\cB(L^{2+4/d}_x,L^2(X))$ linear and continuous, where $X$ is some measure space, and for $\ell\in\{1,2\}$, $B^{(\ell)}$ are some Banach spaces and $g^{(\ell)}: H^s_x\to B^{(\ell)}$ are such that for all $W\in H^s_x$ one can decompose
 \begin{equation}\label{eq:decomp-leibniz-schatten-K}
    \nsd^s  W\nsd^{-s} =  K^{(1)}(g^{(1)}(W))^* K^{(2)}(g^{(2)}(W)),
 \end{equation}
 \begin{equation}\label{eq:est-gnK}
  \|g^{(1)}(W)\|_{B^{(1)}}\lesssim \|W\|_{H^s_x}^{\tfrac{d}{d+2}},\quad \|g^{(2)}(W)\|_{B^{(2)}}\lesssim \|W\|_{H^s_x}^{\tfrac{2}{d+2}}.
\end{equation}
In particular, we have
\begin{equation}\label{eq:est-Kg}
\begin{aligned}
  \|K^{(1)}(g(t))e^{it\Delta_x}\sd^{-s}\|_{\gS^{2d(d+2)/(d^2-2)}(L^2_x\to L^2_tL^2(X))}&\lesssim \|g\|_{L^{2(d+2)/d}_t B^{(1)}},\\
  \|K^{(2)}(g(t))e^{it\Delta_x}\|_{\gS^{2(d+2)}(L^2_x\to L^2_{t}L^2(X))}&\lesssim \|g\|_{L^{d+2}_t B^{(2)}}.
\end{aligned}
\end{equation}

\end{lemma}

 \begin{lemma}\label{lem:leibniz-op-4}

Let $d\ge3$ and $s=d/2-1$. Then, there exists $N\in\N\setminus\{0\}$ and for all $n=1,\ldots,N$ and $\ell\in\{1,2\}$ there exist  $p_n^{(\ell)},q_n^{(\ell)}\in(1,+\ii)$, $L_n^{(\ell)}:B_n^{(\ell)}\to\cB(L^{2q_n^{(\ell)}}_x,L^2(X_n))$ linear and continuous, with $B_n^{(\ell)}$ some Banach space and $X_n$ some measure space, and $g_n^{(\ell)}: H^s_x\to B_n^{(\ell)}$ such that for all $W\in H^s_x$ one can decompose
 \begin{equation}\label{eq:decomp-leibniz-schatten-M}
    \nsd^s  W\nsd^{-s} =  \sum_{n=1}^N M_n^{(1)}(g_n^{(1)}(W))^* M_n^{(2)}(g_n^{(2)}(W)),
 \end{equation}
 with for all $n=1,\ldots,N$ and all $\ell\in\{1,2\}$,
 \begin{equation}\label{eq:est-gnM}
    \|g_n^{(\ell)}(V(t))\|_{L^{2(p_n^{(\ell)})'}_t B_n^{(\ell)}}\lesssim \|V\|_{L^2_t H^s_x}^{\theta_n^{(\ell)}},
 \end{equation}
 as well as
\begin{equation}\label{eq:est-Mg}
\begin{aligned}
  \|M_n^{(1)}(g(t))e^{it\Delta_x}\sd^{-1/2}\|_{\gS^{2(\alpha_n^{(1)})'}(L^2_x\to L^2_{t,x})}&\lesssim \|g\|_{L^{2(p_n^{(1)})'}_t B_n^{(1)}},\\
  \|M_n^{(2)}(g(t))e^{it\Delta_x}\|_{\gS^{2(\alpha_n^{(2)})'}(L^2_x\to L^2_{t,x})}&\lesssim \|g\|_{L^{2(p_n^{(2)})'}_t B_n^{(2)}},
\end{aligned}
\end{equation}
for some $1\le \alpha_n^{(\ell)}<p_n^{(\ell)}$ such that
$$\frac{1}{2(\alpha_n^{(1)})'}+\frac{1}{2(\alpha_n^{(2)})'}\ge\frac{1}{d}.$$
 \end{lemma}

 \begin{proof}
 We go back to the proof of Lemma \ref{lem:leibniz-without-loss}, except that we will not absorb $\sd^{-s_n}$ in the term $J_n^{(2)}$. From this proof we obtain
 \begin{equation}\label{eq:decomp-leibniz-schatten-J-2}
    \sd^s W\sd^{-s} = \sum_{a\in\{0,1,\ldots,\lfloor s\rfloor,s\}} L_a^{(1)}(g_a^{(1)}(W))^*L_a^{(2)}(g_a^{(2)}(W))\sd^{-a},
 \end{equation}
 with
 $$\|L_a^{(2)}(g(t))\|_{L^{2(p_a^{(2)})'}_t\cB(L^{2\tilde{q_a}^{(2)}}_x\to L^2(X_a))}\lesssim \|g\|_{L^{2(p_a^{(2)})'}_t B_a^{(2)}},$$
 and
 $$\frac{2}{p_a^{(2)}}+\frac{d}{\tilde{q_a}^{(2)}}=d-2a.$$
 Defining $\tilde{\alpha_a}^{(2)}$ by
 $$\frac{1}{\tilde{\alpha_a}^{(2)}}=\frac{1}{dp_a^{(2)}}+\frac{1}{\tilde{q_a}^{(2)}},$$
 we thus have, by Corollary \ref{coro:general-schatten},
 $$\|L_a^{(1)}(g(t))e^{it\Delta_x}\|_{\gS^{2(\alpha_a^{(1)})'}(L^2_x\to L^2_t L^2(X_a))} \lesssim \|g\|_{L^{2(p_a^{(1)})'}_t B_a^{(1)}},$$
 $$\|L_a^{(2)}(g(t))e^{it\Delta_x}\sd^{-a}\|_{\gS^{2(\tilde{\alpha_a}^{(2)})'}(L^2_x\to L^2_t L^2(X_a))} \lesssim \|g\|_{L^{2(p_a^{(2)})'}_t B_a^{(2)}},$$
 with
 $$\frac{1}{2(\alpha_a^{(1)})'}+\frac{1}{2(\tilde{\alpha_a}^{(2)})'}=\frac{1}{2d}+\frac{a}{d}.$$
 For $a\neq0$, we thus set $M_a^{(1)}=L_a^{(1)}$ and $M_a^{(2)}=L_a^{(2)}\sd^{-a}$ to get the desired properties (even without the additional $\sd^{-1/2}$). The additional $\sd^{-1/2}$ will be useful for the term with $a=0$. From the proof of Lemma \ref{lem:leibniz-decomp-general}, there is only one such term (it corresponds to $\alpha=0$, namely $W|D|^s\sd^{-s}$). For this term, we set $M^{(1)}(g)=g:L^{2q^{(1)}}_x\to L^2_x$ for $g\in B^{(1)}:=L^{2(q^{(1)})'}$ with $(q^{(1)})'=5d/9$, $M^{(2)}(g)=g|D|^s\sd^{-s}:L^{2q^{(2)}}_x\to L^2_x$ for $g\in B^{(2)}:=L^{2(q^{(2)})'}$ with $(q^{(2)})'=5d/6$, $g^{(1)}(W)=|W|^{3/5}$, $g^{(2)}(W)=W^{2/5}$. Notice that $g^{(1)}(W)\in B^{(1)}$ and that $g^{(2)}(W)\in B^{(2)}$ since $W\in H^s_x\hookrightarrow L^{2d/3}_x$ since $d\ge3$. Defining $p^{(1)}$ and $p^{(2)}$ by the relations $2/p_1^{(1)}+d/q_1^{(1)}=d-1$, $2/p_1^{(2)}+d/q_1^{(2)}=d$ (that is, $p^{(1)}=5/2$ and $p^{(2)}=5/3$), we deduce by Corollary \ref{coro:general-schatten} that
  \begin{align*}
   \|M^{(1)}(g(t))e^{it\Delta_x}\sd^{-1/2}\|_{\gS^{10d/7}(L^2_x\to L^2_tL^2(X_1))} &\lesssim \|M_1^{(1)}(g(t))\|_{L^{10/3}_t\cB(L^{10d/9}_x\to L^2(X_1))} \\
   &\lesssim \|g\|_{L^{10/3}_t B_1^{(1)}},
  \end{align*}
  \begin{align*}
   \|M^{(2)}(g(t))e^{it\Delta_x}\|_{\gS^{10d/3}(L^2_x\to L^2_t L^2(X_1))} &\lesssim \|M_1^{(2)}(g(t))\|_{L^{5}_t\cB(L^{5d/3}_x\to L^2(X_1))} \\
   &\lesssim \|g\|_{L^{5}_t B_1^{(2)}}.
  \end{align*}
 \end{proof}

\begin{proof}[Proof of Proposition \ref{prop:rho22}]
 Let $W\in L^2_{t,x}$. We have
$$
\begin{multlined}
 \int W(t,x)\rho_{2,2}(V)(t,x)\,dx\,dt= \tr_{L^2_x}\int_0^\ii dt\,\int_0^t dt_1\,\int_0^{t_1}dt_2\,\sd^{-s}e^{-it\Delta_x}W(t)e^{it\Delta_x}\times \\
 \times\sd^{-s} e^{-it_1\Delta} \sd^sV(t_1)\sd^{-s}e^{i(t_1-t_2)\Delta}\sd^{s}V(t_2)\sd^{-s}e^{it_2\Delta}\sd^{-s}\times\\
 \times\sd^{2s}g(-i\nabla_x)\sd^s.
\end{multlined}
$$
We split $W(t)=|W(t)|^{\tfrac{d}{d+2}}W(t)^{\tfrac{2}{d+2}}$, $\sd^sV(t_1)\sd^{-s}$ according to the dual of Lemma \ref{lem:leibniz-op-2}, and $\sd^sV(t_2)\sd^{-s}$ according to Lemma \ref{lem:leibniz-op-2}. Combining Corollary \ref{coro:general-schatten} and Corollary \ref{coro:CK} as in the previous proofs, we deduce that
$$\left|\int W(t,x)\rho_{2,2}(V)(t,x)\,dx\,dt\right|\lesssim \|W\|_{L^2_{t,x}}\|V\|_{L^2_t H^s_x}^2.$$
Let us now estimate $|D_x|^s\rho_{2,2}(V)$ in $L^2_{t,x}$. By Lemma \ref{lem:leibniz-gamma} applied to $\sigma=s$, $\sigma_1=0$, $\sigma_2=s$, $q=2$, $q_1=d+2$, and $q_2=2(d+2)/d$, we have
\begin{multline*}
 \int W(t,x)(|D_x|^s\rho_{2,2}(V))(t,x)\,dx\,dt = \tr_{L^2_x}\int_0^\ii dt\,\int_0^t dt_1\,\int_0^{t_1}dt_2\,\times\\
 \times e^{-it\Delta_x}\left(|D_x|^sW(t)+\sum_{|\beta|\le s}\frac{1}{\beta!}\partial^\beta_x W(t) D^{s,\beta}_x+A^{(2)}(h^{(2)}(W(t)))^* A^{(1)}(h^{(1)}(W(t)))\right)e^{it\Delta_x}\times\\
 \times e^{-it_1\Delta} V(t_1)e^{i(t_1-t_2)\Delta}V(t_2)g(-i\nabla_x)e^{it_2\Delta}.
\end{multline*}
Let us first treat the term
 \begin{multline*}
 \tr_{L^2_x}\int_0^\ii dt\,\int_0^t dt_1\,\int_0^{t_1}dt_2\,e^{-it\Delta_x}A^{(2)}(h^{(2)}(W(t)))^* A^{(1)}(h^{(1)}(W(t)))e^{it\Delta_x}\times\\
 \times e^{-it_1\Delta} V(t_1)e^{i(t_1-t_2)\Delta}V(t_2)e^{it_2\Delta}g(-i\nabla_x)\\
 =\tr_{L^2_x}\int_0^\ii dt\,\int_0^t dt_1\,\int_0^{t_1}dt_2\,\sd^{-s}e^{-it\Delta_x}\sd^{-s}A^{(2)}(h^{(2)}(W(t)))^* \times\\
 \times A^{(1)}(h^{(1)}(W(t)))e^{it\Delta_x}\sd^{-s} e^{-it_1\Delta} \sd^s V(t_1)\sd^{-s}e^{it_1\Delta}\times\\
 \times e^{-it_2\Delta}\sd^{s}V(t_2)\sd^{-s} e^{it_2\Delta}\sd^{-s}\sd^{2s}g(-i\nabla_x)\sd^{2s}.
 \end{multline*}
 We expand $\sd^s V(t_1)\sd^{-s}$ according to Lemma \ref{lem:leibniz-op-3} and $\sd^{s}V(t_2)\sd^{-s}$ according to Lemma \ref{lem:leibniz-op-2}. We obtain
 \begin{multline*}
 \tr_{L^2_x}\int_0^\ii dt\,\int_0^t dt_1\,\int_0^{t_1}dt_2\,e^{-it\Delta_x}A^{(2)}(h^{(2)}(W(t)))^* A^{(1)}(h^{(1)}(W(t)))e^{it\Delta_x}\times\\
 \times e^{-it_1\Delta} V(t_1)e^{i(t_1-t_2)\Delta}V(t_2)e^{it_2\Delta}g(-i\nabla_x)\\
 =\tr_{L^2_x}\int_0^\ii dt\,\int_0^t dt_1\,\int_0^{t_1}dt_2\,\sd^{-s}e^{-it\Delta_x}\sd^{-s}A^{(2)}(h^{(2)}(W(t)))^* \times\\
 \times A^{(1)}(h^{(1)}(W(t)))e^{it\Delta_x}\sd^{-s} e^{-it_1\Delta} K^{(1)}(g^{(1)}(V(t_1)))^* K^{(2)}(g^{(2)}(V(t_1)))e^{it_1\Delta}\times\\
 \times e^{-it_2\Delta}L^{(1)}(g^{(1)}(V(t_2)))^* L^{(2)}(g^{(2)}(V(t_1)))e^{it_2\Delta}\sd^{-s}\sd^{2s}g(-i\nabla_x)\sd^{2s} .
 \end{multline*}
 Combining Lemma \ref{lem:leibniz-gamma} and Corollary \ref{coro:general-schatten} as in the proof of Proposition \ref{prop:rho21}, we have that
 $$
 \|A^{(1)}(h(t))e^{it\Delta_x}\|_{\gS^{2(d+2)}(L^2_x\to L^2_{t}L^2(X))} \lesssim \|h\|_{L^{2(d+2)/d}_t B^{(1)}},
 $$
 $$
  \|A^{(2)}(h(t))\sd^{-s}e^{it\Delta_x}\sd^{-s}\|_{\gS^{2d(d+2)/(d^2-2)}(L^2_x\to L^2_{t}L^2(X))} \lesssim \|h\|_{L^{d+2}_t B^{(2)}}.
  $$
  As in the proof of Proposition \ref{prop:rho21}, we deduce the bound
  \begin{multline*}
 |\tr_{L^2_x}\int_0^\ii dt\,\int_0^t dt_1\,\int_0^{t_1}dt_2\,e^{-it\Delta_x}A^{(2)}(h^{(2)}(W(t)))^* A^{(1)}(h^{(1)}(W(t)))e^{it\Delta_x}\times\\
 \times e^{-it_1\Delta} V(t_1)e^{i(t_1-t_2)\Delta}V(t_2)e^{it_2\Delta}g(-i\nabla_x)|\lesssim \|W\|_{L^2_{t,x}}\|V\|_{L^2_t H^s_x}^2.
 \end{multline*}
 The same decomposition can be applied to the term $|D_x|^s W(t)$ by splitting it as
 $$|D_x|^s W(t)=|D_x|^s\sd^s \sd^{-s}W(t)^{\tfrac{d}{d+2}}|W(t)|^{\tfrac{2}{d+2}},$$
 so we obtain the same estimate for this term. It remains to treat the terms involving $\partial^\beta_x W(t) D^{s,\beta}_x$ for $|\beta|\le s$. We first write these terms as
 \begin{multline*}
 \tr_{L^2_x}\int_0^\ii dt\,\int_0^t dt_1\,\int_0^{t_1}dt_2\,e^{-it\Delta_x}\partial^\beta_x W(t) D^{s,\beta}_xe^{it\Delta_x}\times\\
 \times e^{-it_1\Delta} V(t_1)e^{i(t_1-t_2)\Delta}V(t_2)e^{it_2\Delta}g(-i\nabla_x)\\
 =\tr_{L^2_x}\int_0^\ii dt\,\int_0^t dt_1\,\int_0^{t_1}dt_2\,\sd^{-s}e^{-it\Delta_x} W(t)  e^{it\Delta_x}D^{s,\beta}_x\sd^{-s} \times\\
 \times  e^{-it_1\Delta} \sd^s V(t_1)\sd^{-s}e^{it_1\Delta}\times\\
 \times e^{-it_2\Delta}\sd^{s}V(t_2)\sd^{-s} e^{it_2\Delta}\sd^{-s}\sd^{2s}g(-i\nabla_x)\sd^{s}\partial_x^\beta.
 \end{multline*}
 We then split $\sd^s V(t_1)\sd^{-s}$ according to Lemma \ref{lem:leibniz-op-4} and $\sd^s  V(t_2)\sd^{-s}$ according to Lemma \ref{lem:leibniz-op-2}. We get
 \begin{multline*}
 \tr_{L^2_x}\int_0^\ii dt\,\int_0^t dt_1\,\int_0^{t_1}dt_2\,e^{-it\Delta_x}\partial^\beta_x W(t) D^{s,\beta}_xe^{it\Delta_x}\times\\
 \times e^{-it_1\Delta} V(t_1)e^{i(t_1-t_2)\Delta}V(t_2)e^{it_2\Delta}g(-i\nabla_x)\\
 =\tr_{L^2_x}\int_0^\ii dt\,\int_0^t dt_1\,\int_0^{t_1}dt_2\,\sd^{-s}e^{-it\Delta_x}|W(t)|^{\tfrac{d}{d+2}}W(t)^{\tfrac{2}{d+2}}e^{it\Delta_x}\times\\
 \times  \sd^{-s}D^{s,\beta}e^{-it_1\Delta} M^{(1)}(g^{(1)}(V(t_1)))^* M^{(2)}(g^{(2)}(V(t_1)))e^{it_1\Delta}\times\\
 \times e^{-it_2\Delta}L^{(1)}(g^{(1)}(V(t_2)))^* L^{(2)}(g^{(2)}(V(t_1)))e^{it_2\Delta}\sd^{-s}\sd^{2s}g(-i\nabla_x)\sd^{s}\partial_x^\beta .
 \end{multline*}
 For $\beta\neq0$, we have
 \begin{align*}
    \|M_n^{(1)}(g(t))e^{it\Delta_x}\sd^{-s}D^{s,\beta}\|_{\gS^{2(\alpha_n^{(1)})'}(L^2_x\to L^2_{t,x})} &\lesssim \|M_n^{(1)}(g(t))e^{it\Delta_x}\sd^{-1/2}\|_{\gS^{2(\alpha_n^{(1)})'}(L^2_x\to L^2_{t,x})} \\
    &\lesssim \|g\|_{L^{2(p_n^{(1)})'}_t B_n^{(1)}},
 \end{align*}
 $$
  \|M_n^{(2)}(g(t))e^{it\Delta_x}\|_{\gS^{2(\alpha_n^{(2)})'}(L^2_x\to L^2_{t,x})}\lesssim \|g\|_{L^{2(p_n^{(2)})'}_t B_n^{(2)}}.
 $$
 Using that
 $$\frac{d-1}{2d}+\frac{1}{2(\alpha_n^{(1)})'}+\frac{1}{2(\alpha_n^{(2)})'}+\frac{d-1}{2d}\ge \frac{d-1}{2d} + \frac1d+\frac{d-1}{2d}=1,$$
 we can use the same strategy as before to estimate the terms with $\beta\neq0$. It remains to estimate the term with $\beta=0$, that is
 \begin{multline*}
  \tr_{L^2_x}\int_0^\ii dt\,\int_0^t dt_1\,\int_0^{t_1}dt_2\,e^{-it\Delta_x}W(t)e^{it\Delta_x}\times\\
 \times  \sd^{-s}|D_x|^{s}e^{-it_1\Delta} M^{(1)}(g^{(1)}(V(t_1)))^* M^{(2)}(g^{(2)}(V(t_1)))e^{it_1\Delta}\times\\
 \times e^{-it_2\Delta}L^{(1)}(g^{(1)}(V(t_2)))^* L^{(2)}(g^{(2)}(V(t_1)))e^{it_2\Delta}\sd^{-s}\sd^{2s}g(-i\nabla_x)\\
 =\tr_{L^2_x}\int_0^\ii dt\,\int_0^t dt_1\,\int_0^{t_1}dt_2\,e^{-it\Delta_x}W(t)|D_x|^{1/2}e^{it\Delta_x}\sd^{-(s-1/2)}|D_x|^{s-1/2}\times\\
 \times\sd^{-1/2}e^{-it_1\Delta} M^{(1)}(g^{(1)}(V(t_1)))^* M^{(2)}(g^{(2)}(V(t_1)))e^{it_1\Delta}\times\\
 \times e^{-it_2\Delta}L^{(1)}(g^{(1)}(V(t_2)))^* L^{(2)}(g^{(2)}(V(t_1)))e^{it_2\Delta}\sd^{-s}\sd^{2s}g(-i\nabla_x)
 \end{multline*}
 We decompose $W(t)|D_x|^{1/2}$ using the adjoint of Lemma \ref{lem:leibniz-op-general} applied to $\sigma=1/2$, $\sigma_1=0$, $\sigma_2=1/2$, and $q'=2(d+2)/d$, $q_1=2$:
 $$W(t)|D_x|^{1/2}=(|D_x|^{1/2}W)(t)+|D_x|^{1/2}W(t)+L^{(2)}(g^{(2)}(W(t)))^* L^{(1)}(g^{(1)}(W(t))).$$
 The last two terms can be treated as before, adding a $\sd^{-s}$ on their left which is absorbed by $g(-i\nabla_x)$. The only term remaining is the one involving $(|D_x|^{1/2}W)(t)$, for which we require a specific tool since $W$ is merely in $L^2_{t,x}$ and hence $|D_x|^{1/2}W\in L^2_t \dot{H}^{-1/2}_x$. We write it as $\tr_{L^2_x} AB$ with
  $$
  \begin{multlined}
  A:=\sd^{-\tfrac{d-1}{2}-\epsilon}\int_0^\ii dt\, e^{-it\Delta_x}(|D_x|^{1/2}W(t)) e^{it\Delta_x}\sd^{-(s-1/2)}|D_x|^{s-1/2}\times\\
  \times\int_0^t\,dt_1  (M_1^{(1)}(g_1^{(1)}(V(t_1))e^{it_1\Delta}\sd^{-1/2})^*:L^2_{t_1}L^2(X_1)\to L^2_x,
  \end{multlined}
  $$
  $$
  \begin{multlined}
  B:=M_1^{(2)}(g_1^{(2)}(V(t_1)))e^{it_1\Delta}\int_0^{t_1}\,dt_2 e^{-it_2\Delta}L^{(1)}(g^{(1)}(V(t_2)))^*\times\\
  \times L^{(2)}(g^{(2)}(V(t_1)))e^{it_2\Delta}\sd^{-s}\sd^{2s}g(-i\nabla_x)\sd^{\tfrac{d-1}{2}+\epsilon}:L^2_x\to L^2_{t_1}L^2(X_1)
  \end{multlined}
  $$
  for some $\epsilon>0$. By the same strategy as above, we have $B\in\gS^{10d/(5d-2)}(L^2_x\to L^2_{t_1}L^2(X_1))$. Hence, it remains to show that $A\in\gS^{10d/(5d+2)}(L^2_{t_1}L^2(X_1)\to L^2_x)$. To prove it, we invoke the partially orthogonal Christ-Kiselev result of Corollary \ref{coro:CK-2}.  To apply it, we recall that we have
  $$\|M_1^{(1)}(g(t))e^{it\Delta_x}\sd^{-1/2}\|_{\gS^{10d/7}(L^2_x\to L^2_tL^2(X_1))}\lesssim \|g\|_{L^{10/3}_t B_1^{(1)}}
  $$
  We next invoke Theorem \ref{thm:strichartz-HS} to infer that
  $$\|\sd^{-\tfrac{d-1}{2}-\epsilon}\int_0^\ii\,dt e^{-it\Delta_x}(|D|^{1/2}W(t)) e^{it\Delta_x}\|_{\gS^2(L^2_x)}\lesssim \|W\|_{L^2_{t,x}}.$$
  We can then apply Corollary \ref{coro:CK-2} to $p_1=2$, $p_2=10/3$, $\alpha_1=2$, $\alpha_2=10d/7$, $\alpha=10d/(5d+2)$ so that we indeed have $1/p_1+1/p_2=1/2+3/10>\max(1/\alpha,1/2)=1/2+2/(10d)$ and $1/\alpha=1/2+2/(10d)\le 1/\alpha_1+1/\alpha_2=1/2+7/(10d)$, so that the operator belongs to
  $\gS^{10d/(5d+2)}(L^2_{t_1,x}\to L^2_x)$. This concludes the proof of Proposition \ref{prop:rho22}.
\end{proof}

\section{Proof of Theorem \ref{thm:main1}}\label{sec:proof-thm}

The proof is only done for positive times, the one for negative times being done in the same way. By Theorem \ref{thm:stri-potential} applied to $s=d/2-1$ and $p=q=\mu=\nu=2$, Theorem \ref{thm:reaction}, and Proposition \ref{prop:penrose-inverse}, the map $\Phi$ defined by \eqref{eq:Phi} is a contraction on $B(0,\delta_0)\subset L^2_t(\R_+,H^s_x)$ if $\|Q_{\text{in}}\|_{\cH^{s,2d/(d+1)}}\le\epsilon_0$ for $\delta_0=\delta_0(\delta)>0$ and $\epsilon_0=\epsilon_0(\delta)>0$ small enough so that it admits a unique fixed point $\varrho_0$ on it. Defining $\gamma(t)=U_{w*\varrho_0}(t)(g(-i\nabla)+Q_{\text{in}})U_{w*\varrho_0}(t)^*$, $\gamma(t)$ is a solution to the Hartree equation $i\partial_t\gamma=[-\Delta_x+w*\varrho_0(t),\gamma(t)]$ and by the definition of $\Phi$ we have $\rho_{Q}=(1+\cL)\Phi(\varrho_0)-\cL[\varrho_0]=\varrho_0$. Let us now show that $Q(t):=\gamma(t)-g(-i\nabla)\in C^0_t \cH^{s,2}$. To do so, let us write
$$\sd^s Q(t)\sd^{s}=Q_1(t)+Q_2(t),$$
with
$$Q_1(t):=\sd^s U_V(t)Q_{\rm in}U_V(t)^*\sd^s,$$
\begin{align*}
Q_2(t) &:= \sd^s U_V(t)g(-i\nabla_x)U_V(t)^*\sd^s-\sd^{2s}g(-i\nabla_x) \\
&= -i\sd^s\int_0^t U_V(t,t_1)[V(t_1),g(-i\nabla_x)]U_V(t_1,t)\,dt_1\sd^s,
\end{align*}
where $V:=w*\varrho_0$. To prove the second equality, we used that $i\partial_t Q=[-\Delta+V(t),g(-i\nabla_x)+Q]$ to infer that
$$Q(t)=U_V(t)Q_{\rm in}U_V(t)^*-i\int_0^t U_V(t,t_1)[V(t_1),g(-i\nabla_x)]U_V(t_1,t)\,dt_1$$
and identify $Q_2(t)$ in this way. By the Duhamel formula, we have for any $t>t'$
$$
\begin{multlined}
 \sd^s e^{-it\Delta_x}U_V(t)\sd^{-s}-\sd^s e^{-it'\Delta}U_V(t')\sd^{-s}\\
 =-i\int_{t'}^t e^{-it_1\Delta}\sd^s V(t_1)U_V(t_1)\sd^{-s}\,dt_1.
\end{multlined}
$$
We have seen in the proof of Proposition \ref{prop:UV-Sonae} that
$$\|\sd^s V(t_1)U_V(t_1)\sd^{-s}\|_{L^2_x\to L^{r'}_{t_1}L^{c'}_x}\lesssim \|V\|_{L^2_{t_1}H^s_x}$$
for some $r,c\in(2,+\ii)$ such that $2/r+d/c=d/2$. We deduce by Strichartz estimates that
\begin{equation}\label{eq:est-scat}
\|\sd^s e^{-it\Delta_x}U_V(t)\sd^{-s}-\sd^s e^{-it'\Delta}U_V(t')\sd^{-s}\|_{L^2_x\to L^2_x}\lesssim \|V\|_{L^2([t',t],H^s_x)},
\end{equation}
showing that $\sd^s e^{-it\Delta_x}U_V(t)\sd^{-s}\in C^0_t\cB(L^2_x\to L^2_x)$ and hence $\sd^s U_V(t)\sd^{-s}\in C^0_t\cB(L^2_x\to L^2_x)$.
Since $\sd^s Q_{\rm in}\sd^s\in\gS^{2d/(d+1)}\hookrightarrow\gS^2$, we deduce that $Q_1\in C^0_t\gS^2$. Next, we have
\begin{align*}
 Q_2(t) &= -i\sd^s\int_0^t U_V(t,t_1)[V(t_1),g(-i\nabla_x)]U_V(t_1,t)\,dt_1\sd^s\\
 &=-i\int_0^t\sd^s U_V(t)\sd^{-s}\sd^s U_V(t_1)^*\sd^{-s}\sd^s [V(t_1),g(-i\nabla_x)]\sd^s\times\\
 &\times\sd^{-s}U_V(t_1)\sd^s \sd^{-s}U_V(t)^*\sd^s\,dt_1
\end{align*}
Since $\sd^sU_V(t)\sd^{-s},\sd^sU_V(t)^*\sd^{-s}\in C^0_t\cB(L^2_x\to L^2_x)$, and by the inequality
\begin{equation}\label{eq:KSS-Hs}
 \|\sd^s W(x)g(-i\nabla_x)\|_{\gS^2}\lesssim \|W\|_{H^s_x}\|\langle\xi\rangle^sg(\xi)\|_{L^2_\xi},
\end{equation}
which can be proved first for $s\in\N$ using the Kato-Seiler-Simon inequality and then extended to any $s\ge0$ by interpolation, we deduce that $Q_2\in C^0_t\gS^2$. We now turn to the scattering property. To prove it, it is enough to show that
$$\|\sd^s e^{-it\Delta_x}Q(t)e^{it\Delta_x}\sd^s - \sd^s e^{-it'\Delta}Q(t')e^{it'\Delta}\sd^s\|_{\gS^{2d/(d-1)}}\to0$$
as $t,t'\to+\ii$. As above, we split $\sd^s Q(t)\sd^{s}=Q_1(t)+Q_2(t)$. For the part involving $Q_1$, we again use \eqref{eq:est-scat} and the fact that $Q_{\rm in}\in\gS^{2d/(d+1)}\hookrightarrow\gS^{2d/(d-1)}$ to deduce that
$$\|e^{-it\Delta_x}Q_1(t)e^{it\Delta_x}-e^{-it'\Delta}Q_1(t')e^{it'\Delta}\|_{\gS^{2d/(d-1)}}\to0$$
as $t,t'\to+\ii$. To deal with $Q_2(t)$, we need to expand once more,
$$
\begin{multlined}
 \sd^se^{-it\Delta_x}U_V(t)\sd^{-s} - \sd^se^{-it'\Delta}U_V(t')\sd^{-s} = -i\int_{t'}^t e^{-it_1\Delta}\sd^s V(t_1)\sd^{-s}e^{it_1\Delta}\,dt_1\\
 -\int_{t'}^t\,dt_1\int_0^{t_1}\,dt_2 e^{-it_1\Delta}\sd^s V(t_1)U_V(t_1,t_2)V(t_2)e^{it_2\Delta}\sd^{-s},
\end{multlined}
$$
Combining Lemma \ref{lem:leibniz-op-2} and Corollary \ref{coro:general-schatten}, we deduce that
$$\left\| \int_{t'}^t e^{-it_1\Delta}\sd^s V(t_1)\sd^{-s}e^{it_1\Delta}\,dt_1 \sd^{-s}\right\|_{\gS^{2d/(d-1)}(L^2_x\to L^2_x)}\lesssim \|V\|_{L^2([t',t],H^s_x)}.$$
We factorize the second integral as
$$
\begin{multlined}
 \int_{t'}^t\,dt_1\int_0^{t_1}\,dt_2 e^{-it_1\Delta}\sd^s V(t_1)U_V(t_1,t_2)V(t_2)e^{it_2\Delta}\sd^{-s}\\
 = \int_{t'}^t\,dt_1e^{-it_1\Delta}\sd^s V(t_1)\sd^{-s}\sd^s U_V(t_1)\sd^{-s}\times\\\times\int_0^{t_1}\,dt_2 \sd^s U_V(t_2)^*\sd^{-s}\sd^s V(t_2)\sd^{-s}e^{it_2\Delta}\sd^{-s}\sd^s.
\end{multlined}
$$
Splitting $\sd^s V(t_1)\sd^{-s}$ according to Lemma \ref{lem:leibniz-without-loss} and $\sd^s V(t_2)\sd^{-s}$ according to Lemma \ref{lem:leibniz-op-2}, using Corollary \ref{coro:CK} and Lemma \ref{lem:strichartz-A}, we deduce that
$$\left\|
\int_{t'}^t\,dt_1\int_0^{t_1}\,dt_2 e^{-it_1\Delta}\sd^s V(t_1)U_V(t_1,t_2)V(t_2)e^{it_2\Delta}\sd^{-2s}
\right\|_{\gS^{2d/(d-1)}(L^2_x\to L^2_x)}\lesssim \|V\|_{L^2([t',t],H^s_x)}.$$
We conclude that
$$\|\sd^se^{-it\Delta_x}U_V(t)g(-i\nabla_x) - \sd^se^{-it'\Delta}U_V(t')g(-i\nabla_x)\|_{\gS^{2d/(d-1)}}\lesssim \|V\|_{L^2([t',t],H^s_x)},$$
and hence
$$\|e^{-it\Delta_x}Q_2(t)e^{it\Delta_x}-e^{-it'\Delta}Q_2(t')e^{it'\Delta}\|_{\gS^{2d/(d-1)}}\to0$$
as $t,t'\to+\ii$. The last step of the proof is to show uniqueness of solutions, which we prove in a separate statement below because it requires slightly different tools.

\begin{proposition}\label{prop:uniqueness}
 Let $d\ge2$ and $s=d/2-1$. Assume that $(w,g)$ satisfies the assumptions of Theorem \ref{thm:main1}, with furthermore $\langle\xi\rangle^{d/2+\eta}g\in L^2_\xi$ for some $\eta>0$.  Then, there exists $\epsilon_0>0$ such that for any $T\in[0,+\ii]$ and any $Q_{\rm in}\in\cH^{s,2d/(d+1)}$ with $\|Q_{\rm in}\|_{\cH^{s,2d/(d+1)}}\le\epsilon_0$ the following holds: let $Q,\tilde{Q}\in C^0_t([0,T),\gS^2)$ be two solutions to the Hartree equation \eqref{eq:hartree-Q} such that $\rho_Q,\rho_{\tilde{Q}}\in L^2_{t,{\rm loc}}([0,T), H^s_x)$. Then, for all $t\in[0,T)$ we have $Q(t)=\tilde{Q}(t)$.
\end{proposition}

\begin{remark}
 For any $\delta_0>0$, there exists $T_0\in(0,T)$ such that we have $\|\rho_Q\|_{L^2_t([0,T_0],H^s_x)}\le\delta_0$ and $\|\rho_{\tilde{Q}}\|_{L^2_t([0,T_0],H^s_x)}\le\delta_0$. If $\delta_0$ is small enough, one can show that the same map $\Phi$ as in the beginning of the proof of Theorem \ref{thm:main1} is a contraction on $B(0,\delta_0)\subset L^2([0,T_0],H^s_x)$. Hence, we deduce that $\rho_Q=\rho_{\tilde{Q}}$ and hence $Q=\tilde{Q}$ on $[0,T_0]$. The problem with this argument is that it cannot be iterated, because $\Phi$ is a contraction under the condition that $\|Q_{\rm in}\|_{\cH^{s,2d/(d+1)}}\le\epsilon_0$, a condition that we don't know how to propagate in time. Hence, we will use weaker norms to obtain uniqueness.
\end{remark}

\begin{lemma}\label{lem:est-uniqueness-L2}
 Let $d\ge2$, $s=d/2-1$, $T>0$ and $V_1,V_2\in L^2_t([0,T],H^s_x)$. Then, we have
 $$\|h(t)(U_{V_1}(t)-U_{V_2}(t))\sd^{-s}\|_{\gS^{\frac{2d(d+2)}{d^2+2d-2}}(L^2_x\to L^2_t([0,T],L^2_x))}\lesssim \|h\|_{L^{d+2}_{t,x}([0,T]\times\R^d)}\|V_1-V_2\|_{L^{2}_{t,x}([0,T]\times\R^d)},$$
 $$\|h(t)(U_{V_1}(t)-U_{V_2}(t))\sd^{-s}\|_{\gS^{\frac{2(d+2)}{d+1}}(L^2_x\to L^2_t([0,T],L^2_x))}\lesssim \|h\|_{L^{\frac{2(d+2)}{d}}_{t,x}([0,T]\times\R^d)}\|V_1-V_2\|_{L^{2}_{t}([0,T],H^s(\R^d))}.$$
\end{lemma}

\begin{proof}
 We have
 $$h(t)(U_{V_1}(t)-U_{V_2}(t))\sd^{-s}
 =-ih(t)U_{V_1}(t)\int_0^t U_{V_1}(t_1)^*(V_1(t_1)-V_2(t_1))U_{V_2}(t_1)\sd^{-s}\,dt_1.$$
 Using Proposition \ref{prop:UV-bez}, if $h\in L^{d+2}_{t,x}$ we have  $h(t)U_{V_1}(t)\in\gS^{2(d+2)}(L^2_x\to L^2_{t,x})$ and
 $$\left\|\int_0^\ii U_{V_1}(t_1)^*(V_1(t_1)-V_2(t_1))U_{V_2}(t_1)\sd^{-s}\,dt_1\right\|_{\gS^{\frac{2d}{d-1}}(L^2_x)}\lesssim \|V_1-V_2\|_{L^2_{t,x}}.$$
 The first estimate then follows from Corollary \ref{coro:CK}. For the second estimate, we distribute the derivatives differently:
 $$
 \begin{multlined}
 h(t)(U_{V_1}(t)-U_{V_2}(t))\sd^{-s}
 =-ih(t)U_{V_1}(t)\sd^{-s}\times\\
 \times\int_0^t \sd^s U_{V_1}(t_1)^*\sd^{-s}\sd^s(V_1(t_1)-V_2(t_1))\sd^{-s}\sd^s U_{V_2}(t_1)\sd^{-s}\,dt_1.
 \end{multlined}
 $$
 By Proposition \ref{prop:UV-bez}, we have  that $h(t)U_{V_1}(t)\sd^{-s}\in\gS^{2d(d+2)/(d^2-2)}(L^2_x\to L^2_{t,x})$ if $h\in L^{2(d+2)/d}_{t,x}$. Decomposing $\sd^s(V_1(t_1)-V_2(t_1))\sd^{-s}$ according to Lemma \ref{lem:leibniz-without-loss} (applied to $s_0=s$, $\mu=\nu_0=2$), and using again Proposition \ref{prop:UV-bez}, the integral term leads to a Schatten exponent $\gS^{2d}(L^2_x)$. By Corollary \ref{coro:CK}, we deduce the second estimate.
\end{proof}

\begin{proof}[Proof of Proposition \ref{prop:uniqueness}]
Define
$$T^*:=\sup\{s\in[0,T),\ \|\rho_Q-\rho_{\tilde{Q}}\|_{L^2([0,s],L^2_x)}=0\},$$
and assume that $T^*<T$. Since $\rho_Q=\rho_{\tilde{Q}}$ on $[0,T^*]$, we deduce that $Q=\tilde{Q}$ in $[0,T^*]$. Let $\epsilon>0$ be such that $T^*+\epsilon<T$, and let us show that for $\epsilon>0$ small enough we have $\|\rho_Q-\rho_{\tilde{Q}}\|_{L^2([T^*,T^*+\epsilon],L^2_x)}=0$, reaching a contradiction and proving that $T^*=T$. To prove it, let us define $Q_*:=Q(T^*)=\tilde{Q}(T^*)$, $V:=w*\rho_Q$ and $\tilde{V}:=w*\rho_{\tilde{Q}}$. First, we choose $\epsilon>0$ small enough such that $\|V\|_{L^2_t([T^*,T^*+\epsilon],H^s_x)},\|\tilde{V}\|_{L^2_t([T^*,T^*+\epsilon],H^s_x)}\le1$ and hence all constants that depend implicitly on $\|V\|_{L^2_t([T^*,T^*+\epsilon],H^s_x)}$ and $\|\tilde{V}\|_{L^2_t([T^*,T^*+\epsilon],H^s_x)}$ are uniformly bounded. Choosing $\epsilon_0>0$ small enough, we can also assume that $Q$ is the solution we constructed in Theorem \ref{thm:main1} so that, choosing $\delta_0\le1$ in its proof, we have $\|V\|_{L^2_t([0,T),H^s_x)}\le1$ and hence all implicit constants involving $V$ are uniformly bounded. For any $t\in[T^*,T^*+\epsilon]$, we split
$$\rho_{Q(t)}-\rho_{\tilde{Q}(t)} = I_1(t)+I_2(t)+I_3(t)+I_4(t),$$
where
$$I_1(t):=\rho_{U_V(t,T^*)Q_*U_V(t,T^*)^*}-\rho_{U_{\tilde{V}}(t,T^*)Q_*U_{\tilde{V}}(t,T^*)^*},$$
$$I_2(t):=\rho\left[-i\int_{T^*}^t (U_V(t,t_1)-U_{\tilde{V}}(t,t_1))[V(t_1),g(-i\nabla_x)]U_V(t_1,t)\,dt_1\right],$$
$$I_3(t):=\rho\left[-i\int_{T^*}^t U_{\tilde{V}}(t,t_1)[(V-\tilde{V})(t_1),g(-i\nabla_x)]U_V(t_1,t)\,dt_1\right],$$
$$I_4(t):=\rho\left[-i\int_{T^*}^t U_{\tilde{V}}(t,t_1)[\tilde{V}(t_1),g(-i\nabla_x)](U_V(t_1,t)-U_{\tilde{V}}(t_1,t))\,dt_1\right].$$
To estimate $I_1(t)$, we further decompose it as
$$I_1(t)=\rho_{(U_V(t,T^*)-U_{\tilde{V}}(t,T^*))Q_*U_V(t,T^*)^*}+\rho_{U_{\tilde{V}}(t,T^*)Q_*(U_V(t,T^*)-U_{\tilde{V}}(t,T^*))^*}.$$
Let $W\in L^2_{t,x}$. We have
$$
\begin{multlined}
 \int W(t,x) \rho_{(U_V(t,T^*)-U_{\tilde{V}}(t,T^*))Q_*U_V(t,T^*)^*}(x)\,dx\,dt
 \\
 = \int \tr_{L^2_x} \sd^{-s}U_V(t,T^*)^*W(t,x)(U_V(t,T^*)-U_{\tilde{V}}(t,T^*))\sd^{-s}\sd^s Q_*\sd^s\,dt
\end{multlined}
$$
By Proposition \ref{prop:UV-bez}, we have
$$\| |W|^{\frac{d}{d+2}}U_V(t,T^*)\sd^{-s}\|_{\gS^{\frac{2d(d+2)}{d^2-2}}(L^2_x\to L^2_{t,x})} \lesssim \|W\|_{L^2_{t,x}}^{\frac{d}{d+2}},$$
while by Lemma \ref{lem:est-uniqueness-L2} we have
$$\|W^{\frac{2}{d+2}}(U_V(t,T^*)-U_{\tilde{V}}(t,T^*))\sd^{-s}\|_{\gS^{\frac{2d(d+2)}{d^2+2d-2}}(L^2_x\to L^2_{t,x})}\lesssim \|W\|_{L^2_{t,x}}^{\frac{2}{d+2}}\|V-\tilde{V}\|_{L^2_{t,x}}.$$
By the H\"older inequality in Schatten spaces, we deduce that
$$\|I_1(t)\|_{L^2_{t,x}}\le C\|Q_*\|_{\cH^{s,d}}\|V-\tilde{V}\|_{L^2_{t,x}}\le C\|Q_*\|_{\cH^{s,d}}\|\rho_Q-\rho_{\tilde{Q}}\|_{L^2_{t,x}}.$$
We estimate $I_2(t)$ by
$$\|I_2(t)\|_{L^2_{t,x}}\le\int_{T^*}^{T^*+\epsilon}\|\rho_{(U_V(t,t_1)-U_{\tilde{V}}(t,t_1))[V(t_1),g(-i\nabla_x)]U_V(t_1,t)}\|_{L^2_{t,x}}\,dt_1,$$
and by the same estimate that we used to treat the term $I_1(t)$ we deduce that
$$\|I_2(t)\|_{L^2_{t,x}}\le C\|\rho_Q-\rho_{\tilde{Q}}\|_{L^2_{t,x}}\int_{T^*}^{T^*+\epsilon}\|\sd^s[V(t_1),g(-i\nabla_x)]\sd^s\|_{\gS^d}\,dt_1.$$
Using \eqref{eq:KSS-Hs} above we deduce that
$$\|\sd^s[V(t_1),g(-i\nabla_x)]\sd^s\|_{\gS^d}\le \|\sd^s[V(t_1),g(-i\nabla_x)]\sd^s\|_{\gS^2}\lesssim \|V(t_1)\|_{H^s_x},$$
hence
$$\|I_2(t)\|_{L^2_{t,x}}\le C\sqrt{\epsilon}\|\rho_Q-\rho_{\tilde{Q}}\|_{L^2_{t,x}}.$$
The term $I_4(t)$ has the same form as $I_2(t)$ and hence can be estimated in the same way. To estimate $I_3(t)$, we only explain how to treat the term where the commutator $[(V-\tilde{V})(t_1),g(-i\nabla_x)]$ is replaced by $(V-\tilde{V})(t_1)g(-i\nabla_x)$, the other term in the commutator being treated in the same way. Using the same technique that we used to estimate $I_1(t)$ but using the second estimate of Lemma \ref{lem:est-uniqueness-L2}, we deduce that
$$\|\rho_{U_{\tilde{V}}(t,t_1)\Gamma(U_{V_1}(t_1,t)-U_{V_2}(t_1,t)}\|_{L^2_{t,x}}\lesssim \|\Gamma\sd^s\|_{\gS^2}.$$
Applying this estimate to $V_1=V$ and $V_2=0$ and using the same technique that we used to estimate $I_2(t)$, we deduce that
$$
\begin{multlined}
\left\|\rho\left[-i\int_{T^*}^t U_{\tilde{V}}(t,t_1)(V-\tilde{V})(t_1)g(-i\nabla_x)(U_V(t_1,t)-e^{i(t_1-t)\Delta})\,dt_1\right]\right\|_{L^2_{t,x}}\\
\le C\int_{T^*}^{T^*+\epsilon}\|(V-\tilde{V})(t_1)g(-i\nabla_x)\sd^{s}\|_{\gS^2}\,dt_1\le C\sqrt{\epsilon}\|\rho_Q-\rho_{\tilde{Q}}\|_{L^2_{t,x}}.
\end{multlined}
$$
We estimate the remaining term using that for any $\eta>0$ and by the Kato-Seiler-Simon inequality,
$$\|\rho_\Gamma\|_{L^2_x}\lesssim\|\Gamma\sd^{d/2+\eta}\|_{\gS^2},$$
so that
$$
\begin{multlined}
\left\|\rho\left[-i\int_{T^*}^t U_{\tilde{V}}(t,t_1)(V-\tilde{V})(t_1)g(-i\nabla_x)e^{i(t_1-t)\Delta}\,dt_1\right]\right\|_{L^2_{t,x}}\\
\le C\int_{T^*}^{T^*+\epsilon}\|(V-\tilde{V})(t_1)g(-i\nabla_x)\sd^{d/2+\eta}\|_{\gS^2}\,dt_1\le C\sqrt{\epsilon}\|\rho_Q-\rho_{\tilde{Q}}\|_{L^2_{t,x}},
\end{multlined}
$$
and thus $\|I_3(t)\|_{L^2_{t,x}}\le C\sqrt{\epsilon}\|\rho_Q-\rho_{\tilde{Q}}\|_{L^2_{t,x}}$. Hence, we have proved that
$$\|\rho_Q-\rho_{\tilde{Q}}\|_{L^2_{t,x}}\le C(\sqrt{\epsilon}+\|Q_*\|_{\cH^{s,d}})\|\rho_Q-\rho_{\tilde{Q}}\|_{L^2_{t,x}},$$
for some $C>0$ which is independent of $Q$ and $\tilde{Q}$. It remains to show that $Q_{\rm in}$ can be chosen small enough so that $C\|Q_*\|_{\cH^{s,d}}<1/2$. If we have this property, this concludes the proof by choosing $\epsilon>0$ small enough so that $C\sqrt{\epsilon}<1/2$. To prove the smallness of $Q_*$, we use that $\sd^s U_V(t)\sd^{-s}$ is bounded on $L^2_x$ by Proposition \ref{prop:UV-Sonae} to infer that
$$\|Q_*\|_{\cH^{s,d}}\lesssim \|Q_{\rm in}\|_{\cH^{s,d}}+\left\|\int_0^{T^*}U_V(t_1)^* V(t_1)g(-i\nabla_x)U_V(t_1)\,dt_1\right\|_{\cH^{s,d}}.$$
We expand
$$U_V(t_1)=e^{it_1\Delta}-i\int_0^{t_1}e^{i(t_1-t_2)\Delta}V(t_2)U_V(t_2)\,dt_2,$$
and use Proposition \ref{prop:UV-bez} to infer that
$$
\begin{multlined}
  \left\|\int_0^{T^*}\sd^s U_V(t_1)^*\sd^{-s}\sd^s V(t_1)\sd^{-s}g(-i\nabla_x)\sd^{3s} e^{it_1\Delta}\,dt_1\sd^{-s}\right\|_{\gS^{2d/(d-1)}}\\
  \lesssim\|V\|_{L^2_tH^s_x}.
\end{multlined}
$$
Similarly, we can write
$$
\begin{multlined}
\sd^s \int_0^\ii U_{V}(t_1)^*V(t_1)g(-i\nabla_x)\int_0^{t_1}e^{i(t_1-t_2)\Delta}V(t_2)U_{V}(t_2)\,dt_2 \sd^s\\
=\int_0^{T^*}\sd^s U_V(t_1)^*\sd^{-s} \sd^s V(t_1)\sd^{-s}e^{it_1\Delta}\sd^{-s}\sd^{3s}g(-i\nabla_x)\times\\
\times\int_0^{t_1}e^{-it_2\Delta}\sd^{-s}V(t_2)\sd^s \sd^{-s}U_V(t_2)\sd^s\,dt_2,
\end{multlined}
$$
so that we also have
$$
\begin{multlined}
 \left\|\sd^s \int_0^\ii U_{V}(t_1)^*V(t_1)g(-i\nabla_x)\int_0^{t_1}e^{i(t_1-t_2)\Delta}V(t_2)U_{V}(t_2)\,dt_2 \sd^s\right\|_{\gS^{2d/(d-1)}}\\
  \lesssim\|V\|_{L^2_tH^s_x}.
\end{multlined}
$$
We deduce that
$$
\begin{multlined}
 \left\|\int_0^{T^*}U_V(t_1)^* V(t_1)g(-i\nabla_x)U_V(t_1)\,dt_1\right\|_{\cH^{s,d}}\\
 \le\left\|\sd^s\int_0^{T^*}U_V(t_1)^* V(t_1)g(-i\nabla_x)U_V(t_1)\,dt_1\sd^s\right\|_{\gS^{2d/(d-1)}}\lesssim\|V\|_{L^2_tH^s_x},
\end{multlined}
$$
and thus
$$\|Q_*\|_{\cH^{s,d}}\le C(\|Q_{\rm in}\|_{\cH^{s,d}}+\|V\|_{L^2_tH^s_x}),$$
for some absolute constant $C>0$. Hence, for $\epsilon_0$ small enough and for $\delta_0>0$ small enough in the proof of Theorem \ref{thm:main1}, we have $C\|Q_*\|_{\cH^{s,d}}<1/2$. This concludes the proof.
\end{proof}

 \appendix

  \section{An alternative proof of \cite[Thm. 3.1]{CheHonPav-17}}\label{app:CHP}

  \begin{proof}[Proof of Theorem \ref{thm:strichartz-HS}]

  We have
  $$I=\left\|\sd^{-\alpha_1}\int_\R dt\,e^{-it\Delta_x} V e^{it\Delta_x} \sd^{-\alpha_2}\right\|_{\gS^2}^2 = \int |\hat{V}(|p|^2-|q|^2,p-q)|^2\frac{dp\,dq}{\langle p\rangle^{2\alpha_1}\langle q\rangle^{2\alpha_2}}.$$
  Writing $c=p+q$ and $r=p-q$ we get
  $$I\lesssim \int |\hat{V}(c\cdot r,r)|^2\frac{dc\,dr}{\langle c+r\rangle^{2\alpha_1}\langle c-r\rangle^{2\alpha_2}}.$$
  Writing next $c=\rho\omega$ with $\rho\in(0,+\ii)$ and $\omega\in\Sph^{d-1}$, we get
  $$I\lesssim \int |\hat{V}(\rho \omega\cdot r,r)|^2\frac{\rho^{d-1}d\rho\,d\omega\,dr}{\langle \rho\omega+r\rangle^{2\alpha_1}\langle \rho\omega-r\rangle^{2\alpha_2}}.$$
  For $r\neq0$, we write $\omega=\cos\theta(r/|r|)+\sin\theta\omega'$ with $\omega'\cdot r=0$ and $\theta\in(0,\pi)$. We obtain
  $$I\lesssim \int |\hat{V}(\rho |r|\cos\theta,r)|^2\frac{\rho^{d-1}d\rho\,\sin^{d-2}\theta d\theta\,dr}{\langle \rho\cos\theta+|r|,\rho\sin\theta\rangle^{2\alpha_1}\langle \rho\cos\theta-|r|,\rho\sin\theta\rangle^{2\alpha_2}}.$$
  In the region $\theta\in(0,\pi/2)$ (we denote this integral $I_+$), we change variables  $\rho'=\rho|r|\cos\theta$, we obtain
  $$I_+\lesssim \int_{\R^3}\,dr\int_0^\ii\,d\rho'\,|\hat{V}(\rho',r)|^2\int_0^{\pi/2}\,d\theta \frac{1}{\cos^2\theta|r|^d}  \frac{(\rho')^{d-1}\tan^{d-2}\theta}{\langle \tfrac{\rho'}{|r|}+|r|,\tfrac{\rho'}{|r|}\tan\theta
  \rangle^{2\alpha_1}\langle \tfrac{\rho'}{|r|}-|r|,\tfrac{\rho'}{|r|}\tan\theta\rangle^{2\alpha_2}}.$$
  Defining next $s=\tan\theta$ we have to analyze the integral
  $$\frac{(\rho')^{d-1}}{|r|^d}\int_0^\ii \frac{s^{d-2}\,ds}{\langle \tfrac{\rho'}{|r|}+|r|,\tfrac{\rho'}{|r|}s
  \rangle^{2\alpha_1}\langle \tfrac{\rho'}{|r|}-|r|,\tfrac{\rho'}{|r|}s\rangle^{2\alpha_2}}=\frac{1}{|r|}\int_0^\ii \frac{s^{d-2}\,ds}{\langle \tfrac{\rho'}{|r|}+|r|,s
  \rangle^{2\alpha_1}\langle \tfrac{\rho'}{|r|}-|r|,s\rangle^{2\alpha_2}}.$$
  Notice that it converges only when $2\alpha_1+2\alpha_2>d-1$, that is if $\alpha_1+\alpha_2>(d-1)/2$. To bound $I_+\lesssim \| |\nabla|^{-1/2}\sd^{-\alpha_0}V\|_{L^2_{t,x}}^2$, we need to show that
  $$\langle r\rangle^{2\alpha_0}\int_0^\ii \frac{s^{d-2}\,ds}{\langle \tfrac{\rho'}{|r|}+|r|,s
  \rangle^{2\alpha_1}\langle \tfrac{\rho'}{|r|}-|r|,s\rangle^{2\alpha_2}}\lesssim1$$
  uniformly in $(\rho',r)$. For $|r|\le1$, this is obvious so we consider the case $|r|\ge1$. Given that $\rho'\ge0$ we can bound
  $$\int_0^\ii \frac{s^{d-2}\,ds}{\langle \tfrac{\rho'}{|r|}+|r|,s
  \rangle^{2\alpha_1}\langle \tfrac{\rho'}{|r|}-|r|,s\rangle^{2\alpha_2}}\lesssim
  \int_0^\ii \frac{s^{d-2}\,ds}{\langle |r|,s
  \rangle^{2\alpha_1}\langle s\rangle^{2\alpha_2}}.
  $$
  Now using that
  $$\frac{1}{\langle |r|,s
  \rangle^{2\alpha_1}\langle s\rangle^{2\alpha_2}}\lesssim\frac{1}{(1+|r|+s)^{2\alpha_1}(1+s)^{2\alpha_2}}\lesssim\frac{1}{(1+|r|)^{2\alpha_1+2\alpha_2}}\frac{1}{(1+\frac{s}{1+|r|})^{2\alpha_1}}\frac{1}{(\frac{1}{1+|r|}+\frac{s}{1+|r|})^{2\alpha_2}}
  $$
  we obtain
  $$\int_0^\ii \frac{s^{d-2}\,ds}{\langle |r|,s
  \rangle^{2\alpha_1}\langle s\rangle^{2\alpha_2}} \lesssim \frac{(1+|r|)^{d-1}}{(1+|r|)^{2\alpha_1+2\alpha_2}}\int_0^\ii\frac{s^{d-2}\,ds}{(1+s)^{2\alpha_1}(\frac{1}{1+|r|}+s)^{2\alpha_2}}.
  $$
  Now as $|r|\to+\ii$ we have
  $$
  \int_0^\ii\frac{s^{d-2}\,ds}{(1+s)^{2\alpha_1}(\frac{1}{1+|r|}+s)^{2\alpha_2}} \sim
  \begin{cases}
   1 & \text{if}\ 2\alpha_2<d-1 \\
   \log\frac{1}{1+|r|} & \text{if}\ 2\alpha_2=d-1 \\
   \frac{1}{|r|^{d-1-2\alpha_2}} & \text{if}\ 2\alpha_2>d-1.
  \end{cases}
  $$
  Doing the same reasoning in the region $\theta\in(\pi/2,\pi)$ (with the change of variables $\rho'= -\rho|r|\cos\theta$, which amounts to reverse the roles of $\alpha_1$ and $\alpha_2$), we obtain the result.

 \end{proof}

%
%

\end{document}